\documentclass[letterpaper,11pt,reqno]{amsart}
\usepackage[margin=1in]{geometry}

\usepackage[dvips]{color}
\usepackage{graphicx}
\usepackage{epsfig}
\usepackage{tikz}
\usepackage{tikz-cd}
\usepackage{enumerate}
\usepackage{color}
\usepackage{hyperref}
\hypersetup{
     colorlinks   = true,
     citecolor    = black,
     linkcolor    = blue,
     urlcolor     = blue
}

\usepackage{latexsym}
\usepackage{amssymb}

\usepackage{amsmath}
\usepackage{amsthm}
\usepackage{amscd}
\usepackage{latexsym}
\usepackage{amssymb}
\usepackage{amsmath}
\usepackage{amsthm}
\usepackage{amscd}
\usepackage{bbm}

\newtheorem{thm}{Theorem}[section]
\newtheorem{theoremAlph}{Theorem}

\newtheorem{prop}[thm]{Proposition}
\theoremstyle{definition}
\newtheorem{lem}[thm]{Lemma}

\newtheorem{definition}[thm]{Definition}
\newtheorem{cor}[thm]{Corollary}
\newtheorem{rem}[thm]{Remark}
\newtheorem{claim}[thm]{Claim}

\newtheorem{ques}[thm]{Question} 

\newtheorem{conjecture}[thm]{Conjecture}

\newtheorem{problem}[thm]{Problem}
\newtheorem{convention}[thm]{Convention}

\newtheorem*{acknowledgement}{Acknowledgements}

\newcommand\hol{\mathrm{hol}}
\newcommand{\dist}{{\rm dist}}

\renewcommand{\Re}{\mrm{Re}}
\renewcommand{\Im}{\mrm{Im}}

\newcommand{\N}{\mathbb{N}}
\newcommand{\Z}{\mathbb{Z}}
\newcommand{\Q}{\mathbb{Q}}
\newcommand{\R}{\mathbb{R}}
\newcommand{\C}{\mathbb{C}}

\newcommand{\mc}{\mathcal}

\newcommand{\mrm}{\mathrm}

\renewcommand{\a}{\alpha}
\renewcommand{\b}{\beta}
\newcommand{\g}{\gamma}
\newcommand{\G}{\Gamma}
\renewcommand{\d}{\delta}
\newcommand{\e}{\varepsilon}
\renewcommand{\l}{\lambda}
\renewcommand{\L}{\Lambda}
\newcommand{\w}{\omega}
\newcommand{\s}{\sigma}
\newcommand{\vp}{\varphi}
\renewcommand{\t}{\tau}
\renewcommand{\th}{\theta}

\renewcommand{\k}{\kappa}

\newcommand{\set}[1]{\left\{#1\right\}}
\renewcommand{\r}{\rightarrow}

\newcommand{\norm}[1]{\left\lVert#1\right\rVert}

\newcommand{\Lcal}{\mc{L}}
\newcommand{\Kcal}{\mc{K}}

\newcommand{\Pcal}{\mc{P}}
\newcommand{\Wcal}{\mc{W}}
\newcommand{\Ecal}{\mc{E}}
\newcommand{\Hcal}{\mc{H}}

\newcommand{\Dcal}{\mc{D}}

\newcommand{\Mcal}{\mc{M}}
\newcommand{\Ncal}{\mc{N}}

\newcommand{\Rcal}{\mc{R}}

\newcommand{\Vcal}{\mc{V}}
\newcommand{\Ucal}{\mc{U}}

\newcommand{\Ccal}{\mc{C}}
\newcommand{\Qcal}{\mc{Q}}

\newcommand{\supp}{\mrm{supp}}

\newcommand{\dAGY}{\dist_{\mrm{AGY}}}
\newcommand{\Trem}{\mrm{Trem}}

\newcommand{\hV}{\widehat{\Vcal}}
\newcommand{\SL}{\mathrm{SL}_2(\R)}

\newcommand{\RP}{\R\mathbb{P}^{d}}
\newcommand{\KZ}[2]{\mrm{KZ}(#1,#2)}

\newcommand{\twist}{\mrm{Twist}}
\newcommand{\cylspace}{\mrm{Cyl}}
\newcommand{\T}{\mathbb{T}}

\newcommand{\HHm}{{\mathcal{H}_{\mathrm{m}}}}
\newcommand{\HHu}{\mathcal{H}_{\mrm{u}}}
\renewcommand{\P}{\mathbb{P}}

\newcommand{\tHHu}{\widetilde{\Hcal}_\mrm{u}}
\newcommand{\babs}{\b_{\mrm{abs}}}

\numberwithin{equation}{section}

\AtBeginDocument{%
   \def\MR#1{}
}

\title[Density of Mirzakhani's Twist Tori]{Veech Surfaces and Expanding Twist Tori on Moduli Spaces of Abelian Differentials}

\author{Jon Chaika}
\address{Department of Mathematics, University of Utah, Salt Lake City, UT.}
\email{chaika@math.utah.edu}

\author{Osama Khalil}
\address{Department of Mathematics, Statistics, and Computer Science, University of Illinois Chicago, IL.}
\email{okhalil@uic.edu}

\begin{document}

\begin{abstract}

Let $(M,\w)$ be a translation surface such that every leaf of its horizontal foliation is either closed, or joins two zeros of $\w$.
Then, $M$ decomposes as a union of horizontal Euclidean cylinders.  
The \textit{twist torus} of $(M,\w)$, denoted $\T(\w)$, consists of all translation surfaces obtained from $(M,\w)$ by applying the horocycle flow independently to each of these cylinders. 
Let $g_t$ be the Teichm\"uller geodesic flow.
We study the distribution of the expanding tori $g_t\cdot \T(\w)$ on moduli spaces of translation surfaces in cases where $(M,\w)$ is a \textit{Veech surface}.
We provide sufficient criteria for these tori to become dense within the conjectured limiting locus $\Mcal :=\overline{\SL\cdot \T(\w)}$ as $t\r\infty$.
We also provide criteria guaranteeing a uniform lower bound on the mass a given open set $U\subset\Mcal$ must receive with respect to any weak-$\ast$ limit of the uniform measures on $g_t\cdot \T(\w)$ as $t\r\infty$. 
In particular, \textit{all} such limits must be fully supported in $\Mcal$ in such cases.
Finally, we exhibit infinite families of well-known examples of Veech surfaces satisfying each of these results.
A key feature of our results in comparison to previous work is that they do not require passage to subsequences.

\end{abstract}

\maketitle

\tableofcontents
\section{Introduction}

This article studies the distribution of translates of certain tori in moduli spaces of translation surfaces under the action of the Teichm\"uller geodesic flow.
The tori we study arise from horizontal shearing deformations of a fixed horizontally periodic translation surface. 
Our main results provide sufficient criteria under which these translated tori become dense (Theorem~\ref{thm:density of tori}), and, in other cases, guaranteeing that all weak-$\ast$ limits of their uniform measures are fully supported within the conjectured locus (Theorem~\ref{thm:full support}).
Among the ingredients in the proof are a strengthening of Forni's full density convergence for expanding horocycle arcs (Theorem~\ref{thm:uniform Banach density one}), and a measure rigidity result for horocycle flow-invariant measures on projective bundles arising from locally constant cocycles over quotients of $\SL$ (Theorem~\ref{thm:A-invariance of flag distribution}).
These results are motivated by the problem of equidistribution of expanding horocycle arcs conjectured by Forni, as well as by the analogous  twist torus conjecture of Mirzakhani in the context of hyperbolic surfaces as we now describe.

\subsection{Twist tori in moduli spaces}

A \textit{translation surface} is a pair $(M,\w)$ of a Riemann surface $M$, equipped with a holomorphic $1$-form $\w$. 
A \textit{stratum} is a moduli space of translation surfaces, where the number and orders of the zeros of the $1$-form are fixed. Strata are equipped with an action of $\SL$, induced from its linear action on polygonal presentations of translation surfaces. 
Of interest in this article are actions of the subgroups
\begin{align}
\label{eq:A and U}
    A =\set{ g_t =\begin{pmatrix}
        e^t & 0 \\ 0 & e^{-t}
\end{pmatrix}: t\in\R},
    \qquad \text{and} \qquad 
     U=\set{u(s) = \begin{pmatrix}
        1 & s \\ 0 & 1
    \end{pmatrix}: s\in\R}
\end{align}
 of $\SL$ generating the Teichm\"uller geodesic and horocycle flows respectively.
We refer the reader to the surveys~\cite{Zorich-survey,Yoccoz-ClaySurvey,ForMat,AthreyaMasur-book} for background on these objects.

Fix a translation surface $(M,\w)$ , and let $\Sigma\subset M$ denote the zeros of $\w$. 
The (singular) foliation induced by the imaginary (resp. the real) part of $\w$ is called the \textit{horizontal} (resp. \textit{vertical}) \textit{foliation}.
Throughout this introduction, we assume $(M,\w)$ to be \textit{horizontally periodic}, i.e., every leaf of its horizontal foliation is either closed, or a \textit{saddle connection} (i.e., a leaf joining points in $\Sigma$).

A \textit{horizontal cylinder} is a connected component of $M\setminus \set{\text{horizontal saddle connections}}$, i.e., a maximal connected family of closed horizontal leaves on $M$.
Let $\set{C_i:1\leq i\leq n}$ denote the set of such cylinders.
 We can realize each of the flat surfaces $(C_i,\w |_{C_i})$ as the image of a Euclidean parallelogram $P_i\subset \C$, with two edges parallel to the real axis, under identification of the two other edges by horizontal translation. 
In these coordinates, $\w |_{C_i}$ is given by pullback of the canonical $1$-form $dz$ restricted to $P_i$.
The \textit{core curve} of $C_i$ is the homology class of one (and hence any) of the closed horizontal leaves contained in $C_i$.

The \textit{twist torus} of $(M,\w)$, henceforth denoted $\T(\w)$, is the set of all translation surfaces obtained from $(M,\w)$ by first applying some horocycle flow matrix $u_i\in U$ to each parallelogram $P_i$, re-gluing the non-horizontal edges of $u_i P_i$ by translation to get a new set of horizontal cylinders $u_i \cdot (C_i,\w |_{C_i})$, then re-gluing the new cylinders along their horizontal boundaries using the same pattern of identifications of the old cylinders; cf.~\textsection\ref{sec:twists} for precise definitions.

 Let $h_i=\int_{C_i} \Im(\w), w_i = \int_{C_i} \Re(\w),$ and $m_i=h_i/w_i$ denote the height, circumference, and modulus of $C_i$ respectively. 
 In particular, $u(m_i) \cdot (C_i,\w|_{C_i})$ is isometric to $(C_i,\w|_{C_i})$, 
 Hence, the surjective map $U^n\cong \R^n \rightarrow \T(\w)$, sending a tuple $(u_i)$ to the surface built from the polygons $(u_i P_i)$ as above, factors through the torus $\T^n\cong \R^n/\prod_{i=1}^n m_i\Z $.
 This parametrization also endows $\T(\w)$ with a Lebesgue probability measure, denoted $\mu_\T$, inherited from the uniform measure on $\T^n$.

\subsection{Motivating conjectures}
The tori $\T(\w)$ give rise to submanifolds within the \textit{unstable horospherical leaf} $W^u(\w)$ of $(M,\w)$ inside the locus $\Mcal := \overline{\SL\cdot\T(\w)}$.
Roughly, $W^u(\w)$ is a disk of maximal dimension around $\w$ in $\Mcal$ with the property that the diameter of $g_{-t} W^u(\w)$ tends to $0$ as $t\r\infty$; cf.~\textsection\ref{sec:(un)stable} for definitions.
The study of expanding translates of horospherical leaves under the action of (non-uniformly) hyperbolic flows has a long history, in part due to its connection to counting asymptotics of periodic orbits \cite{Margulis-thesis}, and is intimately tied to mixing properties of the flow; cf.~\cite{Smillie2WeissYgouf} and references therein for results in this direction in the context of Teichm\"uller dynamics.
In particular, it is known that pushforwards of suitable Lebesgue class measures on $W^u(\w)$ under $g_t$ equidistribute towards the unique $\SL$-invariant and ergodic, probability measure $\mu_\Mcal$ fully supported on\footnote{That $\Mcal$ is the support of an ergodic $\SL$-invariant measure is a consequence of the work of Eskin, Mirzakhani, and Mohammadi~\cite{EMM}; cf.~Lemma~\ref{lem:Mcal is ergodic}.} $\Mcal $ \cite[Theorem 1.4]{Smillie2WeissYgouf}, see also~\cite{LindenstraussMirzakhani,ForniDensity1,EMM-EffectiveSCC} for related results.

The problem of understanding the limiting distributions of the tori $g_t\cdot \T(\w)$ is far more delicate.
Indeed, in general, the tori $\T(\w)$ have positive codimension within the unstable leaf $W^u(\w)$, and, hence, the distribution of their pushforwards is not directly connected to mixing of $g_t$.
 Nonetheless, 
 it was shown by Forni in~\cite{ForniDensity1} that the work of Eskin, Mirzakhani, and Mohammadi~\cite{EM,EMM} implies that the measures $(g_t)_\ast\mu_\T$ converge towards $\mu_\Mcal$ a along a sequence of times $t\r\infty$ of full density in $\R_{\geq 0}$.
 Indeed, the tori $\T(\w)$ are foliated by $U$-orbit segments to which the aforementioned results apply.
In light of these results, it is natural to raise the following problem.

\begin{problem}
\label{prob:Mirzakhani}
In the above notation, do the measures $(g_t)_\ast\mu_\T$ converge in the weak-$\ast$ topology towards $\mu_\Mcal$ as $t\to\infty$ ? 
\end{problem}

Problem~\ref{prob:Mirzakhani} provides intermediate grounds towards the following well-known problem in Teichm\"uller dynamics regarding equidistribution of expanding horocycle arcs under the geodesic flow, which served as our primary motivation for studying Problem~\ref{prob:Mirzakhani}.

\begin{problem}[{\cite[Conjecture 1.4]{ForniDensity1}}]
\label{prob:horocycles}
    Let $q=(M',\w')$ be an arbitrary translation surface. Do the measures $\int_0^1 \d_{g_t u(s) q}\;ds$, supported on expanding horocycle arcs through $q$, converge to the unique $\SL$-invariant, ergodic, probability measure fully supported on $\overline{\SL\cdot q}$ as $t\to\infty$? Here, $\d_x$ denotes the Dirac mass at $x$. 
\end{problem}

Beyond its intrinsic interest, Problem~\ref{prob:horocycles} has many applications to counting problems in flat geometry~\cite{EskinMasur}.
Nonetheless, it is currently wide open outside of certain special settings~\cite{EskinMarklofMorris,EskinMasurSchmoll,BSW}. 
We refer the reader to~\cite[Conjecture 1.5]{LindenstraussMohammadiWang-effective} for a discussion of connections between Problem~\ref{prob:horocycles} and recent developments on its effective analogues in homogeneous dynamics.

Another motivation for studying Problem~\ref{prob:Mirzakhani} comes from \textit{Mirzakhani's twist torus conjecture} in the related context of moduli spaces of hyperbolic surfaces; cf.~\cite[Problem 13.2]{Wright-Survey}.
 To formulate this conjecture, fix a pants decomposition $\Pcal$ of a genus $g$ surface $S$. 
 For each $L>0$, there is a torus in the moduli space of hyperbolic structures on $S$, consisting of surfaces obtained by taking each cuff in $\Pcal$ to be length $L$ and performing all possible Dehn twists around those cuffs. The aforementioned conjecture of Mirzakhani predicts the limiting distributions of those tori as $L \to\infty$.

The connection between these two problems was made precise in the work of Calderon and Farre~\cite{CalderonFarre-TwistTorus}, who generalized Mirzakhani's work on measurable conjugacies between horocycle flows on the space of flat structures and earthquake flows on the space of hyperbolic structures; cf.~\cite{Mirzakhani-conjugacy,CalderonFarre-ContinuityMirzakhaniConjugacy}.
Roughly speaking, under this conjugacy, Mirzakhani's twist tori give rise to certain expanding flat twist tori of the form in Problem~\ref{prob:Mirzakhani}; cf.~\cite[Section 4.1]{CalderonFarre-TwistTorus} for the precise correspondence.
Moreover, and crucially, they show that this, apriori measurable conjugacy, maps weak-$\ast$ limits of uniform measures on certain sequences of hyperbolic twist tori to limits of the corresponding measures on flat tori.
Combined with the above full density results of~\cite{EMM,ForniDensity1}, they identified the limit of the hyperbolic twist tori along the corresponding sequence of $L\r\infty$.
In particular, it follows from their work that a complete answer to Mirzakhani's twist torus conjecture follows from an affirmative answer to Problem~\ref{prob:Mirzakhani}.

\subsection{Main results}

Our main results, Theorems~\ref{thm:density of tori} and~\ref{thm:full support}, provide partial progress on Problem~\ref{prob:Mirzakhani} by addressing its topological forms in certain cases where the twist torus contains a Veech surface.

\subsubsection{Density of twist tori}
    
    The \textit{Veech group} of a translation surface $(M,\w)$, denoted $\mrm{SL}(M,\w)$, is its stabilizer for the action of $\SL$ on the stratum containing $(M,\w)$.
    Recall that $(M,\w)$ is said to be a \textit{Veech surface} if $\mrm{SL}(M,\w)$ is a lattice in $\SL$, i.e., $\mrm{SL}(M,\w)$ is discrete with finite covolume.
    It is known by a result of Smillie that $(M,\w)$ is a Veech surface if and only if the orbit $\Vcal := \SL\cdot (M,\w)$ is closed \cite[pg 226]{Veech-SmillieThm}.
    We refer to such closed orbits as \textit{Veech curves} based on the fact that Veech initiated their modern study in dynamics.
    
    Given an $\SL$-orbit closure $\Mcal$  containing $(M,\w)$, we say that $(M,\w)$ is \textit{$\Mcal$-primitive} if there is no proper intermediate $\SL$-orbit closure between $\Vcal$ and $\Mcal$,  i.e., if $\Ncal$ is an $\SL$-orbit closure with $\Vcal \subseteq \Ncal\subseteq \Mcal$, then $\Ncal=\Vcal$,  or  $\Ncal=\Mcal$.

\begin{theoremAlph}\label{thm:density of tori}
    Let $(M,\w)$ be a horizontally periodic Veech surface, and let $\Mcal = \overline{\SL\cdot\T(\w)}$.
    Assume that $(M,\w)$ is $\Mcal$-primitive.
    Then, 
    the expanding tori $g_t\cdot \T(\w)$ become dense in $\Mcal$ as $t\r\infty$,
    i.e., for every $\e>0$ and compact set $\Kcal\subset \Mcal$, there is $t_0>0$ so that $\Kcal\cap g_t\cdot \T(\w)$ is $\e$-dense in $\Kcal$ for all $t\geq t_0$.
\end{theoremAlph}

We note that, if $(M,\w)$ is square-tiled\footnote{Recall that $(M,\w)$ is said to be \textit{square-tiled} if its Veech group is commensurable with $\mrm{SL}_2(\Z)$, or equivalently, if $M$ is a finite-sheeted translation cover of a flat torus branched over one point~\cite{Veech-TeichEisenstein,GutkinJudge}.}, 
then it can be shown that the torus $\T(\w)$ has a dense subset of square-tiled Veech surfaces. 
Moreover, $\T(\w)$ meets the closed $\SL$-orbit of each of these surfaces in a periodic horocycle. In this case, Theorem~\ref{thm:density of tori} is an immediate consequence of the work of Eskin, Mirzakhani, and Mohammadi, and the fact that these periodic horocycles become dense within their respective closed $\SL$-orbits.

On the other hand, it follows from the finiteness results of~\cite{EskinFilipWright} that, if the Veech group of a Veech surface $(M,\w)$ has trace field of degree $\geq 3$, then $\T(\w)$ intersects at most finitely many Veech curves; cf.~Proposition~\ref{prop:dense examples}.  
In Section~\ref{sec:eg density}, we use this criterion to provide an infinite family of examples satisfying Theorem~\ref{thm:density of tori}, but where the above direct argument for square-tiled surfaces is not available. 
These examples are obtained by gluing parallel sides of regular $2n$-gons by translations, for $n>5$.
That these examples are $\Mcal$-primitive
is a consequence of strong results on orbit closures in hyperelliptic components of strata~\cite{McMullen-ClassificationGenusTwo,Ca,Apisa-HighRankInHyp,Apisa-Rk1InHyp}; cf.~Proposition~\ref{prop:dense examples}.

\begin{rem}
\label{rem:generalizations}
    We note some generalizations of Theorem~\ref{thm:density of tori} that follow from our arguments:
    \begin{enumerate}
        \item \textbf{The Decagon.}
        The proof of Theorem~\ref{thm:density of tori} proceeds by analyzing small pieces of $\T(\w)$ locally near the Veech curve $\Vcal=\SL\cdot (M,\w)$, which are well-approximated by their linearizations.
        The $\Mcal$-primitivity hypothesis is used to rule out that such linearizations collapse on tangent spaces of intermediate orbit closures under the action of $g_t$; cf.~Lemma~\ref{lem:dist to exceptional orbit closures}. 
        In some cases, this collapse can be ruled out in absence of $\Mcal$-primitivity using information on the position of these intermediate tangent spaces relative to the Lyapunov spaces of the derivative of $g_t$.
        To highlight this flexibility in our methods, we show in Appendix \ref{sec:decagon} that the conclusion of Theorem \ref{thm:density of tori} continues to hold for the \textit{decagon surface}, even though it is not $\mathcal{M}$-primitive.
        Note that the decagon is the unique Veech surface within its twist torus, up to the action of $U\subset \SL$; cf. Proposition~\ref{prop:dense examples}\eqref{item:2n-gon finite}.
        
        \item \textbf{Proper sub-tori.} The analogue of Theorem~\ref{thm:density of tori} holds if the full torus $\T(\w)$ is replaced with a proper, $U$-invariant, sub-torus $\T'\leqslant \T(\w)$ containing the periodic horocycle through $\w$, and $\Mcal$ is replaced with $\Mcal':=\overline{\SL\cdot \T'}$, under the hypothesis that $(M,\w)$ is $\Mcal'$-primitive.
        
    \end{enumerate}
\end{rem}

\subsubsection{Full support of limiting measures}

Our next result provides criteria under which \textit{all} limit measures of the expanding tori $g_t\cdot \T(\w)$ as $t\r\infty$ are fully supported in the locus $\overline{\SL\cdot \T(\w)}$.
Note that the property of weak-$\ast$ limits having full support is much stronger than density of $g_t\cdot \T(\w)$ as $t\r\infty$, and is new in all cases considered, even when $\T(\w)$ contains a dense set of Veech surfaces.

 To formulate the result, let $\Sigma(\w)\subset M$ denote the finite set of zeros of $\w$.
 Recall that $(M,\w)$ admits an atlas of charts to $\C$ where transition maps are given by translations, and in which $\w$ is the pullback of the canonical $1$-form $dz$.
 Let $\mrm{Aff}^+(M,\w)$ denote the group of orientation preserving homeomorphisms of $M$ that fix $\Sigma(\w)$, and are given by affine maps in these charts. 

Denote by $\cylspace(\w)\subseteq H^1(M,\Sigma(\w);\R)$ the smallest $\mrm{Aff}^+(M,\w)$-invariant subspace containing all dual classes to core curves of horizontal cylinders of $(M,\w)$.
In particular, $\cylspace(\w)$ contains the \textit{tautological plane} of $\w$ spanned by its real and imaginary parts. The tautological plane is invariant by $\mrm{Aff}^+(M,\w)$, and admits an $\mrm{Aff}^+(M,\w)$-invariant complement inside $\cylspace(\w)$; cf.~\textsection\ref{sec:balanced spaces} for the precise construction. We denote this invariant complement by $\cylspace^0(\w)$.

\begin{theoremAlph}
    \label{thm:full support}
        Let $(M,\w)$ be a horizontally periodic Veech surface, and let $\Mcal = \overline{\SL\cdot\T(\w)}$.
    Assume that 
    \begin{equation}\label{eq:pA=id hypothesis}
        \text{a pseudo-Anosov element of } \mrm{Aff}^+(M,\w) \text{ acts as the identity matrix on } \cylspace^0(\w).
    \end{equation}
    Then, every weak-$\ast$ limit of the measures $(g_t)_\ast\mu_\T$ as $t\r\infty$ has full support in $\Mcal$, where $\mu_\T$ is any fully supported Lebesgue probability measure on $\T(\w)$.
    More precisely, for every non-empty open set $V\subseteq \Mcal$, there is $\e>0$ and $t_0>0$ so that $(g_t)_\ast\mu_\T(V)>\e$ for all $t\geq t_0$.
\end{theoremAlph}

\begin{rem}\begin{enumerate}
    \item Note that we do not require $(M,\w)$ to be $\Mcal$-primitive in Theorem~\ref{thm:full support}.    

    \item Hypothesis~\eqref{eq:pA=id hypothesis} holds whenever a pseudo-Anosov element acts as the identity matrix on the entire complement of the tautological plane with respect to the intersection form; cf.~\textsection\ref{sec:balanced spaces}.
\end{enumerate}
\end{rem}

In \textsection\ref{sec:eg full support}, we recall results of Matheus and Yoccoz in~\cite{MatYoc} implying that the infinite family of square-tiled surfaces constructed in \textit{loc.~cit.} satisfy
Hypothesis~\eqref{eq:pA=id hypothesis}; cf.~\cite[Section 2]{ApisaWright-EW} for a recent generalization of this construction.
These examples include the well-known  \textit{Eierlegende-Wollmilchsau} surface first studied in~\cite{Forni-EW,HS-EW}, and the \textit{Ornithorynque} studied in~\cite{ForniMatheus-Orni}. In the latter two examples, the subgroup acting trivially on the complement of the tautological plane in fact has finite index in the affine group, while in all the other examples in the Matheus-Yoccoz family, such subgroup has infinite index.

The proof of Theorem~\ref{thm:full support} suggests that Problem~\ref{prob:Mirzakhani} admits an affirmative answer for Veech surfaces satisfying~\eqref{eq:pA=id hypothesis}; see \textsection\ref{sec:outline} for a sketch of the proof of Theorem~\ref{thm:full support}.
In particular, we suspect the following conjecture may be approachable with further refinements of our methods.
\begin{conjecture}
    Let $(M,\w), \mu_\T,$ and $\Mcal$ be as in Theorem~\ref{thm:full support}, and let $\mu_\Mcal $ be the unique $\SL$-ergodic probability measure fully supported on $\Mcal$. Then, $(g_t)_\ast\mu_\T$ converges to $\mu_\Mcal$ as $t\r\infty$.
\end{conjecture}
On the other hand, even strengthening the conclusion of Theorem~\ref{thm:density of tori} from density to full support of limit measures seems to require significant new ideas.

\subsection{Convergence along full Banach density of times}

Among the ingredients in the proof of Theorems~\ref{thm:density of tori} and~\ref{thm:full support} is Theorem~\ref{thm:uniform Banach density one} below, which is an equidistribution result for $g_t$ pushes of horocycle arcs along a full \textit{Banach density} set of times $t$. 
This result is deduced from the fundamental results of~\cite{EMM}, and strengthens an analogous statement obtained by Forni in~\cite{ForniDensity1}, where convergence along \textit{full density} sequences of times was established by more abstract arguments.
We note that our proof of Theorem~\ref{thm:uniform Banach density one} uses the full strength of the results of \cite{EMM}, and, in particular, Theorem \ref{thm:uniform Banach density one} does not hold in the generality of the results of \cite{ForniDensity1}.

\begin{thm}\label{thm:uniform Banach density one}

Let $\Mcal$ be an $\SL$-orbit closure, 
$\epsilon>0$, and $f\in C_c(\Mcal)$.
Then, there exist $L_0>0$ and proper $\SL$-orbit closures $\mathcal{N}_1, \dots, \mathcal{N}_k$ in $\Mcal$ such that for any compact set $F\subset \Mcal\setminus \cup_{i=1}^k \mathcal{N}_i$, we can find $S_0\geq 0$, so that for all $L\geq L_0$,  $S\geq S_0$ and $x\in F$, we have
\begin{equation}
\# \set{ \ell \in [S,S+L]\cap \mathbb{N}: \left|\int_0^1 f(g_\ell u(s) x)\;ds - \int f\;d\mu_{\Mcal}\right| <\epsilon }  >(1-\epsilon)L.
\end{equation}
Here, $\mu_{\Mcal}$ is the unique $\SL$-invariant Borel probability measure whose support is $\Mcal$.
\end{thm}

\begin{rem}
    The key feature of Theorem~\ref{thm:uniform Banach density one} is the uniformity of the parameter $L_0$ over the entire generic set $\Mcal\setminus \cup_{i=1}^k \mathcal{N}_i$. 
    Such uniformity is crucial for the application towards Theorem~\ref{thm:density of tori}.  
\end{rem}

\subsection{A-invariance of the distribution of cocycle output directions}\label{sec:A intro}

Another key ingredient in our proof of Theorem~\ref{thm:density of tori} is the following general rigidity result regarding $\SL$-actions on fiber bundles over its finite volume quotients, induced from linear representations of its lattices, which may be of independent interest. 
The result asserts that all limiting measures of expanding horocycle arcs on such fiber bundles are necessarily invariant by the geodesic flow.  We note that this result is not needed in the proof of Theorem~\ref{thm:full support}; cf.~Section~\ref{sec:outline} for further discussion.

To state the result, we need some notation.
Let $\G$ be a lattice in $G=\SL$, $\Vcal=G/\G$, and $\mu_\Vcal$ be the $G$-invariant probability measure on $\Vcal$. 
Let $\rho:\G \r \mrm{GL}_{d+1}(\R)$ be a representation of $\G$, $d\geq 1$, and denote by $\RP$ the $d$-dimensional projective space.
Then, $\G$ acts diagonally on $G\times \R^{d+1}$ by $\g\cdot (g,v)=(g\g^{-1}, \rho(\g)v)$, and $G$ acts on the first factor by left multiplication.
This induces similar actions of $\G$ and $G$ on $G\times \RP$.
Denote by $\widehat{\Vcal}$ and $\R\widehat\Vcal$ the quotient spaces of $G\times\R^{d+1}$ and $G\times \RP$ by $\G$ respectively.
In particular, $\hV$ and $\P\hV$ are fiber-bundles over $\Vcal$.  Since the actions of $G$ and $\G$ on $G\times \RP$ commute, $G$ also acts on the quotient $\P\hV$.

For $x\in\Vcal$,  we denote its fiber in $\hV$ by $V_x$.
We shall assume that the fibers $V_x$ are equipped with a family of continuously varying norms $\norm{\cdot}_x$.
For $g\in\SL$ and $x\in\Vcal$, we use the notation $B(g,x):V_x\r V_{gx}$ to denote the linear map on the fibers induced from left multiplication by $g$ on $\hV$. We also use the same notation for the action on fibers of $\P\hV$.
Finally, we use $\norm{B(g,x)}_{\mrm{op}}$ to denote the operator norm of the linear map $B(g,x): V_x\r V_{gx}$.

\begin{thm}\label{thm:A-invariance of flag distribution}
Assume that for some $C\geq 1$ and a fixed norm $\norm{-}$ on $\SL$, we have
\begin{align}\label{eq:bounded cocycle}
    |\log \norm{B(g,x)}_{\mrm{op}} | \leq C \log \norm{g},
    \qquad
    \text{for all } g\in \SL,x\in\Vcal.
\end{align}
Then, every $U$-invariant probability measure, which projects to $\mu_\Vcal$, is $A$-invariant, where $U$ and $A$ are the subgroups of $\SL$ in \eqref{eq:A and U}.

In particular, for every $z\in\P\hV$, we have that every weak-$\ast$ limit measure as $t\to\infty$ of the collection of measures
        \begin{align}\label{eq:exp horocycles skewprods}
            \set{\int_0^1 \d_{g_tu(s) \cdot z}\;ds :t \geq 0}
        \end{align}
    is $A$-invariant.
\end{thm}

\begin{rem}\label{remark on A-inv thm}
\begin{enumerate}

    \item An important feature of Theorem~\ref{thm:A-invariance of flag distribution} is that it holds without any restrictions on irreducibility or the Lyapunov spectrum of the cocycle.
    Moreover, it is likely that the boundedness hypothesis~\eqref{eq:bounded cocycle} can be weakened to allow slow growth in the cusps when $\Gamma$ is non-cocompact.

    \item It follows by Theorem~\ref{thm:A-invariance of flag distribution} that every limit measure of the family in~\eqref{eq:exp horocycles skewprods} as $t\r\infty$ is $P:= AU$-invariant. 
     Under the following additional hypotheses on the representation, we show in Appendix~\ref{sec:appendix} that the $P$-action is \textit{uniquely ergodic}, i.e., it admits a unique invariant measure. 
     In particular, this implies that the measures in~\eqref{eq:exp horocycles skewprods} have a unique accumulation point.
     Note that such finer results are not needed for our proof of Theorem~\ref{thm:density of tori}.
    \begin{enumerate}
        \item \textbf{Representations with bounded image.} In this case, we show that the unique $P$-invariant measure is in fact $\SL$-invariant, and is roughly given locally by the product of $\mu_\Vcal$ with the image of the Haar measure on an orbit of the compact group  $\overline{\rho(\G)}$; cf. Theorem~\ref{thm:convergence of flag distribution compact} for a precise statement.
        \item \textbf{Proximal and irreducible representations.} In this case, it follows from the results of \cite{BonattiEskinWilkinson} that the unique $P$-invariant measure projects to $\mu_\Vcal$ with conditional measures along each fiber given by a Dirac mass on a suitable top Lyapunov space. 
        We provide a short proof of this special case in our setting in Theorem~\ref{thm:convergence of flag distribution proximal}.
    \end{enumerate}

\end{enumerate}
    
\end{rem}

\subsection{Organization of the article and proof ideas}
\label{sec:outline}
In Section~\ref{sec:background}, we recall necessary background and introduce notation to be used throughout the article. We also recall important recurrence and non-uniform hyperbolicity results needed for the proof.
Section~\ref{sec:examples} provides infinite families of examples satisfying Theorems~\ref{thm:density of tori} and~\ref{thm:full support}.
In Section~\ref{sec:full Banach density}, we prove Theorem~\ref{thm:uniform Banach density one} using the work of Eskin, Mirzakhani, and Mohammadi~\cite{EMM}.

In Section~\ref{sec:reduce to matching}, we state the key technical statement in the proof of Theorem~\ref{thm:density of tori}, which we refer to as the key matching proposition, Proposition~\ref{prop:key matching}.
Roughly, Proposition~\ref{prop:key matching} asserts that a small neighborhood of the expanded torus $g_t\cdot\T(\w)$ contains a large piece of the $P$-orbit of points in $\T(\w)$ that lie near our Veech curve $\Vcal=\SL\cdot(M,\w)$.
Here, $P=AU\subset\SL$ is the subgroup of upper triangular matrices.
Proposition~\ref{prop:key matching} does not require the $\Mcal$-primitivity hypothesis.
The rest of Section~\ref{sec:reduce to matching} is dedicated to the deduction Theorem~\ref{thm:density of tori} from this matching proposition with the aid of the uniform convergence result in Theorem~\ref{thm:uniform Banach density one}, 
which precisely concerns equidistribution of such large pieces of $P$-orbits.

Proposition~\ref{prop:key matching} is proved in Section~\ref{sec:proof of key matching}, with the key ingredient being Theorem~\ref{thm:A-invariance of flag distribution} in the case of projective vector bundles over Veech curves with fiber action given by the Kontsevich-Zorich cocycle on the invariant bundle $\cylspace^0(-)$.
Roughly, Theorem~\ref{thm:A-invariance of flag distribution} is applied to $g_t$-pushes of the horocycle arc through our horizontally periodic Veech surface $\w$, together with a tangent vector $\b$ to the twist torus $\T(\w)$, to show $A$-invariance of all possible weak limits as $t\r\infty$.
As noted in Remark~\ref{rem:generalizations}, since small pieces of $g_t\cdot \T(\w)$ near $\Vcal$ are well-approximated by their linearizations, this linear $A$-invariance implies that such pieces are close together at different times $t$.
This quickly implies Proposition~\ref{prop:key matching}.

Section~\ref{sec:A-invariance} is dedicated to the proof of Theorem~\ref{thm:A-invariance of flag distribution}.
As noted above, we in fact prove in Proposition~\ref{prop:meas class on suspension} that all $U$-invariant measures on the suspension space $\P\hV$ which project to Haar measure on $\Vcal$ must also be $A$-invariant.
The key idea behind the latter result is to show that the horocycle flow orbits of points that only differ in the fiber direction experience \textit{sub-polynomial divergence}, Lemma~\ref{lem:subpolynomial divergence}.
This lemma implies that the dominant direction of divergence of two general points in the suspension space $\P\hV$ under the $U$-action is parallel to the base $\Vcal$.
This ``fiber-bunching" property enables us to implement ideas from Ratner's proof of measure rigidity in the classical setup of homogeneous dynamics~\cite{Ratner-RagConjSL2}.

Appendix~\ref{sec:appendix} is dedicated to the proofs of consequences of Theorem~\ref{thm:A-invariance of flag distribution} stated in Remark~\ref{remark on A-inv thm}.
In particular, in the case the representation has bounded image, we show that the limiting measure is in fact $\SL$-invariant using the entropy ideas appearing in the proof of Ratner's theorems given in the work of Margulis and Tomanov; cf.~Proposition~\ref{prop:SL-invariance}.

Finally, Section~\ref{sec:trivial monodromy} is dedicated to the proof of Theorem~\ref{thm:full support}.
The strategy is similar to the proof of the density theorem, Theorem~\ref{thm:density of tori}, with the key matching proposition replaced with the much stronger measure-theoretic matching statement in Proposition~\ref{prop:match tori trivial monodromy}. 
Proposition~\ref{prop:match tori trivial monodromy} roughly shows that the average $g_t$-translates of shrinking pieces of the twist torus, over a moving long window of time of the form $[T, T+N]$, remain close to the single $g_T$-translate of a certain absolutely continuous measure $\l$ on $\T(\w)$ as $T\r\infty$.
This essentially amounts to saying that every weak-$\ast$ limit of $g_T\l$ are almost $A$-invariant.

The key step in the proof of Proposition~\ref{prop:match tori trivial monodromy} is Proposition~\ref{prop:match on Vcal trivial monodromy}, which plays the role of Theorem~\ref{thm:A-invariance of flag distribution}, but produces a stronger conclusion more directly using our hypothesis on monodromy.
In particular, the latter hypothesis is used to ensure Proposition~\ref{prop:match on Vcal trivial monodromy}~\eqref{item:equal matrices} on equality of cocycle matrices at matched points, rather than mere closeness of projective images of the tremor.

The key idea in the proof of Proposition~\ref{prop:match on Vcal trivial monodromy} can be summarized as follows.
Suppose we are given two nearby points $x,y\in\Vcal$ which differ only in the \textit{stable} horocycle direction.
Suppose further that both points are horizontally periodic, and that they share a horizontal twist cohomology class, i.e., after identifying the relative cohomology groups of the surfaces corresponding to $x$ and $y$ by parallel transport, the linear span of the (classes dual to) the horizontal cylinders of $x$ intersects that of $y$ non-trivially.
Let $\b$ be one such class in that intersection.
In our proof, $x$ and $y$ will belong to the expanded horocycle arc through $(M,\w)$ at two different times $t_1$ and $t_2=t_1 + \text{ period of pseudo-Anosov}$. The existence of such class $\b$ will be ensured using our monodromy hypothesis.

Let $\t\mapsto x(\t)$ be the path defined by $x_0=x$ and $\dot{x}(\t) \equiv \b$, and define $y(\t)$ similarly. Note that $x(\t)$ and $y(\t)$ are contained in the twist tori of $x$ and $y$ respectively.
The crucial observation is that the property of $x$ and $y$ differing only in vertical periods survives for the surfaces $x(\t)$ and $y(\t)$ for sufficiently small $\t$.
This means that the respective pieces of the twist tori at $x$ and $y$ remain close under the action of $g_t$ for \textit{all} $t\geq 0$.
Our monodromy hypothesis in fact allows us to produce many such classes $\b$ to account for open subsets of the corresponding twist tori.
Passing to the limit along any sequence $t_n\r\infty$, this produces almost $A$-invariance on positive measure sets in the sense of Proposition~\ref{prop:match tori trivial monodromy}.
Theorem~\ref{thm:full support} then follows from the latter result by an application of Theorem~\ref{thm:uniform Banach density one}.

\subsection{Further open questions}

    We end the introduction with several open problems in the study of horocycle flows on moduli spaces that overlap with the questions studied in this article.
 
 \begin{ques}
 Find an explicit surface whose $U$-orbit is Birkhoff generic for the Masur-Veech measure. Note that the topological version of this problem was considered in \cite[Proposition 1.7]{HooperWeiss}, which identified explicit points whose $U$-orbit is dense within their $\SL$-orbit closures.
 \end{ques}

 \begin{ques}
 Generalizing a construction of Calta~\cite{Ca},
 Smillie and Weiss constructed horocycle ergodic measures that are \textit{not} $\SL$-invariant, and give measure $0$ to the set of surfaces with horizontal saddle connections \footnote{The general construction did not appear in print but was described in \href{https://icerm.brown.edu/video_archive/101}{this video lecture} by Weiss.}.
 The constructions start with an $\SL$-ergodic measure, and then `push' it by an element of the horizontal subspace, such as real REL deformations. In a similar vein to Problems~\ref{prob:Mirzakhani} and~\ref{prob:horocycles}, it is natural to ask whether these measures converge when pushed by $g_t$ and, more generally, to ask for the possible weak-$\ast$ limits as $t\r\infty$.
 \end{ques}

 \begin{ques}
 Is the horocycle flow topologically recurrent? That is, for every translation surface $(M,\w)$ and $\epsilon>0$, is it true that the set $R(\w,\e)\stackrel{\mrm{def}}{=}\{s\in \R :d(u(s)\omega,\omega)<\epsilon\}$ is unbounded? Note,  there exist translation surfaces $(M,\omega)$ so that, for a fixed $\epsilon>0$, $R(\w,\e)$ has upper density $0$ \cite[Theorem 1.2]{CSW}.
 \end{ques}

\begin{acknowledgement}
    
    The authors thank Paul Apisa, Alex Eskin, Carlos Matheus, Barak Weiss, and Alex Wright for helpful discussions regarding this project.
    {J.C. is partially supported by NSF grants DMS-2055354, 2350393 and a Warnock chair.}
    O.K. acknowledges NSF support under grants  DMS-2337911 and DMS-2247713.
    
\end{acknowledgement}

\section{Preliminaries and Notation}
\label{sec:background}

In this section, we recall necessary background and refer the reader to the surveys~\cite{AthreyaMasur-book,ForMat,Yoccoz-ClaySurvey,Zorich-survey} for more background on the subject.
We also introduce standing notation to be used throughout the article; see in particular \textsection\ref{sec:torsion free} and \ref{sec:notation}. 
Finally, we recall several recurrence and hyperbolicity results for the Teichm\"uller geodesic flow that will be important for our arguments.

\subsection{Strata and the mapping class group}
\label{sec:basics}

Let $\HHm$ denote a stratum of marked translation surfaces. 
A marked translation surface is given by a map $\phi:(S,\Sigma)\to (M,\Sigma')$ where $S$ is a model surface, $\Sigma$ is a finite subset of $S$ and $\Sigma'$ is the set of cone points in $M$. 
Let $\mathrm{Mod}(S,\Sigma)$ denote the mapping class group of $(S,\Sigma)$, that is the group of isotopy classes of homeomorphisms of $S$ that fix $\Sigma$.
The quotient of $\HHm$ by the right action of $\mathrm{Mod}(S,\Sigma)$, denoted by $\HHu$, is the corresponding stratum of unmarked translation surfaces.

\subsection{Period coordinates and the $\SL$ action}
\label{sec:period coords}
Let $q\in \HHm$ be a point representing a marked flat surface $M_q$. We denote the marking map by $\phi:(S,\Sigma)\to (M_{q},\Sigma(q))$.
Then, $q$ determines a holonomy homomorphism on relative integral homology, $\hol_{q}:H_1(M_q,\Sigma(q);\Z)\r C$. 
In particular, $\hol_q$ can be viewed as an element of $H^1(M_q,\Sigma(q);\C)$. 
We recall the following identifications
\begin{align}\label{eq:tangent space of marked stratum}
    T_q\HHm\cong H^1(M_q,\Sigma(q);\C) \cong H^1(M_q,\Sigma(q);\R) \oplus H^1(M_q,\Sigma(q);\mathbf{i}\R),
\end{align}
where $\mathbf{i}=\sqrt{-1}$, $T_q\HHm$ is the tangent space at $q$, and the second identification is given by postcomposing a $\C$-valued class with coordinate projections.
We refer to elements of $H^1(-;\R)$ and $H^1(-;\mathbf{i}\R)$ as \textit{horizontal} and \textit{vertical} classes respectively.

\begin{rem}
    In what follows, to simplify notation, we use the notation $H^1_\C, H^1_\R$ and $H^1_{\mathbf{i}\R}$ to denote the groups $H^1(M_q,\Sigma(q);k), k=\C,\R,\mathbf{i}\R$ respectively when the surface $q$ is understood from context.
\end{rem}

 The map $q\mapsto \hol_q$ is referred to as \textit{holonomy period coordinates}.
We denote the real and imaginary components of $\hol_q$ by $\hol_{q}^{(x)}$ and $\hol_{ q}^{(y)}$ respectively.
The cohomology class $\hol_{q}^{(x)}$ is represented by the $1$-form $dx_{q}$, viewed as the real part of the holomorphic $1$-form determined by $q$. As a map on homology, it is given by $\hol_{ q}^{(x)}[\gamma]=\int_\gamma dx_{q}$; cf.~\cite[Section 2.1]{CSW} for more information.
We define the \textit{tautological subspace} of $H^1_\R$ at $q$, denoted $\mrm{Taut}_q$, by
\begin{align*}
    \mrm{Taut}_q \stackrel{\mrm{def}}{=}\mrm{Span}\set{\hol_q^{(x)}, \hol_q^{(y)}}.
\end{align*}

Viewing elements of $H^1_\C$ as linear maps on homology with values in $\C$, we note that $\SL$ acts on $H^1_\C$ by post-composition through its linear action on $\C$.
In particular, for $g\in \SL$ and $q\in \HHm$, we have the relation
\begin{align}\label{eq:SL action on hol}
    \hol_{gq} = g\circ \hol_q.
\end{align}
More explicitly, given $\t\in\R$, we have
\begin{align}\label{eq:g_t action on hol}
    \hol^{(x)}_{g_\t q} = e^\t\hol^{(x)}_q,
    \qquad 
    \hol^{(y)}_{g_\t q} = e^{-\t}\hol^{(y)}_q.
\end{align}
For $u^-(\s)=\left(\begin{smallmatrix}
    1 & 0 \\ \s & 1
\end{smallmatrix}\right)$, we have
\begin{align}\label{eq:u-minus action on hol}
    \hol^{(x)}_{u^-(\s)q}= \hol^{(x)}_q,
    \qquad 
    \hol^{(y)}_{u^-(\s)q}=\s \hol^{(x)}_q+ \hol^{(y)}_q.
\end{align}

Finally, for $q=(M_q,\w_q)\in\HHu$, we refer to the stabilizer of $q$ in $\SL$ as the \textit{Veech group} of $q$, and denote it by $\mrm{SL}(M_q,\w_q)$.
By taking derivatives of maps in affine group $\mrm{Aff}^+(M_q,\w_q)$ defined above Theorem~\ref{thm:full support}, we obtain a surjective homomorphism onto the Veech group with kernel the (finite) \textit{automorphism group} of $(M_q,\w_q)$, giving the following exact sequence
\begin{equation}\label{eq:Aff vs Veech gp}
    \begin{tikzcd}
        1 \arrow[r] 
        &\mrm{Aut}(M_q,\w_q) \arrow[r] 
        &\mrm{Aff}^+(M_q,\w_q) \arrow[r, "D"] 
        & \mrm{SL}(M_q,\w_q) \arrow [r]
        &1 .
    \end{tikzcd}
\end{equation}

\subsection{Torsion-free finite cover}
\label{sec:torsion free}
Our arguments require analysis of the Kontsevich-Zorich cocycle, recalled in \textsection\ref{sec:KZ cocycle} below, over the Veech curve in question.
To this end, it will be convenient to identify the cocycle as (a factor of) the derivative of the $\SL$ action in a suitable sense.
For such identification to be well-defined, we pass to a finite cover of $\HHu$ to bypass orbifold issues caused by the finite automorphism groups $\mrm{Aut}(M,\w)$.

Let $\mrm{Mod}_0$ be a torsion-free finite index subgroup of $\mrm{Mod}(S,\Sigma)$; cf. \cite[Theorem 6.9]{FarbMargalit} for a concrete choice of such subgroup.

\begin{lem}\label{lem:finite cover}
    Let $\widetilde{\Hcal}_\mrm{u}=\HHm/\mrm{Mod}_0$.
    Then, it suffices to prove the analogous statements of Theorems \ref{thm:density of tori} and \ref{thm:full support} in the finite cover $\widetilde{\Hcal}_\mrm{u}$ of the stratum $\HHu$.
\end{lem}
\begin{proof}
    
    Let $q=(M,\w)\in \HHu$ be a horizontally periodic translation surface, and retain the notation of Theorem \ref{thm:density of tori}.
    Let $\pi: \tHHu\r \HHu$ be the natural projection, and fix a lift $\tilde{q}\in \tHHu$ of $q$. 
    Let $\widetilde{\T}$ be the connected component of $\pi^{-1}(\T(\w))$ containing $\tilde{q}$.
    Denote by $\widetilde{\Vcal}$ and $\widetilde{\Mcal}$ the closed orbit $\SL\cdot \tilde{q}$ and the orbit closure $\overline{\SL\cdot \widetilde{\T}}$ respectively.
    In particular, $\widetilde{\Vcal}$ and $\widetilde{\Mcal}$ project to $\Vcal$ and $\Mcal$ respectively.

    With this notation in hand, note that, if $q$ is $\Mcal$-primitive, then $\widetilde{\Vcal}$ is $\widetilde{\Mcal}$-primitive.
    Indeed, a hypothetical intermediate orbit closure $\widetilde{\Vcal}\subsetneq \widetilde{\Ncal} \subsetneq\widetilde{\Mcal}$ must have an intermediate dimension between those of $\widetilde{\Vcal}$ and $\widetilde{\Mcal}$. 
    Since $\tHHu$ is a finite cover of $\HHu$, dimensions are preserved under projections and, hence,
    the projection $\pi(\widetilde{\Ncal})$ would violate $\Mcal$-primitivity of $q$.
    Moreover, density of $g_t\cdot \tilde{\T}$ in $\widetilde{\Mcal}$ as $t\r\infty$ implies the analogous statement in $\Mcal$.

    For Theorem \ref{thm:full support}, let $\mu_{\tilde{\T}}$ be a Lebesgue probability measure on $\tilde{\T}$ projecting to $\mu_\T$ under $\pi$.
    Since pseudo-Anosov elements have infinite order, $\mrm{Mod}_0\cap \mrm{Aff}^+$ must still satisfy the hypothesis \eqref{eq:pA=id hypothesis}.
    Moreover, note that $(g_t)_\ast\mu_\T(V) = (g_t)_\ast\mu_{\tilde{\T}}(\pi^{-1}(V))$ for any Borel set $V$, and for all $t\geq 0$. Hence, if the conclusion of Theorem \ref{thm:full support} holds in for $(g_t)_\ast\mu_{\tilde{\T}}$, then it must also hold for $(g_t)_\ast\mu_\T$.
\end{proof}

The advantage of passing to the cover $\tHHu$ is that, for every translation surface $(M,\w)$, $\mrm{Mod}_0$ has trivial intersection with its finite automorphism group $\mrm{Aut}(M,\w)$, and hence, $\mrm{Mod}_0\cap \mrm{Aff}^+(M,\w)$ injects into the Veech group $\mrm{SL}(M,\w)$.
Moreover, the tangent space $T_\bullet\tHHu$ is well-defined at every point in $\tHHu$, and can be identified with a relative cohomology group as in \eqref{eq:tangent space of marked stratum}.
We shall thus adopt the following convention in the rest of the article.

\begin{convention}\label{torsion free convention}
    In light of Lemma \ref{lem:finite cover}, we shall assume and without loss of generality that $\mrm{Aut}(M,\w)$ is trivial.
    To avoid cluttering notation, 
    we shall continue to use the notation $\mrm{Aff}^+(M,\w)$ and $\mrm{SL}(M,\w)$ to denote their respective finite index subgroups $\mrm{Aff}^+(M,\w)\cap \mrm{Mod}_0$ and its image $D(\mrm{Aff}^+(M,\w)\cap \mrm{Mod}_0)$ under the derivative homomorphism $D$ in \eqref{eq:Aff vs Veech gp}.
    In particular, in this notation, we have
    \begin{align}
        \mrm{Aff}^+(M,\w) \cong \mrm{SL}(M,\w).
    \end{align}
    By a further slight abuse of notation, we shall also use $\HHu$ and $\mrm{Mod}(S,\Sigma)$ to denote $\tHHu$ and $\mrm{Mod}_0$ respectively.
    
\end{convention}

\subsection{Cylinder twists}\label{sec:twists}

Suppose $q\in \HHm$ is such that $M_q$ contains a horizontal cylinder $C$. 
Then, $C$ determines a cohomology class $\b_C\in H^1_\R$ defined as follows: $\b_C(\g)=0$ for all homology classes $\g$ in $H_1(M_q,\Sigma(q);\Z)$ represented by either the core curve of $C$ or a curve that is disjoint from $C$, and $\b_C(\g)$ equal to the height of $C$ for any curve $\g$ joining a zero in $\Sigma(q)$ on the bottom edge of $C$ to a zero on its top edge. 
As such curves span $H_1(M_q,\Sigma(q);\Z)$, this definition determines $\b$; cf.~\cite[Section 2]{Wright-Cylinder} for further properties of $\b_C$.

For $\t\in \R$, let $q_\t \stackrel{\mrm{def}}{=} \mrm{Trem}(q,\t\b_C)\in \HHm$ denote the translation surface obtained from $M_q$ by applying the horocycle flow $u(\t)$ to $C$ and the identity map to $M_q\setminus C$.
The surfaces $q$ and $q_\t$ are related in period coordinates by the following formula (cf.~\cite[Lemma 2.3]{Wright-Cylinder}): 
\begin{align}\label{eq:tremor in coordinates}
    \hol^{(x)}_{q_\t} = \hol^{(x)}_q + \t\b, 
    \qquad
    \hol^{(y)}_{q_\t} = \hol^{(y)}_q.
\end{align}
Note that equations\eqref{eq:u-minus action on hol}-~\eqref{eq:tremor in coordinates} remain valid for $q$ in the unmarked stratum $\HHu$ and for all $g\in\SL$ sufficiently close to identity so that $q$ and $gq$ both belong to a small ball on which period coordinates are injective.

We denote by $\twist(q)\subset H^1_\R$ the linear span of the cohomology classes defined above, i.e.,
\begin{align*}
    \twist(q) \stackrel{\mrm{def}}{=} \mrm{Span}\set{\b_C: C \text{ is a horizontal cylinder on } q}.
\end{align*}
Given $\t_1,\t_2\in \R$ and two classes $\b_1,\b_2\in \twist(q)\subset H^1_\R$, corresponding to two horizontal cylinders $C_1,C_2 \subset M_q$, we can define $\Trem(q, \t_1\b_1+\t_2\b_2)$  to be the surface obtained from $M_q$ by first applying $u_{\t_1}$ to $C_1$, followed by applying $u_{\t_2}$ to $C_2$.
Since horizontal cylinder twists commute with one another, the resulting surface is well-defined.
Finally, we note that the horocycle flow itself is a special example of cylinder twists. In particular, we have
\begin{align*}
    \b=\hol_q^{(y)} \Longrightarrow \Trem(q,\t\b) =u_\t q.
\end{align*}

The notation $\mrm{Trem}$ refers to \textit{tremor deformations} introduced and studied in~\cite{CSW}, of which (horizontal) cylinder twists form the simplest examples.
We refer the reader to~\cite{Wright-Cylinder} for the role of cylinder twists in the study of $\SL$-orbit closures.

\subsection{Balanced spaces}
\label{sec:balanced spaces}

Let $H^1_{\R,\mrm{abs}}$ denote the absolute cohomology group of $M_q$ and let $p:H^1_\R\r H^1_{\R,\mrm{abs}}$ denote the forgetful map.
We extend the intersection product from $H^1_{\R,\mrm{abs}}$ to relative cohomology $H^1_\R$ by composing it with the forgetful projection $p$.
Let $L_q$ denote the linear functional on $H^1_\R$ given by taking the intersection product with $\b$. More explicitly, we set $L_q(\b) = \int_{M_q} dx_q \wedge p(\b)$, where the notation indicates evaluating the cup product of $p(\hol_q^{(x)})$ with $p(\b)$ on the fundamental class of $M_q$. 
We set
\begin{align*}
    \twist^0(q) \stackrel{\mrm{def}}{=} \twist(q) \cap \mrm{Ker}(L_q ).    
\end{align*}
Note that $L_q(\hol^{(y)}_q)$ is the nonzero (signed) area of $M_q$. It follows that
\begin{align}
    \label{eq:twist=horocycle plus balanced}
    \twist(q)=\twist^0(q)\oplus \R\cdot\hol_q^{(y)}.
\end{align}

Moreover, since $L_q(\hol^{(y)}_q)\neq 0$, its restriction to the tautological plane $\mrm{Taut}_q$ is non-degenerate.
We denote by $\mrm{Taut}_q^0$ the orthogonal complement to $\mrm{Taut}_q$ with respect to this intersection form.
Following~\cite{CSW}, we say that a cohomology class is \textit{balanced} if it belongs to $\mrm{Taut}_q^0$.

The action of the affine group $\mrm{Aff}^+(q)$ on $M_q$ induces an action on $H^1_\R$, which preserves  the tautological plane $\mrm{Taut}_q$ and its complement $\mrm{Taut}_q^0$.
Moreover, since every element of $\twist^0(q)$ has $0$ intersection pairing with $\mrm{Taut}_q$, we have
\begin{align*}
    \twist^0(q) \subseteq \mrm{Taut}^0_q.
\end{align*}
We let $\cylspace^0(q)$ denote the smallest $\mrm{Aff}^+(q)$-invariant subspace of $\mrm{Taut}^0_q$ containing $\twist^0(q)$.
We set
\begin{align*}
    \cylspace(q) \stackrel{\mrm{def}}{=} \cylspace^0(q)\oplus \mrm{Taut}_q.
\end{align*}
In particular, the above splitting is $\mrm{Aff}^+(q)$-invariant, and
\begin{align*}
    \twist(q) \subseteq \cylspace(q).
\end{align*}

\subsection{The AGY norm}

    Given $q$ in the marked stratum $\HHm$, the AGY norm on $H^1_\C$, denoted $\norm{\cdot}_q$ is defined for every $v\in H^1_\C$ by
    \begin{align*}
        \norm{v}_q \stackrel{\mrm{def}}{=} \sup_{\g\in \L_q} \frac{|v(\g)|}{|\hol_q(\g)|},
    \end{align*}
    where $\L_q\subseteq H_1(S, \Sigma;\Z) $ denotes the set of saddle connections of $q$.
    This norm induces a (Finsler) metric denoted $\dAGY$ on $\HHm$ given by the infimum of lengths of $C^1$-paths joining points.
    Since these norms, and hence the metric, are invariant by the mapping class group, they descend to $\HHu$.
    The following Lipschitz estimate on norms of parallel transported vectors will be useful for our analysis.

\begin{prop}
    [{\cite[Proposition 5.5]{AG}}]
    \label{prop:norm of parallel transport}
    Let $\k:[0,1]\r \HHm$ be a $C^1$-path and $v\in H^1_\C$. Then,
    \begin{align*}
        e^{-\mrm{length}(\k)} 
        \leq \frac{\norm{v}_{\k(1)}}{\norm{v}_{\k(0)}}
        \leq e^{\mrm{length}(\k)},
    \end{align*}
    where $\mrm{length}(\k)=\int_0^1 \norm{\dot(\k)(t)}_{\k(t)}\;dt$.
    Moreover, for $q=\k(0)$, and for all $0\leq t< 1/\norm{v}_q$, we have
    \begin{align*}
        \int_0^t \norm{v}_{\k(s)}\;ds \leq -\log(1-t \norm{v}_q).
    \end{align*}
\end{prop}
\begin{rem}
    The second assertion of Proposition~\ref{prop:norm of parallel transport} follows from the proof given in the cited reference.
\end{rem}

We also need the following basic norm estimates of the derivative of the geodesic flow with respect to AGY norms.

\begin{lem}[{\cite[Lemma 5.2]{AG}}]
\label{lem:nonexpansion of stable}
    For all $q\in \HHu$, all $v\in H^1_{\mathbf{i}\R}$, and all $t\geq 0$, we have
    \begin{align*}
        \norm{Dg_t(q)v}_{g_tq} \leq \norm{v}_q,
    \end{align*}
    where $Dg_t(q):T_q\HHu\r T_{g_tq}\HHu$ is the derivative of the geodesic flow.
    Moreover, for all $v\in H^1_\C$,
    \begin{align*}
        e^{-2|t|} \norm{v}_q\leq \norm{Dg_t(q)v}_{g_tq} \leq e^{2|t|} \norm{v}_q
    \end{align*}
\end{lem}

\begin{proof}
    The first estimate was shown in~\cite[Lemma 5.2]{AG} for the action on the marked stratum $\HHm$, which implies the corresponding estimates in $\HHu$ by invariance of the AGY norms under the mapping class group. 
    As noted in the discussion following Lemma 5.2 in \cite{AG}, the second inequality follows by the same proof of the first bound.
\end{proof}

\subsection{Local (un)stable manifolds}
\label{sec:(un)stable}

We recall the parametrization of local strong stable/unstable manifolds. 
Define $E^s(q)$ (resp.~$E^u(q)$) as the subspace of $H^1_{\mathbf{i}\R}$ (resp.$H^1_\R$) with $0$ intersection product with $\hol_q^{(x)}$ (resp.~$\hol_q^{(y)}$), where the intersection product is extended to relative cohomology by composing it with the projection to absolute cohomology as in \textsection\ref{sec:balanced spaces}.
Let $v \in E^s(q)$ be such that there is a path $\k: [0,1]\to \HHm$ with $\k(0)=q$ and $\dot{\k}(t)=v$ for all $t\in [0,1]$. Then, we define $\Psi^s_q(v)=\k(1)$. 
    In coordinates, if $\Psi_q^s(v)$ is defined, then 
    \begin{align}\label{eq:stable exp in coords}
     \hol^{(y)}_{\Psi_q^s(v)}   = \hol^{(y)}_q +v,
     \qquad 
     \hol^{(x)}_{\Psi_q^s(v)} = \hol^{(x)}_q.
    \end{align}
    The map $\Psi^u_q$ is defined analogously on $\R$-valued cohomology classes.
    The maps $\Psi^s_\bullet$ and $\Psi^u_\bullet$ play the role of exponential maps parametrizing strong stable/unstable leaves using their respective tangent spaces.
    The following key properties for this map will be important for us.
\begin{prop}
    [{\cite[Proposition 5.3]{AG}}]
    \label{prop:stable exp lipschitz}
    For all $q\in \HHu$, the map $v\mapsto\Psi^s_q(v)$ is well-defined for $v\in E^s(q)$ with $\norm{v}_q<1/2$.
    Moreover,  we have the bi-Lipschitz estimates
    \begin{align*}
        \dAGY(q,\Psi^s_q(v)) \leq 2 \norm{v}_q,
   \qquad
    \text{and}
    \qquad
        1/2 \leq \frac{\norm{v}_{\Psi^s_q(v)}}{\norm{v}_q}\leq 2.
    \end{align*}
    The analogous estimates also hold for $\Psi^u_q$.
\end{prop}

Lemma~\ref{lem:nonexpansion of stable} implies the following natural equivariance property of the $\Psi^s_\bullet$.
\begin{cor}\label{cor:equivariance of stable exp}
For all $t\geq 0$, $q\in\HHu$, $v\in E^s(q)$ with $\norm{v}_q<1/2$, we have
   \begin{align*}
       g_t \Psi_q^s(v) = \Psi^s_{g_tq}(Dg_t(q)v).
   \end{align*}
\end{cor}
\begin{proof}
    First, we note that it suffices to prove the corollary in the marked stratum $\HHm$.
    By Lemma~\ref{lem:nonexpansion of stable}, we have $\norm{Dg_t(q)v}_{g_tq}<1/2$, and hence $\norm{Dg_t(q)v}_{g_tq}$ is well-defined by Proposition~\ref{prop:stable exp lipschitz}.
    Let $\k:[0,1]\r \HHm$ be a path such that $\k(0)=1, \k(1)=\Psi^s_q(v)$, and $\dot\k(r)=v$ for all $r\in [0,1]$.
    Then, $r\mapsto g_t\k(r)$ is a path joining $g_tq$ to $g_t \Psi_q^s(v)$, with constant derivative equal $Dg_t(q)v$.
    The corollary follows by definition of $\Psi^s_{g_tq}$.
    \qedhere
\end{proof}

The following substantial strengthening of Lemma~\ref{lem:nonexpansion of stable} follows from non-uniform hyperbolicity of the Teichm\"uller geodesic flow proved by Forni in~\cite[Lemma 2.1']{Forni-Annals}.
\begin{prop}
    [{\cite[Proposition 4.3]{AG}}]
    \label{prop:nonuniform hyp}
    Given a compact subset $L\subset\HHu$ and $\d>0$, there is $T=T(L,\d)>0$ such that for all $q\in L$, $v\in E^s(q)$, and $t\geq 0$ such that $g_tq\in L$ and the set of $r\in [0,t]$ with $g_rq\in L$ has measure $\geq T$, 
    we have $\norm{Dg_t(q)v}_{g_tq}\leq \d \norm{v}_q$.
\end{prop}
\begin{rem}
    Proposition~\ref{prop:nonuniform hyp} is stated in~\cite{AG} for $\d=1/2$, however the same argument works for any $\d>0$.
\end{rem}

\subsection{The Kontsevich-Zorich cocycle}
\label{sec:KZ cocycle}
The standard reference for the discussion in this section is~\cite{ForMat}.
Recall the notation set in Convention \ref{torsion free convention}.
The Universal Coefficient Theorem provides a splitting
\begin{align*}
    H^1(M_q,\Sigma(q);\C) \cong H^1(M_q,\Sigma(q);\R)\otimes \C.
\end{align*}
Recalling that the left hand-side is identified with the tangent space to $\HHu$ at $q$, we also have that this splitting is invariant by the derivative $Dg(q):T_q\HHu\r T_{gq}\HHu$ of $g\in \SL$ ; cf.~\cite[Section 2.2]{ChaikaKhalilSmillie}.
Moreover, there is a linear dynamical cocycle, known as the \textit{Kontsevich-Zorich cocycle} (KZ for short), and denoted\footnote{Note that the KZ cocycle is usually defined on absolute cohomology elsewhere in the literature.} $\mrm{KZ}(g,q): H^1(M_q,\Sigma(q);\R) \r H^1(M_{gq},\Sigma(gq);\R)$, so that the derivative can be written as
\begin{align*}
    Dg(q) = \mrm{KZ}(g,q)\otimes g,
\end{align*}
where $g$ acts on $\C$ via its standard linear action on the plane.
In particular, the chain rule implies the cocycle property
\begin{align}\label{eq:cocycle property}
    \mrm{KZ}(gh,q) = \mrm{KZ}(g,hq)\mrm{KZ}(h,q).
\end{align}
Alternatively, this cocycle can be defined as the quotient by $\mrm{Mod}(S,\Sigma)$ of the trivial cocycle over $\SL$ action on $\HHm\times H^1(S,\Sigma;\R)$, given by $g\cdot (\tilde{q},v) = (g\tilde{q},v)$.
This quotient gives a well-defined linear cocycle in light of our Convention \ref{torsion free convention}.

We record the following immediate corollary of Lemma~\ref{lem:nonexpansion of stable} on norm bounds of the cocycle.
Fix a matrix norm on $\SL$, denoted $\norm{\cdot}$.

\begin{cor}
    \label{cor:KZ norm}
    For all $g\in\SL$ and $q\in \HHu$, we have
    \begin{align*}
        \norm{\mrm{KZ}(g,q)}_{q\r gq} \ll \norm{g}^{O(1)},
    \end{align*}
    where $\norm{\mrm{KZ}(g,q)}_{q\r gq}$ denotes the operator norm of the cocycle with respect to the AGY norms at $q$ and $gq$ respectively.
\end{cor}
\begin{proof}
    The corollary follows by the polar decomposition for $\SL$, the cocycle property, $\mrm{SO}(2)$-invariance of AGY-norms, and Lemma~\ref{lem:nonexpansion of stable}.
\end{proof}

\subsection{Standing notation}
\label{sec:notation}
We introduce convenient notation to be used throughout the article.
Let $\w$ be a horizontally periodic Veech surface.
We use $\Vcal$ to denote the closed $\SL$-orbit of $\w$.
For convenience, we always assume that our Veech surface is $1$-periodic for $u(s)$, i.e. $u(1)\w=\w$.

Given $\b\in \twist(\w)$, and $t,s,r>0$, we let
    \begin{align}\label{eq:notation}
    \T(\w) &= \set{u(s)\cdot \Trem(\w,\b): \b\in \twist(\w),s\in\R},
        \qquad 
        &\w(t,s) &= g_t u(s)\w \in\Vcal,
    \nonumber\\ 
        \T(\w,\b) &=\set{u(s)\cdot\Trem(\w,r\b):r,s\in\R},
        \qquad 
        &\b(t,s) &= e^t\cdot \KZ{g_t}{u(s)\w}\cdot\b ,
        \nonumber\\
        \mrm{Trem}_\b(t,s,r) &= \mrm{Trem}(\w(t,s),r\cdot \b(t,s)) \in \HHu,
        \qquad
        &N_\b(t,s) &= \norm{\b(t,s)}_{\w(t,s)}.
    \end{align}

\begin{rem}
    When the vector $\b$ is fixed, we write $\Trem(t,s,r)$ and $N(t,s)$ for $\Trem_\b(t,s,r)$ and $N_\b(t,s)$ respectively to simplify notation.
\end{rem}

With the above notation, we recall the following equivariance property of horizontal cylinder twists under the action of  $g_t$.
\begin{lem}[{\cite[Lemma 2.4]{ChaikaKhalilSmillie} and \cite[Proposition 6.5]{CSW}}]
\label{lem:renorm}
Let $q\in \HHu$ and $\b\in\twist(q)$.
Then, for all $t\in\R$,
\begin{align*}
    g_t\cdot \mrm{Trem}(q,\b) 
    &= \mrm{Trem}(g_tq, e^t\cdot \mrm{KZ}(g_t,q)\cdot\b).
\end{align*}
\end{lem}
We will also need the following simple lemma.
\begin{lem}\label{lem:equivariance of notation}
    For all $t,\ell,s,r\in \R $ and $\b\in \twist(\w)$, we have
    \begin{enumerate}
        \item $N_\b(t+\ell,s)\leq e^{2|\ell|} N_\b(t,s) $.
        \item $ g_\ell \cdot \Trem_\b(t,s, r)= \Trem_\b(t+\ell,s,r)$.
    \end{enumerate}
\end{lem}
\begin{proof}
     The first assertion follows by the cocycle property and Corollary~\ref{cor:KZ norm}.
    The second assertion is a restatement of Lemma~\ref{lem:renorm}. 
\end{proof}

\subsection{Exponential recurrence and contraction of vertical classes}
We recall the following result asserting that except for a set of exceptionally decaying measure, geodesic flow orbits of points on the torus $\T(\w)$ spend a definite proportions of their time inside large compact sets.

\begin{prop}\label{prop:exp recurrence}
    Let $(M,\w)$ be a horizontally periodic surface and let $\mu$ be a $U$-invariant probability measure on its twist torus $\T(\w)$.
    Then, there is a compact set $L\subset\HHu$, and $\e_1,\e_2\in (0,1)$ such that for all large enough $T>0$, the set of $x\in \T(\w)$ with 
    \begin{align*}
        \int_0^T \mathbf{1}_L(g_tx)\;dt \leq (1-\e_1) T
    \end{align*}
    has $\mu$-measure at most $e^{-\e_2 T}$.
\end{prop}
\begin{proof}
    This result was shown in stronger form in~\cite[Proposition 3.9]{AAEKMU} for the Lebesgue probability measure on a horocycle arc of the form $\set{u(s)q:s\in [-1,1]}$, $q\in\HHu$, following ideas of~\cite{EskinMasur,Athreya,KKLM}, with uniform estimates as $q$ varied in fixed compact sets in $\HHu$.
    The claimed estimate now follows for $\mu$ since $\T(\w)$ is compact and since $\mu$ disintegrates as a convex combination of Lebesgue measures on horocycle arcs as above by $U$-invariance.
\end{proof}

Combined with Proposition~\ref{prop:nonuniform hyp}, the above recurrence result yields the following contraction estimate for the action of $g_t$ on the strong stable foliation.
\begin{cor}\label{cor:contraction of vertical classes}
    Let the notation be as in Proposition~\ref{prop:exp recurrence}.
    Then, for $\mu$-almost every $q\in \T(\w)$, we have $\sup \norm{Dg_t(q)\cdot v}_{g_tq} \xrightarrow{t\r \infty} 0$, where the supremum is over all vertical cohomology classes $v\in E^s(q)$ tangent to the strong stable leaf through $q$
    with $\norm{v}_q\leq 1$.
\end{cor}
\begin{proof}
    Fix an arbitrary $\d>0$ and let $L\subset \HHu$ be the compact set provided by Proposition~\ref{prop:exp recurrence}.
    Let $T=T(L,\d)>0$ be the parameter provided by non-uniform hyperbolicity in Proposition~\ref{prop:nonuniform hyp}.
    By Proposition~\ref{prop:exp recurrence} and the Borel-Cantelli lemma, for $\mu$-almost every $q\in\T(\w)$, we can find $t>T$ such that $g_tq\in L$ and
    \begin{align*}
        \mrm{Leb}(r\in [0,t]: g_r q\in L)>T.
    \end{align*}
    Hence, for each such $q$ and $t$, Proposition~\ref{prop:nonuniform hyp} gives $\sup \norm{Dg_t(q)\cdot v}_{g_tq}\leq \d $ for all unit norm classes $v\in E^s(q)$.
    The non-expansion estimate of Lemma~\ref{lem:nonexpansion of stable} then implies that $$\lim_{t\to\infty} \sup \set{\norm{Dg_t(q)\cdot v}_{g_tq}: v\in E^s(q), \norm{v}_q\leq 1 }\leq \d.$$ 
    The corollary now follows as $\d$ was arbitrary.
    \qedhere
\end{proof}


\section{Examples}
\label{sec:examples}
In this section, we provide infinite families of examples satisfying Theorems~\ref{thm:density of tori} and~\ref{thm:full support}.
These examples are meant to be illustrative rather than exhaustive.

\subsection{Examples for Theorem~\ref{thm:density of tori}}
\label{sec:eg density}

In what follows, for $n\geq 5$, we let $(M_n,\w_n)$ be the (horizontally periodic) Veech surface obtained from gluing parallel sides of the regular $2n$-gon with a horizontal edge by translations.
These surfaces were discovered by Veech~\cite{Veech-BilliardInRegularPolygon,Veech-TeichEisenstein} and have provided a rich source of examples in flat geometry since.
In Proposition~\ref{prop:dense examples}~\eqref{item:M-primitive examples}, we show that these surfaces satisfy the $\Mcal$-primitivity hypothesis of Theorem~\ref{thm:density of tori}.
Part~\eqref{item:2n-gon finite} of that proposition~\ref{prop:dense examples} shows that the conclusion of this theorem holds non-trivially for those examples since their twist tori meet at most finitely many closed $\SL$-orbits.

\begin{prop}\label{prop:dense examples}
\begin{enumerate}
    \item\label{item:2n-gon finite} {\cite{McM,Ca,McMullen-ClassificationGenusTwo,EskinFilipWright}}. For all $n\geq 5$, the twist torus $\T(\w_n)$ intersects at most finitely many closed $\SL$-orbits.
        Moreover, this property holds for any horizontally periodic Veech surface with trace field of degree $\geq 3$ over $\Q$.

    \item\label{item:M-primitive examples} {\cite{McM,Ca,McMullen-ClassificationGenusTwo,Apisa-HighRankInHyp,Apisa-Rk1InHyp}}.
    Let $\Mcal_n=\overline{\SL\cdot\T(\w_n)}$.
    Then, for $n>5$, $(M_n,\w_n)$ is $\Mcal_n$-primitive.
\end{enumerate}    
\end{prop}

We begin with the following useful lemma which allows us to control the trace field of Veech surfaces belonging to the same twist torus.
This lemma in fact proves a stronger property than what we need for the concrete examples $(M_n,\w_n)$ discussed in this section.

    \begin{lem}[{\hspace{.1pt}\cite{KenyonSmillie,Wright-Cylinder,Wright-FieldofDef}}]
    \label{lem:constant trace field}
        Let $(M,\w)$ be a horizontally periodic Veech surface, and suppose that $(M',\w')$ is another Veech surface in $\T(\w)$. Then, the trace fields of the Veech groups of both surfaces coincide.
    \end{lem}
    \begin{proof}
        Let $\Vcal =\SL \cdot (M,\w)$.
        Let $\set{C_i:1\leq i\leq n}$ be the full set of  horizontal cylinders of $M$, and without loss of generality assume that $n>1$.
        Let $c_i$ denote the circumference of the cylinder $C_i$.
        Then, these horizontal cylinders of $(M,\w)$ are $\Vcal$-parallel in the language of~\cite[Definition 4.6]{Wright-Cylinder}.
        It follows by~\cite[Theorem 7.1]{Wright-Cylinder} that the affine field of definition of $\Vcal$ is $\Q[c_2/c_1, \dots, c_n/c_1]$. 
        On the other hand, by~\cite[Theorem 1.1]{Wright-FieldofDef}, the affine field of definition is the same as the holonomy field of $(M,\w)$ (cf.~\cite[Appendix]{KenyonSmillie} for a definition of the holonomy field).
        By~\cite[Theorem 28]{KenyonSmillie}, since $(M,\w)$ is a Veech surface, and hence its Veech group contains at least one pseudo-Anosov element, its holonomy field coincides with the trace field of its Veech group.

        Now, $(M',\w')$ admits a horizontal cylinder decomposition of the form $M' = \cup_{i=1}^n u(s_i) C_i$, for some $s_i\in \R$.
        In particular, these cylinders have the same set of circumferences $\set{c_i:1\leq i\leq n}$. Thus, its trace field coincides with that of $(M,\w)$.
    \end{proof}

    In what follows, we let $\G_n$ be the Veech group of $(M_n,\G_n)$, i.e.,
    \begin{align*}
        \G_n = \mrm{SL}(M_n,\w_n).
    \end{align*}
     We let $k(\G_n)$ be its trace field.
     Given an integer $g\geq 2$, 
    we denote by $\Hcal^{hyp}(2g-2)$, respectively $\Hcal^{hyp}(g-1,g-1)$,  the hyperelliptic components of strata of translation surfaces of genus $g$  having either one zero of order $2g-2$, respectively  two zeros of order $g-1$ each and which are interchanged by a hyperelliptic involution of the underlying Riemann surface; cf.~\cite[Def. 2]{KontsevichZorich-ConnComps} for the precise definition.
     We will need the following elementary lemma.
    \begin{lem}
    \label{lem:2n-gons degree and hyperelliptic}
    For all $n >5$, the degree of the trace field satisfies $[k(\G_n):\Q]\geq 3$, and for $n=5$, we have $[k(\G_5):\Q]=2$.
    Moreover, for all $n\geq 5$, the surfaces $(M_n,\w_n)$ belongs to $ \mathcal{H}^{hyp}(n-2)$ when $n$ is even, and to $\mathcal{H}^{hyp}(\frac{n-3}2,\frac{n-3}2)$ when $n$ is odd.
    
    \end{lem}

    The proof of this lemma is standard and is included for completeness.
    
    \begin{proof}
    Without loss of generality, we assume the $2n$-gon generating $(M_n,\w_n)$ has unit length edges.
    By considering the top horizontal edge and the closest parallel chord to it respectively, we obtain two horizontal saddle connections with holonomy $(1,0)$ and $(\a_n,0)$, where $\a_n=1+2\cos(2\pi/2n)$.
    Thus, by~\cite[Theorem 28]{KenyonSmillie}, $\a_n$ belongs to the trace field $k(\G_n)$ of $\G_n$.
    Moreover, for $\zeta_n=\exp(2\pi i/2n)$,
    we have that $\Q(\a_n) = \Q(\zeta_n + \zeta_n^{-1})$ is the fixed subfield under the complex conjugation automorphism of $\Q(\zeta_n)$.
    Hence, since $\Q(\zeta_n)$ has degree $\phi(2n)$ over $\Q$, where $\phi$ is Euler's totient function, it follows that
    $[\Q(\a_n):\Q] =\phi(2n)/2$. 
    This implies the first part of the lemma.

    For the last claim, observe that when $n$ is even, all vertices of the regular $2n$-gon get identified, and hence $M_{n}$ has one cone point of cone angle $(2n-2)\pi$. which is access angle $(2n-2)\pi$ and the genus of the surface is $\frac{n}2$. Similarly, when $n$ is odd there are two cone points of equal angle $(n-1)\pi$, which is access angle $2\pi(\frac{n-1}2-1)$ and this is genus $\frac{n-1}2$. We now see that these are hyperelliptic. Observe that rotation by $\pi$ is a symmetry of the surface that has order 2. When $n$ is even it fixes $n+2$ which is twice genus plus 2 points. These points are the cone point, the center of the polygon and the midpoint of each side. By~\cite[Section 7.4]{FarbMargalit}, this is a hyperelliptic involution. Similarly, when $n$ is odd, rotation by $\pi$ is a symmetry of the surface that fixes $n+1$ points, which is twice genus +2. These points are the center of the polygon and the midpoint of each side. It exchanges two cone points. Hence, these surfaces belong to  the claimed strata.
    \qedhere
    \end{proof}

We are now ready for the proof of Proposition~\ref{prop:dense examples}.

    \begin{proof}[Proof of Proposition~\ref{prop:dense examples}]
     For Part~\eqref{item:2n-gon finite}, fix a natural number $n\geq 5$. Let $\G_n$ denote the Veech group of $(M_n,\w_n)$, and let $k(\G_n)$ be its trace field.
    By Lemma~\ref{lem:constant trace field}, every Veech surface $(M',\w')\in \T(\w_n)$ has the same trace field $k(\G_n)$.
    Hence, by  Lemma~\ref{lem:2n-gons degree and hyperelliptic}, the common trace field of Veech surfaces in $\T(\w_n)$ has degree $\geq 3$ over $\Q$.
    Moreover, by~\cite[Corollary 1.6]{EskinFilipWright}, each stratum contains at most finitely many Veech curves with trace field of degree $\geq 3$.
    Thus, $\T(\w_n)$ can only meet finitely many Veech curves in this case.
    The same argument holds for any Veech surface $(M,\w)$ with trace field of degree $\geq 3$.

    By Lemma~\ref{lem:2n-gons degree and hyperelliptic}, for $n=5$, $k(\G_5)$ has degree $2$. Moreover, every Veech surface with trace field of degree $\geq 2$ in the  stratum $\Hcal(1,1)$ is contained in $\SL\cdot (M_5,\w_5)$ \cite{McM,Ca}.
    Thus, the $\SL \cdot (M_5,\w_5)$ is the unique Veech curve intersecting $\T(\w_5)$ in this case as well.

    For Part~\eqref{item:M-primitive examples}, by Lemma~\ref{lem:2n-gons degree and hyperelliptic}, for $n>5$, $(M_n,\w_n)$ is contained in a hyperelliptic component of a stratum of translation surfaces in genus $>2$.
    By work of Apisa,~\cite{Apisa-HighRankInHyp,Apisa-Rk1InHyp}, the only proper $\SL$-orbit closures in such components are either Veech curves, or loci of branched coverings. 
    Since the trace field of $\G_n$ is irrational by Lemma~\ref{lem:2n-gons degree and hyperelliptic}, $\G_n$ is non-arithmetic, i.e., $\G_n$ is not commensurable with any conjugate of any finite index subgroup of $\mrm{SL}_2(\Z)$. 
    Hence, by~\cite[Corollary 1.5]{Wright-VeechWardBouwMollerExamples}, this implies\footnote{The cited result concerns the so-called Bouw-M\"oller family of Veech surfaces, which includes the surfaces $(M_n,\w_n)$ considered here; cf.~\cite{Ward-BilliardsRationalTriangles,BouwMoller,Hooper-BouwMoller,Wright-VeechWardBouwMollerExamples} for more on this larger family of examples.} that $(M_n,\w_n)$ is geometrically primitive, i.e., it cannot arise as a branched cover of a lower genus surface. It follows that $(M_n,\w_n)$ cannot be contained in a locus of covers, and in particular, that $(M_n,\w_n)$ is $\Mcal_n$-primitive in this case.
    \qedhere
    \end{proof}

\subsection{Examples for Theorem~\ref{thm:full support}}
\label{sec:eg full support}

In~\cite[Section 3]{MatYoc}, Matheus and Yoccoz constructed an infinite family of square-tiled surfaces parametrized by odd integers $m\geq 3$.
In this section, we recall the computations in \textit{loc.~cit.} to show:
\begin{prop}[{\hspace{.1pt}\cite{MatYoc}}]
\label{prop:Matheus-Yoccoz examples}
    The Matheus-Yoccoz family of square-tiled surfaces, parametrized by the odd integers, satisfy the hypothesis of Theorem~\ref{thm:full support}.
    More precisely, for each surface in this family, the affine group admits a pseudo-Anosov element acting trivially on the entire symplectic complement of the tautological plane.
    Moreover, for $m\geq 5$, the (infinite) subgroup generated by such elements has infinite index in the affine group  of the corresponding surface.
\end{prop}

This proposition follows directly from the computations in~\cite[Section 3]{MatYoc}. We briefly recall their results for the reader's convenience.
For $m=3$, the resulting surface is the \textit{Ornithorynque} studied in~\cite{ForniMatheus-Orni}, for which the hypothesis of Theorem~\ref{thm:full support} is known to hold for a finite index subgroup of the affine group.
We thus restrict our attention to the case $m\geq 5$.
In what follows, we fix an odd integer $m\geq 5$, and let $(M_m,\w_m)$ be the corresponding (horizontally periodic) Veech surface defined in~\cite[Section 3.1]{MatYoc}. 
Let $\G_m$ denote the affine group of $(M_m,\w_m)$.

In~\cite[Section 3.6]{MatYoc}, it is shown that the homology group of $M_m$ admits the following $\G_m$-invariant decomposition:
\begin{align*}
    H_1(M_m, \Sigma (\w_m); \Q) &= H_1(M_m, \Q) \oplus H_{rel}
    \nonumber\\
    H_1(M_m, \Q) &= H_1^{st} \oplus H_\t \oplus \Breve{H},
\end{align*}
for certain subspaces $H_{rel}, H_\t$ and $\Breve{H}$, and where $H_1^{st}$ is the two-dimensional space that is dual to the tautological plane. 
More precisely, $H_1^{st}$ is the annihilator of the symplectic orthogonal complement of the tautological plane in cohomology.
In particular, since the space $\cylspace^0(\w_m)$ is symplectic orthogonal to the tautological plane in $H^1_\R$, it suffices to exhibit non-trivial pseudo-Anosov elements of $\G_m$ acting trivially on $H_1^{(0)} \stackrel{\mrm{def}}{=} H_\t \oplus \Breve{H}\oplus H_{rel}$.

It is shown in~\cite[Sections 3.4 and 3.5]{MatYoc} that the action of $\G_m$ on $H_\t \oplus H_{rel}$ factors through a finite group.
Moreover, in~\cite[Section 3.6]{MatYoc}, it is shown that the complexified space $\Breve{H}\otimes \C$ splits as a $\G_m$-invariant direct sum $\oplus_\rho \Breve{H}(\rho)$ of two-dimensional spaces, parametrized by non-trivial $m$-roots of unity.
In~\cite[Section 3.1]{MatYoc}, it is shown that the elements
\begin{align*}
    T^2 = \begin{pmatrix}
        1 & 2 \\ 0 & 1
    \end{pmatrix},
    \qquad S^2 = \begin{pmatrix}
        1 & 0 \\ 2 & 1
    \end{pmatrix}
\end{align*}
belong to the Veech group $\mrm{SL}(M_m,\w_m)$.
Let $\widetilde{S}^2$ and $\widetilde{T}^2$ denote elements $\G_m$ with derivative given by $S^2$ and $T^2$ respectively.
On $\Breve{H}(\rho)$, these affine homeomorphisms act with the following matrices
\begin{align*}
    \widetilde{T}^2 \mapsto \begin{pmatrix}
        \rho & 1+\rho \\ 0 & 1
    \end{pmatrix},
    \qquad 
    \widetilde{S}^2 \mapsto \begin{pmatrix}
        1 & 0 \\ 1+\rho^{-1} & \rho^{-1}
    \end{pmatrix}
\end{align*}
in a certain distinguished basis. A straightforward computation thus shows that $A_m :=(\widetilde{T}^2)^m$ and $B_m:=(\widetilde{S}^2)^m$ act trivially on all of $\Breve{H}$.

On the other hand, $A_m$ and $B_m$ are themselves non-trivial (since they map to the non-trivial unipotent  matrices $(T^2)^m$ and $(S^2)^m$ in $\SL$). Moreover, by examining the trace of the image of $A_mB_m$ in the Veech group, one obtains a pseudo-Anosov element acting trivially on $\Breve{H}$ (in fact, $(T^2)^m$ and $(S^2)^m$ generate a Zariski-dense subgroup of $\SL$, and thus contains many such pseudo-Anosovs).
Since the action on $H_\t\oplus H_{rel}$ factors through a finite group, one obtains the desired pseudo-Anosov element by taking a suitable power of $A_mB_m$.

Finally, as noted in~\cite[Remark 3.2]{MatYoc}, the element $\widetilde{S}^2\widetilde{T}^2$ acts on $\Breve{H}(\exp(2\pi i /m))$ via a hyperbolic matrix infinite order for all odd $m\geq 5$, and hence the full action of $\G_m$ on $\Breve{H}$ cannot factor through a finite group $m>3$. In particular, by a result of M\"oller~\cite{Moller-ShimuraAndTeichCurves}, the complement of the tautological plane admits at least one positive exponent in these cases.


\section{Uniform Convergence along Full Banach Density of Times}
\label{sec:full Banach density}

The goal of this section is to prove Theorem~\ref{thm:uniform Banach density one} on uniform convergence of expanding horocycle arcs along a full Banach density of times.
It is the key tool in the deduction of Theorem~\ref{thm:density of tori} from the key matching proposition in \textsection\ref{sec:reduce to matching}.
We retain the notation of Theorem~\ref{thm:uniform Banach density one} throughout this section.

\subsection{Results from Eskin-Mirzakhani-Mohammadi}
In this section we recall some results of Eskin-Mirzakhani-Mohammadi \cite{EMM}.
The first result follows from \cite[Theorem 2.7]{EMM} and the fact that $g_t$ pushes of $r_\theta$ arcs fellow travel with $g_t$ pushes of horocycle arcs.

\begin{thm}[{\cite[Theorem 2.7]{EMM}}]\label{thm:EMM 2.7} Let $\phi\in C_c(\Mcal)$, and 
$\epsilon>0$. There exists a finite set of invariant manifolds, $\Ncal_1,\dots,\Ncal_n \subset \Mcal$ so that for any compact set  
$\mathcal{C} \subset \Mcal \setminus \cup_{i=1}^n\Ncal_i$, there exists $T_1>0$ so that for all $\omega \in \mathcal{C}$  and  for all $T\geq T_1$, 
$$\frac 1 {T} \int_0^T\int_{0}^1\phi(g_tu(s)\omega)dsdt>\int_{\Mcal} \phi d\mu_{\Mcal}-\epsilon.$$
\end{thm}

\begin{prop}[{\cite[Proposition 2.13]{EMM}}] \label{prop:EMM 2.13} Let $\Ncal\subset \HHu$ be an affine invariant submanifold. (In this proposition $\Ncal=\emptyset$ is allowed.) Then there exists an $\mrm{SO}(2)$-invariant function 
$f_{\Ncal}: \HHu \to [1,\infty]$  with the following properties:
 \begin{enumerate}
 \item $f_{\Ncal}(\omega)=\infty$ if and only if $\omega\in \Ncal$, and $f_{\Ncal}$ is bounded on compact subsets of $\HHu\setminus \Ncal$. For any $\rho>0$, the set 
 $\overline{\set{\omega : f_{\Ncal}(\omega)\leq \rho}}$ is a compact subset of 
 $\HHu\setminus \Ncal$.
\item There exists $b > 0$ (depending on $\Ncal$) and for every $0 < c < 1$ there exists $t_0 >0$ (depending on $\Ncal$ and $c$) such that for all $\omega\in \HHu\setminus \Ncal$ 
and all $t>t_0$, 
$$\frac 1 {2\pi} \int_0^{2\pi} f_{\Ncal}(g_tr_{\theta}\omega)d\theta \leq cf_{\Ncal}(\omega) + b.$$
\item  There exists $\sigma > 1$  and $V\subset \SL$ a neighborhood of the identity so that 
for all $g\in V$ and all $\omega \in \HHu$,
$$\sigma^{-1}f_{\Ncal}(\omega) \leq  f_{\Ncal}(g\omega) \leq \sigma f_{\Ncal}(\omega).$$
\end{enumerate}
\end{prop}
The next result is a straightforward modification of  \cite[Lemma 3.5]{AAEKMU}, where we average from $0$ to $1$ instead of $-1$ to $1$.
\begin{lem}\label{lem:2.13 for hor}\cite[Lemma 3.5]{AAEKMU} Let $f_{\Ncal}$ be as in Proposition \ref{prop:EMM 2.13}. Then there exists a constant $b'>0$ so that for all $0<a<1$ there exists $\bar{t}_0= \bar{t}_0(a)$ such that for all $t> \bar{t}_0$ and for all $\omega \in \HHu\setminus \Ncal$ we have 
$$\int_{0}^1 f_{\Ncal}(g_tu(s) \omega)ds <  a f_{\Ncal}(\omega)+  b'\,.$$
\end{lem}

\begin{lem}[{\cite[Proposition~3.6]{EMM}}]
\label{lem:marg big} Let $f_{\Ncal}$ be a function as in Proposition \ref{prop:EMM 2.13}. If $\epsilon>0$ there exists $N$ so that for all $x \notin \Ncal$ there exists $S_0$ so that for all $t>S_0$ we have 
$$\left|\set{s\in [0,1]: f_{\Ncal}(g_tu(s)\omega)>N}\right|<\epsilon.$$ 
Moreover, $S_0$ can be chosen to depend only on $f_{\Ncal}(\omega)$.
\end{lem}

\subsection{Consequences}

In this section we develop some consequences of the previous results via fairly straightforward arguments.
Let $d_\ast(\mu,\nu)$ be any metric giving the weak-$\ast$ topology on Borel measures so that $\mu(\Mcal)\leq 1$.

\begin{prop}\label{prop:func unif}For every $\epsilon>0$ there exists $T\geq 0$,  
$0\leq \phi\leq 1$ with $\phi \in C_c(\Mcal)$ so that $\int_{\Mcal}\phi d\mu>1-\epsilon$ and for any $x \in \Mcal$ so that $\phi(x)>0$ we have $d_\ast(T^{-1}\int_0^T\delta_{u(s)x}ds,\mu_{\Mcal})<\epsilon$. 
\end{prop}
\begin{proof}
\textit{Step 1:} Finding a compact set $\mathcal{K}$, an open set $U \supset \mathcal{K}$ and $T>0$ so that $\mu_{\Mcal}(\mathcal{K})>1-\epsilon$ and $$
d_\ast(T^{-1}\int_0^T\delta_{u(s)x}ds,\mu_\Mcal)<\frac{\epsilon}2
$$ for all $x \in U$.

There exists a finite set of continuous compactly supported function $F=\set{f_i}_{i=1}^n$ and $\delta>0$ so that if $\nu$ is a Borel probability measure  
$$\left|\int f_i d\mu_{\Mcal}-\int f_i d\nu\right|<\delta  \text{ for all }i$$ then $d_\ast(\mu,\nu)<\frac{\epsilon}2$. 

Applying the Birkhoff (or Von Neumann) Ergodic Theorem 
$n$ times we obtain $T$ and a measurable set $G$ with $\mu_{\Mcal}(G)>1-\frac \epsilon 4 $ and 
$$ \left|\int f_i(u(s)z)-\int f_i d\mu_{\Mcal}\right|<\delta$$ for all $z \in G$ and $f_i \in F$. By the uniform continuity of the $f_i$ and $u(s)$, we may assume $G$ is open. By inner regularity of measures, there exists $\mathcal{K}\subset U$ with $\mu_{\Mcal}(\mathcal{K})>1-\frac \epsilon 2$.

\textit{Step 2:} Completion. Recall that in any locally compact Hausdorff space, for any $\mathcal{K} \subset U$ with $\mathcal{K}$ compact and $U$ open, there exists $\phi \in C_c$ so that $\phi|_{\mathcal{K}}=1$, $\phi|_{U^c}=0$ and $0\leq \phi(x)\leq 1$ for all $x$. Our $\phi$ is such a $\phi$ for $\mathcal{K}$ and $U$ as in the previous step. Indeed, $\int \phi d\mu_{\Mcal}\geq \mu_{\Mcal}(\mathcal{K})>1-\epsilon.$

\end{proof}

\begin{cor}\label{cor:weak close}

Let 
$\epsilon>0$.
There exist $T_0>0$ and proper $\SL$-orbit closures $\Ncal_1, \dots, \Ncal_n$ in $\Mcal$ such that for any compact set $F\subset \Mcal\setminus \cup_{i=1}^k \Ncal_i$, we can find $S_0\geq 0$, so that for all $T\geq T_0$,  $S\geq S_0$ and $x\in F$, we have
\begin{equation}
\label{eq:banach density}
\bigg|\set{\ell \in [S,S+T]: d_\ast\big( \int_{0}^1 \delta_{g_\ell u(s)x}ds,\mu_{\Mcal}\big)<\epsilon }\bigg|>(1-\epsilon)T.
\end{equation}
\end{cor}

\begin{proof}
Let $\delta>0$ so that whenever $t \mapsto \nu_t$ is a measurable assignment of measures, with $$|\set{t\in [0,1]: d_\ast(\nu_t,\mu_{\Mcal})>\delta}|<\delta,$$
we have $$d_\ast(\int_{[0,1]}\nu_t,\mu_{\Mcal})<\epsilon.$$ 
Let $\phi$ be as in Proposition \ref{prop:func unif} applied with $\epsilon \delta/12$ in place of $\epsilon$. 
Applying Theorem \ref{thm:EMM 2.7} with  $\epsilon \delta/12$ in place of $\epsilon$, we obtain a finite number of closed $\SL$-invariant manifolds $\Ncal_1,...,\Ncal_n$ satisfying the conclusion of the theorem.
For each of these, we obtain $f_i\stackrel{\mrm{def}}{=}f_{\Ncal_i}$ as in Proposition \ref{prop:EMM 2.13}. We now apply Lemma \ref{lem:marg big} to the $f_i$ with $\epsilon \delta/12n$ in place of $\epsilon$, a compact set $F\subset \Mcal\setminus \bigcup_{i=1}^n \Ncal_i$, and obtain $N_i, S_i$ for each $f_i$. The conclusion of the lemma holds for all $t>S_i$ and $x\in F$ by the last assertion of the lemma.

Let $\hat{N}=\max\set{N_i:i}$ and $S_1=\max\set{S_i:i}$. 
Let $\mathcal{K} =\cap f_i^{-1}([0,\hat{N}])$. Now, let $T_1$ be as in Theorem \ref{thm:EMM 2.7} applied with $\phi$, $\mathcal{K}$ and $\epsilon \delta/12$.  
Then, if $T\geq T_1$, $S\geq S_1$ and $x \in F$, 

$$\int_S^{S+T}\int_0^1\phi\big(g_\ell u(s)x\big)dsd\ell\geq \int_0^1\bigg(\int_S^{S+T}  \int_0^1\phi(   g_\ell u(r) u(s)g_Sx) drd\ell\bigg) ds  -2e^{-2S}.$$
 Now if $u(s)g_Sx \in \mathcal{K}$ by Theorem \ref{thm:EMM 2.7} we have 
$$\int_S^{S+T}\int_0^1\phi\big(g_\ell u(r)u(s)g_Sx\big)drd\ell>T(\int \phi-\frac{\epsilon \delta}{12})>T-2 \cdot \frac{\epsilon \delta }{12}.$$ Also, by Lemma \ref{lem:marg big} 
$$|\set{s \in [0,e^{2S}]:u(s)g_S x\notin \mathcal{K}}|<n \frac{\epsilon \delta}{12 n}. $$ So, if $S_0\geq S_1$ is big enough then for all $S\geq S_0$ and $T\geq T_1$,
 
$$\int_S^{S+T}\int_0^1\phi\big(g_\ell u(s)x\big)dsd\ell>T-\frac{\epsilon \delta }2.
$$ 
Since $0\leq \phi\leq 1$ we have that $$\bigg|\set{\ell\in [S,S+T]: \int_{0}^1\phi\big(g_\ell u(s)x\big)ds>1-\frac{\delta}2}\bigg|>1-{\epsilon}.$$ 
For each $\ell$ so that $|\set{s\in [0,1]:\phi(s)>0}|>1-\frac{\delta}2$ we have that if $L$ is the $T$ in Proposition \ref{prop:func unif} $$\bigg|\set{s \in [0,e^{2\ell}]: d_\ast(L^{-1}\int_0^L\delta_{u(r)u(s)g_\ell x}dr,\mu_{\Mcal})<\delta}\bigg|<\frac{\delta}2e^{2\ell}.$$ 
Thus if $\ell$ is large enough, $$\bigg|\set{s \in [0,e^{2\ell}-L]: d_\ast(L^{-1}\int_0^L\delta_{u(r)u(s)g_\ell x}dr,\mu_{\Mcal})<\delta}\bigg|<\delta e^{2\ell}.$$ Thus by the choice of $\delta$ at the start of the proof, 
$$d_\ast(\int_0^{1}\delta_{g_\ell u(s)x},\mu_{\Mcal})<\epsilon.$$
The proof is completed by observing that because  $\phi(x)\leq 1$ for all $x$, 
$$ \int_0^1\phi(g_\ell u(s)xdx>1-\frac \delta 2 \,\, \Longrightarrow \, \, \bigg|\set{s\in [0,e^{2\ell}]: \phi(s)>0}\bigg|>(1-\epsilon)e^{2\ell}.$$
\end{proof}

\subsection{Proof of Theorem~\ref{thm:uniform Banach density one}}
In this section, we use the above results to complete the proof of Theorem~\ref{thm:uniform Banach density one}.
First, given $\epsilon$, by rescaling $d_\ast$, we may assume that 
\begin{equation}\label{eq:meas close func close}
d_\ast(\mu,\nu)<\epsilon \quad \Longrightarrow \quad \left|\int f d\mu-\int f d \nu\right|<\epsilon. 
\end{equation}
Next, because $\mu_{\Mcal}$ is $g_t$ invariant and $(g_t)_\ast$ acts continuously with respect to the weak-$\ast$ topology, which when restricted to measures of total variation at most 1 is compact, there exists $\delta>0$ so that if $d_\ast(\nu,\mu_{\Mcal})<\delta$ we have 
\begin{equation}\label{eq:push close}\max_{s\in [-1,1]}\, d_\ast(g_{s}\nu,\mu_{\Mcal})<\epsilon.
\end{equation} 
Now, applying Corollary \ref{cor:weak close} with $\epsilon=\frac \delta 4$ we have $T_0$ and 
closed $\SL$-invariant sets $\Ncal_1,...,\Ncal_n$ so that for all compact $F \subset \Mcal\setminus \cup_{i=1}^n\Ncal_i$, there exists $S_0$ so that for all $x \in F$, $S\geq S_0$ and $T\geq T_0$ we have 
$$\left|\set{\ell \in [S,S+T]: d_\ast( \int_{0}^1 \delta_{g_\ell u(s)x}ds,\mu_{\Mcal})<\delta}\right|>(1-\epsilon)T.$$
In particular, the 1-neighborhood of this set contains at least $(1-\epsilon)T$ integers giving, 
$$\#\set{\ell \in [S,S+T] \cap \mathbb{N}: d_\ast(\int_{0}^1 \delta_{g_\ell u(s)x}ds,\mu_{\Mcal})<\epsilon}>(1-\epsilon)T.$$
Letting $L_0=T_0$, this establishes Theorem \ref{thm:uniform Banach density one}.

\section{The Key Matching Proposition and Proof of Theorem~\ref{thm:density of tori}}
\label{sec:reduce to matching}

The goal of this section is to reduce Theorem~\ref{thm:density of tori} to Proposition~\ref{prop:key matching} below.
Roughly, this result asserts that, when $t$ is large, the pushed twist torus $g_t \cdot \T(\w)$ will be close to a whole family of pushes of pieces of the twist torus at all times between $t-L_0$ and $t$, for any given $L_0>0$. 
The reduction relies on a refinement of the equidistribution theorems of Eskin, Mirzakhani, and Mohammadi, Theorem~\ref{thm:uniform Banach density one}, proved in \textsection\ref{sec:full Banach density}.
The proof of Proposition~\ref{prop:key matching} is given in \textsection\ref{sec:proof of key matching}.

Recall the notation in \textsection\ref{sec:notation}.

\begin{prop}[Key Matching Proposition]
    \label{prop:key matching}
        Let $(M,\w)$ be a horizontally periodic Veech surface, and let $\Vcal = \SL\cdot (M,\w)$ be its $\SL$-orbit.
        Let $0\neq \b\in\twist^0(\w)$.
        For every $\e>0$, there exist a compact subset $\Kcal\subset \Vcal$ and $\d>0$, so that the following hold for every $ L_0\geq 0, T\geq 1$, and for all large enough $t>0$.
        For every $0\leq \ell\leq L_0$, there is a set $S_\ell\subseteq [0,1]$ of measure at least $1-\e$ such that for all $s\in S_\ell$, we have
        \begin{enumerate}
            \item\label{item:key match in cpt}  $\w(t-T-L_0+\ell,s)\in \Kcal$,
            and
            \item\label{item:key match closeness} for all $0\leq r<\d/N_\beta(t-T-L_0+\ell,s)$, we have
            \begin{align*}
            g_{T+\ell} \cdot \mrm{Trem}_\b(t-T-L_0,s,r)
            \in  B\left(g_t \cdot \T(\w,\b), \e \right).
        \end{align*}
        \end{enumerate}
        Here, for a subset $E\subseteq \HHu$, $B(E,\e)$ denotes its open $\e$-neighborhood in the AGY metric.
\end{prop}

\begin{rem}
\begin{enumerate}
    \item Proposition~\ref{prop:key matching} holds in general for all horizontally periodic Veech surfaces $(M,\w)$, and does not require the $\Mcal$-primitivity hypothesis.

    \item In our proof of Theorem~\ref{thm:density of tori}, we use the full strength of item~\eqref{item:key match closeness}, but we only apply item~\eqref{item:key match in cpt} for $\ell=0$.
\end{enumerate}

\end{rem}

In the rest of this section, we deduce Theorem~\ref{thm:density of tori} from Proposition~\ref{prop:key matching}.
Our goal is to build a suitable compact set $F$ of almost generic points in the sense of Theorem~\ref{thm:uniform Banach density one}.
We will apply Proposition~\ref{prop:key matching} to tremors that land in $F$ to show, roughly speaking, that expanding horocycle arcs starting from these points remain close to the expanded torus $g_t\cdot \T(\w)$ for a long interval of time.
This will imply that a neighborhood of $g_t\cdot\T(\w)$ contains a large piece of the $P$-orbit of an almost generic point, where $P=AU$ is the subgroup of upper triangular matrices.
The set $F$ will be defined in equations~\eqref{eq:Rcal_delta0} and~\eqref{eq:F}. Its construction requires preparation that occupies the next two subsections.

\subsection{The role of convergence along full Banach density set of times}
\label{sec:applying Banach density}

Let 
$$\Ucal\subseteq \Mcal\stackrel{\mrm{def}}{=}\overline{\SL\cdot \T(\w) }$$
be an arbitrary non-empty open ball of radius $r$.
Assume that $(M,\w)$ is $\Mcal$-primitive.
We will show that there is $\b\in \twist^0(\w)$ so that for all $t$ large enough,
\begin{align}\label{eq:intersect 2Ucal}
    g_t\cdot\T(\w,\b) \cap 2\Ucal \neq\emptyset,
\end{align}
where $2\Ucal$ is the ball with the same center and twice the radius as $\Ucal$.
Since $\Ucal$ is arbitrary, this will conclude the proof of Theorem~\ref{thm:density of tori}.

The following lemma enables us to apply Theorem~\ref{thm:uniform Banach density one} by showing that $\Mcal$ is the orbit closure of a single point.
\begin{lem}\label{lem:Mcal is ergodic}
    There exists $x\in \T(\w)$ such that $\Mcal = \overline{\SL\cdot x}$. In particular, $\Mcal$ is the support of a unique $\SL$-invariant and ergodic probability measure.
\end{lem}
\begin{proof}
    
    For each $y\in \T(\w)$, let $\Ncal_y = \overline{\SL\cdot y}$ denote the $\SL$-orbit closure of $y$. Let $\Ccal = \set{\Ncal_y: y\in \T(\w)}$, and denote by $\Ccal^\wedge$ the subset of maximal elements of $\Ccal$ with respect to inclusion. 
    By~\cite{Wright-FieldofDef} and~\cite[Proposition 2.16]{EMM}, $\Ccal$ is a countable set, and hence so is $\Ccal^\wedge$.
    In particular, $\T(\w) = \sqcup_{\Ncal\in \Ccal^\wedge} \Ncal \cap \T(\w)$ is a decomposition of $\T(\w)$ as a countable disjoint union of closed sets.
By a result of Sierpinski~\cite{sierpinski1918theoreme} (cf.\cite[Theorem 6.1.27]{Engelking-GeneralTop}),  since $\T(\w)$ is compact and connected, $\Ccal^\wedge$ must consist of a single element, completing the proof of the first assertion. The second assertion follows from the first by~\cite{EMM}.
\end{proof}

Denote by $\mu$ the $\SL$-invariant ergodic probability measure on $\Mcal$; the existence of which follows by~\cite{EMM} and Lemma~\ref{lem:Mcal is ergodic}.
Let $f$ be a compactly supported continuous function such that $0\leq f\leq \mathbf{1}_{\Ucal}$ and $\eta\stackrel{\mrm{def}}{=}\int f\;d\mu >0$.

We apply Proposition~\ref{prop:key matching} with
\begin{equation}
    \label{eq:eps for key matching}
    \e=\min\set{10^{-5}\eta^2,\mrm{radius}(\Ucal)/2}
\end{equation}
to get a compact set $\Kcal \subset \Vcal$ and $\d>0$ satisfying its conclusion.
By Theorem~\ref{thm:uniform Banach density one}, applied with $f$ as above, there are finitely many $\SL$-orbit closures $\Ncal_1,\dots,\Ncal_k$ inside $\Mcal$ and $L_0>0$, such that given any compact set $F\subset \Mcal \setminus \cup_{i=1}^k\Ncal_i$, we can find $T_0>0$ so that 
\begin{align}\label{eq:applying Banach density one}
    \# \set{\ell \in [T,T+L]\cap\N : \int_0^1 f(g_\ell u(s) x)\;ds  >9\eta/10 }  >(1-\eta/10)L,
    \qquad \forall x\in F, T\geq T_0, L\geq L_0.
\end{align}

Let $x\in \T(\w)$ be as in Lemma~\ref{lem:Mcal is ergodic} and let $\b\in \twist(\w)$ be a unit norm cylinder twist such that 
\begin{align}\label{eq:choice of x}
    x=\Trem(\w,r\b), \quad \text{for some } r\in \R.    
\end{align}
Since $\twist(\w)$ is spanned by the direction tangent to the $U$-orbit of $\w$ together with the subspace $\twist^0(\w)$, cf.~\eqref{eq:twist=horocycle plus balanced}, after replacing $x$ with $u(s)x$ for a suitable $s\in\R$, we shall assume that
\begin{equation}\label{eq:choice of beta}
    \b\in \twist^0(\w).
\end{equation}

\subsection{\texorpdfstring{$\Mcal$-primitivity}{Primitivity} and avoidance of exceptional orbit closures\texorpdfstring{: the set $\Rcal_\d$}{}}
\label{sec:R_delta}
This subsection is the only place in the proof of Theorem~\ref{thm:density of tori} where our $\Mcal$-primitivity hypothesis is used.
In Appendix~\ref{sec:decagon}, we show how to carry out this part of the argument in the case of the decagon, where this hypothesis fails to hold.

Given $\d_0>0$, let
\begin{align}\label{eq:Rcal_delta0}
    \Rcal_{\d_0} \stackrel{\mrm{def}}{=} \set{ \Trem_\b(t,s,r) : t\geq 0, s\in [0,1], \w(t,s) \in \Kcal, \d_0/2 <rN(t,s) < \d_0}.
\end{align}

\begin{lem}\label{lem:dist to exceptional orbit closures}
    For all sufficiently small $\d_0>0$, we have that the infimum distance between $\Rcal_{\d_0}$ and the exceptional orbit closures is positive, i.e.,
    \begin{align*}
    \inf \set{ \dAGY(y, \Ncal_i) : y\in \Rcal_{\d_0}, 1\leq i\leq k } >0.
\end{align*}    
\end{lem}

\begin{proof}
Let $    C(\d_0) = \sup \set{\dAGY(y,\Kcal): y\in \Rcal_{\d_0} }$ and define $\d_1$ as follows: 
\begin{align*}
    \d_1 \stackrel{\mrm{def}}{=}\min \set{\dAGY(\Kcal,\Ncal_i): 1\leq i\leq k, \Vcal \not\subseteq \Ncal_i} >0.
\end{align*}
Taking $\d_0$ small enough, we can ensure that $C(\d_0)<\d_1/2$\footnote{Indeed, using~\cite[Lemma 3.3]{ChaikaKhalilSmillie} which relates distances locally to norms on the tangent space, one can show that $\dAGY(x,\Kcal)\ll \d$.}.
This ensures positivity of the distance to all $\Ncal_i$ not containing $\Vcal$.

Now, let $\Ncal_i$ be such that $\Vcal\subseteq \Ncal_i$.
Since $(M,\w)$ is $\Mcal$-primitive, this implies that $\Ncal_i=\Vcal$.
Recall the notation set in Convention \ref{torsion free convention}.
Let $\Gamma=\mrm{SL}(M,\w)$ be the Veech group of $\w$.
For every $g\in\SL$, let $\cylspace^0(g\w)$ be the image of $\cylspace^0(\w)$ in the tangent space of $\Mcal$ at $g\w$ under the derivative $D_\w g:T_\w\Mcal\r T_{g\w}\Mcal$ of left multiplication by $g$. The spaces $\cylspace^0(g\w)$ are well-defined and depend only on the point $g\w$ in $\Vcal$, and not on the choice of $g$ in view of $\mrm{Aff}^+(M,\w)$-invariance of $\cylspace^0(\w)$.
This defines a vector bundle over $\Vcal$ with fibers the spaces $\cylspace^0(-)$.
By definition, we have that $\b(t,s)$ belongs to $\cylspace^0(\w(t,s))$.

Recall that the tangent space $T_\w\Vcal$ is the complexification $\mrm{Taut}_\w\otimes \C$ of the tautological plane at $\w$.
Hence, since $\cylspace^0(\w)$ is a subspace of the balanced space at $\w$ (cf.~\textsection\ref{sec:balanced spaces} and \ref{sec:twists}), $T_\w\Vcal$ has trivial intersection with $\cylspace^0(\w)$.
Thus, $\SL$-invariance of the bundles $T\Vcal$ and $\cylspace^0(-)$ implies that their respective fibers have trivial intersection at every point in $ \Vcal$.
Since fibers of these two bundles vary continuously and $\Kcal$ is a compact subset of $\Vcal$, this provides positivity of the infimum over $q\in\Kcal$ of the AGY distance between the unit norm spheres in $T_q\Vcal$ and $\cylspace^0(q)$.
Hence, taking $\d_0$ sufficiently small, this implies that $\dAGY(q,\Vcal)$ is bounded away from $0$ over all $q\in\Rcal_{\d_0}$.
\qedhere
\end{proof}

Recall the sets $S_\ell$ and the parameter $\d>0$ provided by Proposition~\ref{prop:key matching}.
The following lemma elaborates several useful consequences of that proposition.
\begin{lem}\label{lem:producing sequence of approximating horocycles}
    Assume that $\d_0$ is chosen sufficiently smaller than $\d e^{-2L_0}$.
    Then, for all large $t>0$ and all $1\leq T\leq t-L_0 $, there exist $r>0$ and a subinterval $I\subseteq [0,1]$ of length $1/ \lceil e^{2(t-T-L_0)}\rceil$  such that the following hold:
    \begin{enumerate}
        \item\label{item:good indices} For every $\a\in (0,1)$, $\# \set{ \ell \in [0,L_0]\cap\N: |I\cap S_\ell| <(1-\a)|I| } < \sqrt{\e} L_0/\a$.
        \item\label{item:s_0 in Rcal_delta0}  There exists $s_0\in I$ such that $\Trem(t-T-L_0,s_0,r)\in \Rcal_{\d_0}$.

        \item\label{item:landing near torus} For every $0\leq \ell\leq L_0$ and $s\in I\cap S_\ell$, we have that
        \begin{align*}
            g_{T+\ell} \cdot \Trem(t-T-L_0, s,r) \in B(g_t \cdot \T(\w,\b), \e).
        \end{align*}
    \end{enumerate}
    
\end{lem}
\begin{proof}
    Decompose $[0,1] = \sqcup I$ into a disjoint union of intervals, each of length $1/\lceil  e^{2(t-T-L_0)}\rceil $.
    Since each $S_\ell$ has measure at least $1-\e$, we find that $\sum_I \sum_\ell |I\cap S_\ell|=\int_0^1 \sum_\ell \mathbf{1}_{S_\ell}(s)\;ds \geq (1-\e)L_0$.
    Letting $B$ denote the subset of intervals $I$ with $\sum_\ell |I\cap S_\ell| < (1-\sqrt{\e})L_0|I|$, we see that
    \begin{align*}
        (1-\e)L_0
        &\leq (1-\sqrt{\e})L_0 \sum_{I\in B}|I| 
        +L_0 \sum_{I\notin B}  |I|.
    \end{align*}
    Hence, we find that $\sum_{I\in B}|I| <\sqrt{\e}$.
    Since $|S_0|\geq 1-\e$, it follows that we can find an interval $I$ such that $I\cap S_0 \neq \emptyset$ and
     $\sum_\ell |I\cap S_\ell| \geq (1-\sqrt{\e})L_0 |I|$.
    Let $\a \in (0,1)$ and let $m$ denote the cardinality of the set of indices $\ell$ such that $|I\cap S_\ell| <(1-\a) |I|$.
    It follows that $(1-\sqrt{\e})L_0 <  (1-\a) m + L_0-m$, from which we conclude that $m< \sqrt{\e} L_0/\a$.

    For the second assertion, fix some arbitrary $s_0\in I\cap S_0$ and let $r$ be such that $rN(t-T-L_0,s_0)$ belongs to the interval $(\d_0/2,\d_0)$.
    Then, $\w(t-T-L_0,s_0)\in \Kcal$ by definition of $S_0$. In particular, the second assertion follows by definition of $\Rcal_{\d_0}$.

    For the last assertion, by Proposition~\ref{prop:key matching}, it suffices to check for each $\ell$ that $r<\d/N(t-T-L_0+\ell,s)$ for all $s\in I\cap S_\ell$.
    To this end,
    note that the orbits $\set{\w(\t,s):0\leq \t\leq t-T-L_0}$ all remain within distance $O(1)$ in $\Vcal$ from one another as $s$ varies in $I$.
    Hence, using the bound $\norm{\KZ{g}{\cdot}}\ll \norm{g}^{O(1)}$ from Corollary~\ref{cor:KZ norm}, we get that 
    $N(t-T-L_0,s_1) \asymp N(t-T-L_0,s_2)$
    for all $s_1,s_2\in I$.
    Moreover, by Lemma~\ref{lem:equivariance of notation}, we also have that $ N(t-T-L_0+\ell,s)\leq e^{2L_0} N(t-T-L_0,s)$ for all $0\leq \ell \leq L_0$.

    It follows by our choice of $r$ that for every $\ell$, we have
    \begin{align*}
        r N(t-T-L_0+\ell,s) \ll e^{2L_0}\d_0.
    \end{align*}
    Taking $\d_0$ sufficiently smaller than $\d e^{-2L_0}$, this ensures that $r$ is $<\d/N(t-T-L_0+\ell,s)$ for all $s\in I$ and all $0\leq \ell\leq L_0$.
    This implies the last assertion of the lemma in view of Proposition~\ref{prop:key matching}.
\qedhere
\end{proof}

\subsection{Conclusion of the proof of Theorem~\ref{thm:density of tori} assuming Proposition~\ref{prop:key matching}}
\label{sec:end of density proof}

Recall we are fixing an open ball $\Ucal$ and a bump function $f$ with $\supp(f)\subset\Ucal$ and $\int f\;d\mu_\Mcal=\eta$.
Moreover, we have a parameter $\d>0$ provided by Proposition~\ref{prop:key matching} when applied with $\e$ as in~\eqref{eq:eps for key matching}.
Let $0<\d_0<\d e^{-2L_0}$ be a parameter satisfying the conclusion of Lemmas~\ref{lem:dist to exceptional orbit closures} and~\ref{lem:producing sequence of approximating horocycles}. Set
\begin{align}\label{eq:F}
    F = \bigcup_{\t,s\in [-1,1]} g_\t u(s)\cdot \overline{\Rcal_{\d_0}}.
\end{align}
Then, by Lemma~\ref{lem:dist to exceptional orbit closures} and $\SL$-invariance of $\cup_i\Ncal_i$, we have that $F\subset \Mcal \setminus \cup_i \Ncal_i$.

Fix some $T$ such that~\eqref{eq:applying Banach density one} holds for $L=L_0$ and for this $F$.
Let $I$, $s_0$, and $r$ be as provided by Lemma~\ref{lem:producing sequence of approximating horocycles} and let $x_0=\Trem(t-T-L_0,s_0,r)\in \Rcal_{\d_0}$.

Roughly, we wish to apply~\eqref{eq:applying Banach density one} to the horocycle arc $\set{g_{t-T-L_0}\cdot \Trem(u(s)\w,r\b):s\in I}$.
However, as stated,~\eqref{eq:applying Banach density one} technically holds for arcs of length one with left endpoint in $F$. 
This is remedied by adjusting $x_0$ by a suitably small upper triangular matrix. 
First, we replace $s_0$ with the left endpoint $s_1$ of the interval $I$, so we let
\begin{align*}
    x_1 =\Trem(t-T-L_0,s_1,r)
    =u(e^{2(t-T-L_0)} (s_1-s_0))x .
\end{align*}
We also need to find a slight adjustment to the geodesic flow time $t-T-L_0$ so that a suitable horocycle arc parametrized by $I$ becomes length $1$ after flowing. To this end, let $\t\geq 0$ be such that 
\begin{align*}
    I_1 = e^{2(t-T-L_0+\t)}(I-s_1)  =[0,1].
\end{align*}
Note that
\begin{align}\label{eq:tau goes to 0}
    \t\r0 \text{ as } t\r \infty    
\end{align}
since $|I|=1/\lceil e^{2(t-T-L_0)}\rceil$.
Let $x_2=g_{\t} x_1$.
Then, since $|I|\leq e^{-2(t-T-L_0)}$ and $x_0\in\Rcal_{\d_0}$, we have that $x_1\in \cup_{\s\in [-1,1]}u(\s)\cdot\Rcal_{\d_0}$.
It follows by definition of $F$ that $x_2 =g_\t x_1\in F$.

Thus, we can finally apply~\eqref{eq:applying Banach density one} with $x_2$ in place of $x$ to get
\begin{align}\label{eq:Banach holds}
    \int_{I_1} \mathbf{1}_\Ucal(g_{T+\ell} u(s) x_2)\;ds
    \geq 
    \int_{I_1} f(g_{T+\ell} u(s) x_2)\;ds  >9\eta/10
\end{align}
for a set of indices $\ell\in [0,L_0]$ of cardinality $>(1-\eta/10)L_0$.
Next, applying Lemma~\ref{lem:producing sequence of approximating horocycles}\eqref{item:good indices} with $\a=\eta/20$, we have that
\begin{align}\label{eq:intersect S_ell}
    |I\cap S_\ell|\geq (1-\eta/20)|I|
\end{align}
for a set of indices $\ell\in [0,L_0]$ of cardinality at $\geq 1-\sqrt{\e}/\a$.
Hence, our choices of $\e$ and $\a$ ensure that we can find $\ell$ so that~\eqref{eq:Banach holds} and~\eqref{eq:intersect S_ell} hold simultaneously.
In particular, we can find $s'\in I\cap S_\ell$ such that for $s=e^{2(t-T-L_0+\t)}(s'-s_1)$, we have
\begin{align*}
    g_{T+\ell} u(s) x_2\in \Ucal,
    \qquad 
    \text{and}
    \qquad 
    g_{T+\ell} \cdot \Trem(t-T-L_0,s',r)\in B(g_t\cdot \T(\w,\b), \e).
\end{align*}
Moreover, it follows from the definitions that $g_{T+\ell}u(s)x_2=g_{\t+T+\ell}\cdot \Trem(t-T-L_0,s',r)$.
In particular, we obtain
\begin{align*}
    B(g_t\cdot \T(\w,\b), \e) \cap g_{-\t}\cdot \Ucal \neq \emptyset.
\end{align*}

Thus, in view of~\eqref{eq:tau goes to 0} and our choice of $\e$ in~\eqref{eq:eps for key matching}, it follows that $g_t\cdot \T(\w,\b)$ intersects the ball $2\Ucal$ with twice the radius and same center as $\Ucal$. This verifies~\eqref{eq:intersect 2Ucal} and concludes the proof of Theorem~\ref{thm:density of tori}.

\section{A-invariance of Limiting Distributions of Output Directions}
\label{sec:A-invariance}

The goal of this section is prove Theorem~\ref{thm:A-invariance of flag distribution}.
We keep the same notation of the theorem throughout this section.

\subsection{Proof of Theorem~\ref{thm:A-invariance of flag distribution}}

First, we quickly reduce the second assertion of the theorem to the first.
Fix $z\in \P\hV$ and let $\hat{\nu}$ be an arbitrary weak-$\ast$ limit measure of the measures $\int_0^1\d_{g_tu(s)\cdot z}\;ds$ along a sequence of $t_n\to\infty$.
First, we note that, using the identity $u(r)g_t = g_t u(e^{-2t}r)$, it is easy to see that $\hat{\nu}$ is $U$-invariant, where $U=\set{u(r):r\in\R}$.
Moreover, by equidistribution of expanding horocycle arcs on $\Vcal$, we have that $\hat{\nu}$ projects to $\mu_\Vcal$ on $\Vcal$.
In particular, $\hat{\nu}$ is a probability measure.

\begin{lem}
    Almost every $U$-ergodic component of $\hat\nu$ projects to $\mu_\Vcal$ on $\Vcal$.
\end{lem}
\begin{proof}
    Let $\hat\nu =\int \hat\nu_x\;d\l(x)$ be an ergodic decomposition of $\hat\nu$ and let $\pi:\widehat\Vcal\r \Vcal$ denote the standard projection.
    Then, $\nu_x\stackrel{\mrm{def}}{=} \pi_\ast\hat\nu_x$ is a $U$-ergodic measure.
    Since $\pi_\ast\hat\nu=\mu_\Vcal$, it follows that $\mu_\Vcal=\int \nu_x\;d\l(x)$.
    By ergodicity of $\mu_\Vcal$, we have $\nu_x=\mu_\Vcal$ for almost every $x$.
\end{proof}

In light of this lemma, it suffices to prove $A$-invariance of $\hat\nu$ under the additional hypothesis that it is $U$-ergodic.
In particular, the second assertion of Theorem~\ref{thm:A-invariance of flag distribution} is an immediate consequence of the following measure classification statement, which proves its first assertion.
\begin{prop}\label{prop:meas class on suspension}
    Every $U$-ergodic probability measure $\hat\nu$ on $\P\hV$, which projects to Haar measure $\mu_\Vcal$ on $\Vcal$, is $A$-invariant.
\end{prop}

The remainder of this section is dedicated to the proof of Proposition~\ref{prop:meas class on suspension}.
Let $\hat{\nu}$ be as in the statement.
The proof of this proposition proceeds by adaptation of Ratner's shearing arguments in her work on measure classification of unipotent invariant measures on quotients of $\SL$~\cite[Section 4]{Ratner-RagConjSL2}.
The key observation that enables implementing these arguments in our skew-product setup is that the action $u(s)$ on the fiber has much slower expansion than the base; cf.~Lemma~\ref{lem:subpolynomial divergence} below.

Using the norm on the fibers of $\hV$, we define a metric on the fiber $\RP_x$ of $\P\hV$ over $x\in\Vcal$ as follows: given $x\in\Vcal$ and $\bar v ,\bar w \in \RP_x$, let\footnote{Here, we use the same notation for an induced norm on $\wedge^2 V_x$ from $\norm{\cdot}_x$; cf.~\cite[Appendix A4]{QuasEtal-ExteriorNorms} for an explicit construction of induced norms.} 
\begin{align}\label{eq:proj dist}
    \dist(\bar v ,\bar w ) = \frac{\norm{v\wedge w}_x}{\norm{v}_x \norm{w}_x},
\end{align}
where $v$ and $w$ are representatives in $V_x$ of $\bar v$ and $\bar w$ respectively. The following is the key estimate on subpolynomial divergence of distances in the fiber under the cocycle that underlies our proof of Theorem~\ref{thm:A-invariance of flag distribution}.
The proof of the lemma is postponed to the next subsection.

\begin{lem}[subpolynomial divergence in the fibers]
\label{lem:subpolynomial divergence}
    For every $\e>0$, there is a set $F$ with $\hat\nu(F)>1-\e$ so that for all $s\in\R$, if $(x,v)$ and $(x,w)$ belong to $F$, then 
    \begin{align*}
        \dist(u(s)\cdot(x,v), u(s)\cdot(x,w)) \ll (1+|s|)^\e\dist(v,w).
    \end{align*}
\end{lem}

Armed with Lemma~\ref{lem:subpolynomial divergence}, the rest of the argument is now very similar to Ratner's original proof as we now describe.
Let $\L(\hat{\nu})\subseteq \SL$ denote the subgroup of elements of $\SL$ preserving $\hat\nu$ and suppose that $A \not\subset \L(\hat\nu)$.
Using the fact that $A$ preserves the space of $U$-ergodic measures (cf.~\cite[Proofs of Lemma 4.1 and Theorem 4.1]{Ratner-RagConjSL2}), we can find a full measure set $\hat{\mathcal{E}}$ of $U$-generic points and $\th>0$ so that for all $0<|t|\leq \th$, $g_t  \hat{\mathcal{E}}\cap \hat{\mathcal{E}} =\emptyset$.
In particular, there is a compact set $\hat{\mathcal{L}}\subseteq \hat{\mathcal{E}}$ with $\hat{\nu}(\hat{\mathcal{L}})$ arbitrarily close to $1$ so that 
\begin{align}\label{eq:move generic points off themselves A-inv}
    \e\stackrel{\mrm{def}}{=}\inf\set{\dist(g_t\hat{\mathcal{L}}, \hat{\mathcal{L}}): \th/2\leq |t| \leq \th }>0.
\end{align}

Let $\d>0$ be sufficiently small, to be chosen depending on $\th,\e,$ and $\hat{\Lcal}$.
Since $\hat{\nu}$ projects to the Haar measure on $\Vcal$, we can find two generic points $(y,v_1), p\cdot (y,v_2)\in \hat{\mathcal{L}}$, $i=1,2$, where $p\in \SL$ is a non-trivial lower triangular matrix which is $\d$-close to identity, and moreover $v_1$ is $\d$-close to $v_2$.

For $r\in\R$, denote by $u^-(r) =\left(\begin{smallmatrix}
    1 & 0 \\ r & 1
\end{smallmatrix}\right)$.
Let $r,\t \in (-\d,\d)$ be such that $p= u^-(r)g_\t$.
For $s\in \R$ with $|sr|<1$, let $s_p = \t^{-2} s/(1+sr)$, $r_p=r/(1+sr)$, and $t_p=\t+\log (1+rs)$.
Then, we have
$$u(s)p = u^-(r_p) g_{t_p} u(s_p).$$

Let $I\subset\R$ be the interval of parameters $s$ such that $\th/2\leq |t_p| \leq \th$. Since $|s_p|\asymp |s|$, it follows by Lemma~\ref{lem:subpolynomial divergence} that for all $s\in I$, the distance between the points $u(s_p)\cdot (y,v_i)$, $i=1,2$, is $O(\d)$. 
Moreover, we have that $|r_p| =O(\d)$.
In particular, the distance between $g_{t_p}u(s_p)\cdot (y,v_1)$ and the generic point $u(s)p\cdot (y,v_2)$ satisfies
\begin{align}\label{eq:generic points differ along g_t A-inv}
    \dist(g_{t_p}u(s_p)\cdot (y,v_1), u(s)p\cdot (y,v_2)) =O(\d).
\end{align}
Following Ratner, with the aid of Birkhoff's ergodic theorem, we can find $s\in I$ so that the two points $u(s_p)\cdot (y, v_1)$ and $u(s)p\cdot (y,v_2)$ also belong to the compact set $\hat{\mathcal{L}}$.
This contradicts~\eqref{eq:move generic points off themselves A-inv} if $\d$ is sufficiently small depending on $\e$.

\subsection{Proof of Lemma~\ref{lem:subpolynomial divergence}}

Note that for all $x$ and any two unit norm vectors $v,w$, we have
\begin{align}\label{eq:bound proj dist by norms}
    \dist(B(u(s),x)v,B(u(s),x)w) \leq \frac{\norm{\wedge^2 B(u(s),x)}_{\mrm{op}} \norm{v\wedge w}_x}{\norm{B(u(s),x)v)}_{u(s)x}\norm{B(u(s),x)w}_{u(s)x}},
\end{align}
where $B(u(s),x):\mathcal{V}_x \to \mathcal{V}_{u(s)x}$ is the cocycle defined in Section~\ref{sec:A intro}.
Our goal is to show that the rate of growth of $B(u(s),x)v$ is close to that of the norm of $B(u(s),x)$ for almost every $(x,v)$.
To this end, we first relate these rates of growth to growth along suitable orbits of $g_t$ instead of $u(s)$.
 By a direct calculation of the singular values and vectors of $u(s)$, we see that $u(s)\in \mrm{SO}(2) g_{t_s} k_s$, where $e^{t_s}\asymp |s|$, and $k_s$ is a matrix that is $O(1/|s|)$-close to $\left(\begin{smallmatrix}
        0& -1 \\ 1 & 0
    \end{smallmatrix}\right)$ and with top-right entry $-1+O(1/|s|^2)$
    as $|s|\r \infty$.
    Hence, by the cocycle property and our boundedness hypothesis~\eqref{eq:bounded cocycle}, we have $B(u(s),x) = M_sB(g_{-t_s},x)$, where $M_s$ is a matrix of size $O(1)$. Thus,
    \begin{align}\label{eq:bound proj dist using KAK}
        \dist(B(u(s),x)v,B(u(s),x)w) \ll \frac{\norm{\wedge^2 B(g_{-t_s},x)}_{\mrm{op}} \norm{v \wedge w}_x}{\norm{B(g_{-t_s},x)v}_{g_{-t_s}x} \norm{B(g_{-t_s},x)w}_{g_{-t_s}x}}
    \end{align}

 The projection $\P\hV\r \Vcal$ provides a disintegration of $\hat{\nu}$ along the fibers.
We denote by $\hat\nu_x$ the corresponding conditional measure of $\hat\nu$ on the fiber over $x$.
Let $N_x$ denote the smallest projective subspace of $\RP_x$ containing the support of $\hat\nu_x$.
We also use $N_x$ to denote the corresponding linear subspace of $V_x$.
Let $X\subseteq \Vcal$ denote the $\mu_\Vcal$-full measure set so that $\hat{\nu}_x$ is defined for every $x\in X$.

To proceed, we need the following definitions.
\begin{definition}
A measurable sub-bundle of $\hV$ (resp. $\P\hV$) is a measurable assignment of a vector subspace (resp. projective subspace) to each point in a full measure subset of $\Vcal$. 
Given a subgroup $H\subset \SL$ and a sub-bundle $\widehat{\Wcal}=\set{(x,q): x\in X' w\in W_x}$, where $X'\subseteq \Vcal$ is a subset of full measure, we say that $\widehat{\Wcal}$ is $H$-invariant if $X'$ is $H$-invariant and for every $x\in X'$ and $h\in H$, we have $W_{hx}= B(h,x)W_x$.    

\end{definition}

Let $P=AU\subset \SL$. We introduce certain measurable invariant sub-bundles of full measure. Since the measure $\hat\nu$ is apriori not $P$-invariant, it will be convenient to restrict our attention to a countable subgroup.

Recall that $\hat\nu$ is $U$-ergodic.
In particular, it is ergodic for the action of one element in $U$.
Without loss of generality, we shall assume in what follows for concreteness that $\hat\nu$ is ergodic for the action of a matrix in $U$ with rational entries.

In what follows, given a ring $R\in \set{\R,\Q,\Z}$ and a subgroup $H\subseteq \SL$, we let $H_R \stackrel{\mrm{def}}{=}H\cap \mrm{SL}_2(R)$.
 Since the measure $\mu_\Vcal$ is $\SL$-invariant, we have that the set
\begin{align*}
    X' = \bigcap_{p\in P_\Q} pX
\end{align*}
also has full measure. Moreover, $X'$ is $P_\Q$-invariant by construction.
For $x\in X'$, define 
$$W_x \stackrel{\mrm{def}}{=} \mrm{Span}\set{B(p,p^{-1}x)N_{p^{-1}x}:  p\in P_\Q}.$$
Then, $\widehat\Wcal\stackrel{\mrm{def}}{=}\set{(x,w): x\in X', w\in W_x}$ is a measurable $P_\Q$-invariant sub-bundle satisfying $\hat\nu(\widehat\Wcal)=1$.
Moreover, $\widehat{W}$ is minimal among $P_\Q$-invariant sub-bundles with this property in the sense of having fibers with smallest dimension. 
More precisely, any measurable $P_\Q$-invariant sub-bundle $\widehat{\Wcal}'$ of full $\hat\nu$-measure with fibers $W'_x$, has the property that $W'_x\supseteq W_x$ for $x$ in the common $P_\Q$-invariant full $\mu_\Vcal$-measure set on which both bundles are defined.

Since the cocycle is bounded in the sense of~\eqref{eq:bounded cocycle}, it is in particular log-integrable so that Oseledets' theorem applies showing that the limit
\begin{align}\label{eq:lambda_1 minus}
    \l_1^- \stackrel{\mrm{def}}{=} \lim_{\substack{n\r \infty\\ n\in\N}} \log\norm{B(g_{-n},x)|_{W_x}}^{1/n}_{\mrm{op}}
\end{align}
exists and is constant for $\mu_\Vcal$-almost every $x$.
Moreover, there is an almost everywhere defined measurable map $x\mapsto W_x^{<\l_1^-}\subset W_x$, where $W_x^{<\l_1^-}$ is the Oseledets subspace consisting of all $w\in W_x$ with $\limsup_{n\to \infty} (1/n)\log\norm{ B(g_{-n},x)w}_{g_{-n}x} < \l_1^-$.

Let $X''\subseteq X'$ be a full measure set of $x$ where~\eqref{eq:lambda_1 minus} exists and $W_x^{\l_1^-}$ is defined.
Up to replacing $X''$ with $\cap_{p\in P_\Q} pX''$, we may and will assume that $X''$ is a $P_\Q$-invariant subset of $\Vcal$.
Let $\widehat{\Wcal}^{<\l_1^-}$ denote the measurable sub-bundle $\set{(x,w): x\in X'', w\in W_x^{<\l_1^-}}$.
We claim that
\begin{align}\label{eq:measure 0 lower growth}
    \hat{\nu}\left(\widehat{\Wcal}^{<\l_1^-}\right) =0.
\end{align}

Indeed, we first show that $\widehat{\Wcal}^{<\l_1^-}$ is $P_\Q$-invariant.
To see this, note that $\widehat{\Wcal}^{<\l_1^-}$ is $A_\Z$-invariant by definition.
From the cocycle property $B(gh,x)=B(g,hx)B(h,x)$ and our boundedness hypothesis~\eqref{eq:bounded cocycle}, it follows that it is $A_\Q$-invariant.
Similarly, since $g_{-n}ug_n $ tends to identity as $n\r+\infty$, it follows that $\widehat{\Wcal}^{<\l_1^-}$ is $U_\Q$-invariant.

Hence, recalling that $\hat\nu$ is $U_\Q$-ergodic, $\widehat{\Wcal}^{<\l_1^-}$ has measure $0$ or $1$. 
The claim~\eqref{eq:measure 0 lower growth} now follows since $\widehat\Wcal$ is the minimal $P_\Q$-invariant sub-bundle of full measure and $\widehat\Wcal^{<\l_1^-}$ is a proper sub-bundle.

The last ingredient in the proof is another application of Oseledets' theorem to the second exterior power bundle yielding almost sure existence of the limit
\begin{align*}
    \l_1^-+\l_2^- = \lim_{n\to\infty, n\in\N} \log\norm{\wedge^2 B(g_{-n},x)}^{1/n}_{\mrm{op}},
\end{align*}
with value independent of $x$. Here, $\l_2^-\leq \l_1^-$ is the second Lyapunov exponent for the cocycle $B(g_{-t},-)$.
Now, given $\e>0$, we can find $n_\e>0$ so that the sets $F_1$ and $F_2$ defined by
\begin{align*}
    F_1 &= \set{(x,v): \norm{\wedge^2 B(g_{-n},x)}_{\mrm{op} }\leq e^{(\l^-_1+\l^-_2+\e/2)n} \text{ for all } n>n_\e, n\in \N},
    \nonumber \\
    F_2 &= \set{(x,v): \norm{B(g_{-n},x)v}_{g_{-n}x} \geq e^{(\l^-_1-\e/2) n}\norm{v}_x \text{ for all } n>n_\e, n\in \N },
\end{align*}
each has measure $\geq 1-\e/2$.
As above, note that up to replacing the bounds in the definition of $F_1$ and $F_2$ by a suitable uniform constant multiple, these bounds continue to hold for all $t\in\R$ with $t>n_\e$.
The conclusion of the lemma now follows for $F=F_1\cap F_2$ in light of~\eqref{eq:bound proj dist by norms} and~\eqref{eq:bound proj dist using KAK}.
Indeed, using that $\l^-_2\leq \l^-_1$, we have $$\dist(B(u(s),x)v,B(u(s),x)w)\ll \frac{e^{(2\l^-_1+\e/2)t_s}}{e^{(2^-_1-\e/2)t_s}} \leq e^{\e t_s}\ll |s|^\e,$$ for all large enough $s$, where the last inequality follows by definition of $t_s$ above~\eqref{eq:bound proj dist using KAK}.

\section{Proof of the Key Matching Proposition}
\label{sec:proof of key matching}
The goal of this section is to prove Proposition~\ref{prop:key matching}.
The key ingredient is Theorem~\ref{thm:A-invariance of flag distribution}, asserting that any weak-$\ast$ limit of the distributions of output vectors of the cocycle on projective space along expanding horocycle arcs on $\Vcal$ is invariant by the geodesic flow.

\subsection{The bundle with fibers the balanced cylinder space}
We begin by setting up notation for applying Theorem~\ref{thm:A-invariance of flag distribution}.
Fix a horizontally periodic Veech surface $(M,\w)$.
The action of the affine group 
$$\G\stackrel{\mrm{def}}{=} \mrm{Aff}^+(M,\w)$$
by affine maps on $M$ induces a linear action of $\G$ on $H^1_\C$, preserving the subspace $\cylspace^0(\w)$; cf.~\textsection\ref{sec:balanced spaces}.
Moreover, taking derivatives in translation charts gives a surjective homomorphism $D:\G\r \mrm{SL}(M,\w)$ onto the Veech group with a finite kernel; cf.~\textsection\ref{sec:period coords}. 
Recall that we passed to a finite torsion-free cover of the stratum; cf. Convention \ref{torsion free convention}. In particular, using the notation set therein, we view $\G$ as a lattice of $\SL$ through this derivative homomorphism, with quotient the Veech curve $\Vcal = \SL/\mrm{SL}(M,\w)$.

Hence, following the discussion preceding Theorem~\ref{thm:A-invariance of flag distribution}, we can form the vector bundle $\hV=\SL \times \cylspace^0(\w)/\Gamma$, and the associated
projective fiber bundle $\P\hV = \SL \times \mathbb{P}(\cylspace^0(\w))/\Gamma$, where $\mathbb{P}(\cylspace^0(\w))$ denotes the space of lines in $\cylspace^0(\w)$.
In particular, the fiber $\hV_x$ of $\hV$ over $x=g\w\in\Vcal$ is given by $\mrm{KZ}(g,\w)\cdot\cylspace^0(\w)$, which depends only on $x$, and not the choice of $g$ by invariance of $\cylspace^0(\w)$ under $\G$.

We let the norm on the fibers be the restriction of the AGY norm.
The boundedness hypothesis~\eqref{eq:bounded cocycle} follows by Corollary~\ref{cor:KZ norm}.
We let $\pi:\P\hV\r \Vcal$ denote the canonical projection map associated with this fiber bundle.

\subsection{Matching of initial points and directions on the Veech curve}
The key step in the proof of Proposition~\ref{prop:key matching} involves matching points on the expanded horocycle in $\Vcal$ at different times along with matching the images of $\b$ at the corresponding points under the cocycle, so that the matched pairs are close in distance.  
This is done in Proposition~\ref{prop:intermediate matching} below. To state this result, we need some setup.

Fix some $\e\in (0,1)$ and let $\Kcal\subset \Vcal$ be a compact set with boundary having $\mu_\Vcal$ measure $0$ and  such that
\begin{align}\label{eq:matching proof cpt set}
    \int_0^1 \mathbf{1}_\Kcal(\w(t,s))\;ds \geq 1-\e/2, \qquad \text{for all } t\geq 0.
\end{align}
For instance, we may take $\Kcal$ to consist of all points $x\in\Vcal$ with injectivity radius suitably bounded below in terms of $\e$; cf.~\cite[Proposition~5.3]{ChaikaKhalilSmillie}.
Let $\widehat{\Kcal} \stackrel{\mrm{def}}{=} \pi^{-1}(\Kcal)\subset \P\hV$. 

For each $t\geq 0$, let $E_t = \set{(\w(t,s), \b(t,s)/N(t,s)): s\in [0,1], \w(t,s)\in \Kcal}$, where we recall that $N(t,s):=N_{\beta}(t,s)$ was defined in~\eqref{eq:notation}.
We can view $E_t$ as a subset of $\widehat{\Kcal}$.
Elements of $E_t$ are parametrized by a subset of $s\in [0,1]$ for which $\w(t,s)$ lands in $\Kcal$.
We let $\l_t$ be the pushforward to $E_t$ of the Lebesgue measure under this parametrization map, normalized to be a probability measure.

We now define a metric on $\widehat{\mathcal{K}}$ built from the AGY-norms. This will be convenient to relate an estimate on Tremor distance  using the conventional AGY metric and norms (Lemma \ref{lem:tremors are Lipschitz}) and $\P\hV$-distance. 
Let $r_0>0$ be chosen so that the $r_0$-neighborhood of every point in $\mathcal{K}$ is simply connected. 
We extend the AGY metric $\dAGY$ to $\widehat{\mathcal{K}}\subset \mathbb{P}\hat{\mathcal{V}}$ as follows. Let $\bar{x}_i=(x_i,v_i)\in\widehat{\Kcal}, i=1,2$, and set
\begin{align*}
    \dAGY(\bar{x}_1,\bar{x}_2) 
    \stackrel{\mrm{def}}{=} 
    \min\set{\dAGY(x_1,x_2) + \frac{1}{2} \sum_{i=1}^2 \dist_{x_i}(v_1,v_2), r_0 }
\end{align*}
where $\dist_{x}$ is a distance on the projective space fiber over derived from the AGY norm $\norm{-}_x$; cf.~\eqref{eq:proj dist} for a definition. Here, if $\dAGY(x_1,x_2)<r_0$, then we view $v_i, i=1,2$ as elements of the same fiber using parallel transport so that $\dist_{x_i}(v_1,v_2)$ is well-defined for $i=1,2$ by our choice of $r_0$. Otherwise, the distance $\dAGY(\bar{x}_1,\bar{x}_2)$ is set to be equal to $r_0$. 

\begin{prop}\label{prop:intermediate matching}

        Let $\th>0$ and $L_0>0$ be given.
    Then, the following holds for all sufficiently large $t>0$, depending on $\th,\Kcal,$ and $L_0$.
    Let $t_1=t$ and $t_2\geq t_1$ be such that $t_2-t_1\leq L_0$.
    Then, there is a subset $F_{t_1}\subseteq E_{t_1}$ with $\l_{t_1}(F_{t_1})\geq 1-\th$ and a measurable map $\phi: F_{t_1} \r E_{t_2}$ such that for all $\bar{x}=(x,v)\in F_{t_1}$, we have $\dAGY(\bar{x},\phi(\bar{x})) < \th$.

\end{prop}

\subsection{Deduction of Proposition~\ref{prop:key matching} from Proposition~\ref{prop:intermediate matching}}
We begin with the following Lipschitz estimate on the distance between tremored surfaces.

\begin{lem}\label{lem:tremors are Lipschitz}
    There exists $\d>0$, depending only on $\Kcal$, so that the following holds for all $|r|<\d$.
    Let $t_1, t_2 \geq 0$ and $s_1, s_2\in [0,1]$. Let $v_i = \b(t_i,s_i)/N(t_i,s_i)$,
     $x_i=\w(t_i,s_i)$, and $y_i=\Trem(x_i,r v_i)$.
     If $\dist_{\mrm{AGY}}(x_1,x_2)<\d$, then
     $\dist_{\mrm{AGY}}(y_1,y_2) \leq 8(\dist_{\mrm{AGY}}(x_1,x_2)+ |r|\norm{v_1-v_2}_{x_1})$.
\end{lem}

\begin{proof}
    For $i=1,2$, let $\tilde{x}_i\in \HHm$ be a lift of $x_i$ to the marked stratum.
    Then, $\tilde{y}_i = \Trem(\tilde{x}_i,rv_i)$ is a lift of $y_i$.
    Recall the map $\Psi^u_\bullet$ parametrizing local unstable manifolds; cf.~\textsection\ref{sec:(un)stable}.
    Note that since $\b$ belongs to the balanced space at $\w$, we get that $\b(t,s)\in E^u(\w(t,s))$ for all $t,s\in\R$.
    In particular, for $\d<1/2$, $\tilde{y}_i \stackrel{\mrm{def}}{=} \Psi^u_{\tilde{x}_i}(v_i)$ is well-defined for $i=1,2$.
    Moreover, by Proposition~\ref{prop:stable exp lipschitz}, we obtain
    \begin{align*}
        \dAGY(x_i,y_i) \leq \dAGY(\tilde{x}_i,\tilde{y}_i) \leq 2\d.
    \end{align*}
    Hence, if $\dAGY(x_1,x_2)<\d$, we get that all $4$ points $x_i,y_i, i=1,2$ belong to a ball of radius $10\d$ centered in $\Kcal$. 
    In what follows, we choose $\d$ small enough, depending on $\Kcal$, so that holonomy period coordinates (cf.~\textsection\ref{sec:period coords}) are injective on any such ball.
    In particular, it will be enough to prove the lemma in the marked stratum $\HHm$.

    By~\cite[Lemma 3.3]{ChaikaKhalilSmillie}, there exists $\e_0>0$ so that for all $q_1,q_2$ in the unit neighborhood of $\Kcal$, if $\dist_{\mrm{AGY}}(q_1,q_2) <\e_0$, then
    \begin{align}\label{eq:AGY vs period coords}
        2^{-1} \norm{\hol_{q_1}-\hol_{q_2}}_{q_1} 
        \leq  \dist_{\mrm{AGY}}(q_1,q_2) 
        \leq 2 \norm{\hol_{q_1}-\hol_{q_2}}_{q_1}.
    \end{align}
    Moreover, by the previous paragraph, choosing $\d<\e_0/20$ ensures that $\dist_{\mrm{AGY}}(y_1,y_2)
    <\e_0$ so that~\eqref{eq:AGY vs period coords} applies and yields for $w=\hol_{y_1}-\hol_{y_2}$ the bound
    \begin{align*}
        \dAGY(y_1,y_2) 
        \leq 2\norm{w}_{y_1}. 
    \end{align*}
    Next, we apply Proposition~\ref{prop:norm of parallel transport} with $\k$ being the (balanced) tremor path joining $x_1$ to $y_1$, with tangent vector $\dot{\k}(t)\equiv v_1$, to get
    \begin{align*}
       \norm{w}_{y_1}
        \leq \frac{\norm{w}_{x_1}}{1-|r|\norm{v_1}_{x_1}}
        \leq 2 \norm{w}_{x_1},
    \end{align*}
    where we used the bound $|r|<\d<1/2$ and that $\norm{v_1}_{x_1}=1$.
    Recall the relation between periods of tremored surfaces $x_1$ and $y_1$ in~\eqref{eq:tremor in coordinates}. 
    Thus, by the triangle inequality and~\eqref{eq:AGY vs period coords} applied with $q_i=x_i$, we obtain
    \begin{align*}
    \dAGY(y_1,y_2) 
        &\leq  4(\norm{\hol_{x_1}-\hol_{x_2}}_{x_1} + |r| \norm{v_1-v_2}_{x_1})
        \nonumber\\
    &\leq 4(2\dAGY(x_1,x_2) + |r|\norm{v_1-v_2}_{x_1})    .
    \end{align*}
    This concludes the proof.
    \qedhere
\end{proof}

\subsubsection{Conclusion of the proof of Proposition~\ref{prop:key matching}}
Let $\e>0$ be given. 
We show that the proposition holds with our choice of $\Kcal$ and with $\d$ the parameter provided by Lemma~\ref{lem:tremors are Lipschitz}.
Let $T>1$ and $L_0$ be given
and let $\th=\min\set{\e e^{-2T},\d}/C$, where $C\geq 1$ will be chosen to be a suitably large constant depending on $\Kcal$.
In what follows, $t\geq T$ will be large enough so that the conclusion of Proposition~\ref{prop:intermediate matching} holds for these choices of $\th$ and $L_0$.

Fix $\ell\in [0,L_0]$ and let $t_1=t-T-(L_0-\ell)$ and $t_2=t -T$.
Let $F_{t_1}\subseteq [0,1]$ be the set provided by Proposition~\ref{prop:intermediate matching}.
Set $S_\ell=\set{s\in [0,1]: (\w(t_1,s),\b(t_1,s))\in  F_{t_1}}$. Then, Proposition~\ref{prop:intermediate matching} and equation~\eqref{eq:matching proof cpt set} imply that $|S_\ell|\geq (1-\th)(1-\e/2)\geq 1-\e$.
Moreover, by definition of $F_{t_1}$, we have that $\w(t_1,s)\in\Kcal$.
This verifies item~\eqref{item:key match in cpt}.

To verify item~\eqref{item:key match closeness}, let $s\in S_\ell$ and $r\in [0,\d/N(t_1,s)]$.
We wish to show that
\begin{align*}
        g_T \cdot \Trem(t_1,s,r)\in B\left(g_T \cdot g_{t_2}\cdot \T(\w,\b),\e \right).
\end{align*}
Let $x_1 = \w(t_1,s)$, $v_1 = \b_1(t_1,s)/N(t_1,s)$, and $y_1 = \Trem(x_1, rv_1)$.
Let $(x_2,v_2) \in \phi(x_1,v_1)$, where $\phi$ is the map in Proposition~\ref{prop:intermediate matching}, and $(x_2,v_2)\in \hV$ is a closest point to $(x_1,v_1)$ in the equivalence class of the line $\phi(x_1,v_1)$.
Set $y_2 = \Trem(x_2,rv_2)$.
Then, $y_2\in g_{t_2}\cdot \T(\w,\b)$.
Moreover, we get by Proposition~\ref{prop:intermediate matching} that $\dAGY(x_1,x_2)<\th$ and $\norm{v_1-v_2}_{x_1}\ll \th$. 
Hence, since $\th\leq \d$, we get that $\dAGY(y_1,y_2)\ll \th$ by Lemma~\ref{lem:tremors are Lipschitz}.
By choosing $C$ in the definition of $\th$ to be large enough to overcome the implicit constant in this inequality, we obtain $\dAGY(y_1,y_2)\leq \e e^{-2T}$.
By Lemma~\ref{lem:nonexpansion of stable}, it follows that $\dAGY(g_T y_1, g_T y_2)\leq \e$.

\subsection{Proof of Proposition~\ref{prop:intermediate matching}}
Let $\th$ and $L_0$ be given. 
Recall the measures $\l_t$ defined above Proposition~\ref{prop:intermediate matching}.
Note that the family $\set{\l_t:t\gg_{\Kcal}1}$ consists of probability measures supported on the compact set $\widehat{\Kcal}$. 
We have the following immediate consequence of Theorem~\ref{thm:A-invariance of flag distribution}.
\begin{lem}\label{lem:conseq of A-invar}
    The weak-$\ast$ distance between $\l_t$ and $\l_{t+\ell}$ converges to $0$ as $t\r \infty$, uniformly over $\ell\in [0,L_0]$.
\end{lem}
\begin{proof}
    Indeed, suppose not.
    Then, there is a continuous function $f$ on $\widehat{\Kcal}$, and sequences $\ell_n \in [0,L_0]$ and $t_n\to \infty$ so that $|\int f\;d\l_{t_n} -\int f\;d\l_{t_n+\ell_n}| \not \r 0$. After passing to a subsequence if necessary, we may assume the measures $\mu_{t_n} = \int_0^1 \d_{g_{t_n}u(s)\cdot (\w,\b)}ds$ converge to a measure $\hat{\nu}$ and $\ell_n\to \ell_\ast\in [0,L_0]$. 
    It follows that $\l_{t_n}(f)\r \int f\mathbf{1}_\Kcal \;d\hat\nu/\mu_\Vcal(\Kcal)$ and  $\l_{t_n+\ell_n}(f)\r \int f\circ g_{\ell_\ast} \mathbf{1}_\Kcal\circ g_{\ell_\ast} \;d\hat\nu/\mu_{\Vcal}(\Kcal)$, and those limits do not agree.
    We obtain a contradiction in light of Theorem~\ref{thm:A-invariance of flag distribution} which provides that $\hat\nu$ is $g_{\ell_\ast}$-invariant.
\end{proof}

Fix $\ell\in [0,L_0]$ and let $t_1=t$ and $t_2=t+\ell$.
    Lemma~\ref{lem:conseq of A-invar} and the Kantarovich-Rubenstein duality theorem (cf.~\cite[Remark 6.5]{Villani-OptimalTransport}) imply that, 
    if $t$ is large enough,
    the Wasserstein $W_1$-distance between $\l_{t_1}$ and $\l_{t_2}$ is at most $\th^2$.
    In other words, there exists a probability measure $\g_{t_1,t_2}$ on the product space $\P\hV^2$, which projects to $\l_{t_1}$ and $\l_{t_2}$ respectively under the two standard projections, and such that 
    \begin{align*}
        \int_{\P\hV^2} \dist(x,y) \;d \g_{t_1,t_2}(x,y) <\th^2,
    \end{align*}
    where $\dist$ denotes the metric on the product space given by the sum of distances of projections to individual factors.
    Let $\Delta_\th$ denote the $\th$-neighborhood of the diagonal in the product space $\P\hV^2$.
    Then, the above estimate and Markov's inequality imply that $\g_{t_1,t_2}(\Delta_\th)\geq 1-\th$.

    Note that $\g_{t_1,t_2}$ is supported on $E_{t_1}\times E_{t_2}$.
    Let $F_{t_1}$ denote the intersection of the projection of $\Delta_\th$ with the support $E_{t_1}$ of $\l_{t_1}$.
    Then, since $\g_{t_1,t_2}$ projects to $\l_{t_1}$, we see that $\l_{t_1}(F_{t_1})\geq 1-\th$.
    
    Given $x\in F_{t_1}$, let $d(x)=\inf\set{\dAGY(x,y): y\in E_{t_2} }$ and note that $d(x)<\th$ by definition. 
    We let $I(x)$ denote the set of points $y\in E_{t_2}$ so that $d(x) = \dist(x,y)$. Then, $I(x)$ is non-empty and closed by compactness of $E_{t_2}$.
    Recalling that the sets $E_t$ are parametrized by the subset of points $s\in [0,1]$ for which $\w(t,s)$ lands in $\Kcal$, we set $\phi(x)$ to be the point $y\in I(x)$ with the smallest corresponding parameter $s\in [0,1]$. In particular, $\phi$ is a measurable map satisfying the conclusion of the proposition.

\section{Transverse Monodromy, Full Support, and Proof of Theorem~\ref{thm:full support}}
\label{sec:trivial monodromy}

This section is dedicated to the proof of Theorem~\ref{thm:full support}.
The strategy is summarized in Section~\ref{sec:outline}.

\subsection{Weak-stable matching on the Veech curve}

Recall Convention \ref{torsion free convention} and the notation set therein.
Let $\g \in \G=\mrm{Aff}^+(M,\w)$ denote the pseudo-Anosov affine mapping class acting trivially on $\cylspace^0(\w)\subset H^1_\R$. 
As in Section~\ref{sec:proof of key matching}, we let $\hV = \SL \times \cylspace^0(\w)/\Gamma$ denote the vector bundle over $\Vcal$ with fiber $\hV_x$ over $x=g\w$ given by $\mrm{KZ}(g,\w)\cdot\cylspace^0(\w)$.

\begin{convention}\label{trivial monodromy convention}
In what follows, we identify the pseudo-Anosov element $\g$ with its image in the Veech group $\mrm{SL}(M,\w)\subset \SL$. 
To simplify notation, we use the notation $\mrm{KZ}(-)$ to denote the restriction of the cocycle to the $\SL$-invariant sub-bundle $\hV$.
Moreover, whenever $x,y\in\Vcal$ belong to a simply connected open set in $\Vcal$, we shall identify the fibers $\hV_x$ and $\hV_y$ via parallel transport.
In particular, given $g_1,g_2\in\SL$ and $x_1,x_2\in\Vcal$ such that each of the pairs of points $\set{g_ix_i: i=1,2}$ and $\set{x_i:i=1,2}$ belong to simply connected subsets of $\Vcal$, we write $\mrm{KZ}(g_1,x_1)=\mrm{KZ}(g_2,x_2)$ to indicate equality of these linear maps after suitably pre- and post-composing with parallel transport maps that identify their (co-)domains.
\end{convention}

Recall the notation introduced in~\eqref{eq:notation}.
\begin{prop}\label{prop:match on Vcal trivial monodromy}
    Let $\ell>0$ be any multiple of the primitive period of the periodic geodesic corresponding to $\g$.
    For every $\e>0$, there exists a compact set $\Kcal=\Kcal(\e)\subset \Vcal$, so that the following holds for all sufficiently small $\d=\d(\e,\ell)>0$, and all large enough $t=t(\d)>0$.
    There is a subset $G_t\subseteq [0,1]$ of measure at least $1-\e$ and a measurable, locally smooth, map $\vp_{t\to t-\ell} :G_t^{t-\ell}\r [0,1]$ such that for all $s\in G_t^{t-\ell}$, 
    \begin{enumerate}
        \item\label{item:base in compact} $\w(t,s) \in\Kcal$, 
        \item\label{item:lower triangular} there is a lower triangular matrix $p^-$ at distance at most $\d$ from identity in $\SL$ so that $\w(t,s)=p^- \w(t-\ell,\vp_{t\to t_\ell}(s))$, 
        \item\label{item:equal matrices} $\mrm{KZ}(g_t,u(s)\w) = \mrm{KZ}(g_{t-\ell},u(\vp_{t\to t-\ell}(s))\w)$,
        \item\label{item:Jacobian} The Jacobian of $\vp_{t\to t-\ell} $ is of size $ \asymp_\ell 1$ on its domain, where the implicit constant is uniform over all large $t$.
    \end{enumerate}
\end{prop}
\begin{proof}

Let $x\in\Vcal$ be a point on the periodic geodesic corresponding to the pseudo-Anosov element $\g$, i.e., $\g\in \mrm{Stab}_{\SL}(x)$ and $g_\ell x = x$.
Fix $\e>0$, and let $\Kcal\subset\Vcal$ be a compact set so that for all $t\geq0$, the set of $s\in [0,1]$ with $\w(t,s)\in \Kcal$ has measure $\geq 1-\e/2$.
By enlarging $\Kcal$, we may further assume that it contains the entire periodic orbit of $x$, as well as the periodic horocycle through $\w$.

Fix $\d\in (0,1)$ smaller than half the injectivity radius of the $1$-neighborhood of $\Kcal$.
Let $B$ be a flow box of radius $\d e^{-2\ell}$ around $x$.
By equidistribution of $g_t$-pushes of horocycle arcs on $\Vcal$, we can find $t_0>\ell$ so that
\begin{align}\label{eq:G_t}
    \tilde{G}_\ast \stackrel{\mrm{def}}{=} \set{s\in [0,1]: g_\rho u(s)\w \in B \text{ for some } \ell \leq \rho \leq t_0}
\end{align}
 has measure at least $1-\e/2$.
    For $t\geq t_0$, let $G^{t-\ell}_t \subseteq \tilde{G}_\ast$ be the set of points $s$ so that $\w(t,s)\in\Kcal$.
    Then, $G^{t-\ell}_t$ has measure $\geq 1-\e$, and satisfies Part~\eqref{item:base in compact} of the proposition.
 
 Fix $t\geq t_0$.
Define a first hitting time function $\s: G_t \to [\ell,t_0]$ as follows:
\begin{align*}
    \s(s) = \inf \set{\ell\leq \rho \leq t_0: g_\rho u(s)\w\in B }.
\end{align*}

Given $s\in G^{t-\ell}_t$, we note that the distance between $g_{\s(s)-\rho} u(s) \w$ and $g_{-\rho} x$ is at most $\d$ for all $0\leq \rho \leq \ell$.
Since $\d$ is smaller than the injectivity radius of the periodic orbit of $x$, and taking into account Convention~\ref{trivial monodromy convention} and the cocycle property, we obtain
\begin{align*}
    \mrm{KZ}(g_{\s(s)}, u(s)\w) 
    =   \mrm{KZ}(g_{\ell}, \w(\s(s)-\ell,s)) \cdot \mrm{KZ}(g_{\s(s)-\ell},u(s)\w)
    =  \mrm{KZ}(g_{\ell}, x) \cdot \mrm{KZ}(g_{\s(s)-\ell},u(s)\w).
\end{align*}
Now, since $ \mrm{KZ}(g_{\ell}, x)$ is the image of the pseudo-Anosov $\g$ in the monodromy representation, it follows by our assumption that $ \mrm{KZ}(g_{\ell}, x)$ is the identity matrix.
 Hence, we conclude that
 \begin{align}\label{eq:deletion equation}
    \mrm{KZ}(g_{\s(s)}, u(s)\w) 
    =\mrm{KZ}(g_{\s(s)-\ell},u(s)\w).
\end{align}

We define a matching function $\vp_{t\to t-\ell}:G^{t-\ell}_t \to [0,1]$ as follows.
Since the two points $g_{\s(s)} u(s) \w$ and $g_{\s(s)-\ell} u(s) \w$ belong to the flow box of radius $\d$ around $x$, there exists a unique point $\vp_{t\to t-\ell}(s)\in [0,1]$ so that the point $y(s) \stackrel{\mrm{def}}{=}  \w(\s(s)-\ell,\vp_{t\to t-\ell}(s)) $ satisfies the following two properties:
\begin{enumerate}[(a)]
    \item\label{item:same su} $y(s)$ belongs to the same local \textit{strong unstable} horocycle leaf of $g_{\s(s)-\ell} u(s) \w$.
    
    \item\label{item:same cs} $y(s)$ belongs to the same local \textit{weak stable} leaf of $g_{\s(s)} u(s) \w$.
    
\end{enumerate}
In particular, each $s\in G^{t-\ell}_t$ is contained in an interval of radius $\asymp \d e^{-2(t_0+\ell)}$ on which $\vp_{t\to t-\ell}$ is an injective smooth map onto an interval of length $\asymp e^{-2t_0}$, and has Jacobian of size $O(e^{2\ell})$.
This verifies Part~\eqref{item:Jacobian}.
We also note that Property~\eqref{item:same cs} of $y(s)$ is preserved under the forward geodesic flow, thus verifying Part~\eqref{item:lower triangular}.

Towards Part~\eqref{item:equal matrices}, note that for each $s\in G_t^{t-\ell}$, we have
\begin{align}\label{eq:apply deletion}
    \mrm{KZ}(g_t,u(s)\w) &\stackrel{\eqref{eq:cocycle property}}{=}
    \mrm{KZ}(g_{t-\s(s)},g_{\s(s)}u(s)\w)
    \cdot \mrm{KZ}(g_{\s(s)}, u(s)\w)
    \nonumber\\
    &\stackrel{\eqref{eq:deletion equation}}{=}
    \mrm{KZ}(g_{t-\s(s)},g_{\s(s)}u(s)\w)
    \cdot \mrm{KZ}(g_{\s(s)-\ell}, u(s)\w).
\end{align}
Part~\eqref{item:equal matrices} will follow at once from the following Claim, which is a straightforward consequence of the relative position of points in Properties~\eqref{item:same su} and~\eqref{item:same cs}, as well as triviality of the cocycle on small neighborhoods.
Recall Convention~\ref{trivial monodromy convention}.
\begin{claim}\label{claim:apply same cs and su}
    The following holds for each $s\in G_t^{t-\ell}$.
    Let $t_1=\s(\s)-\ell$, $t_2=t-\s(s)$, $\th=\vp_{t\to t-\ell}(s)$.
    Then,
    \begin{enumerate}
        \item\label{item:apply same su} $\mrm{KZ}(g_{t_1}, u(s)\w) = \mrm{KZ}(g_{t_1},u(\th)\w) $.
        \item\label{item:apply same cs} $\mrm{KZ}(g_{t_2}, g_{\s(s)}u(s)\w) = \mrm{KZ}(g_{t_2},y(s)) $.
    \end{enumerate}
\end{claim}
\begin{proof}
    From Property~\eqref{item:same su}, there is $r=O(\d)$  such that $\w(t_1,s) = u(r) y(s)$.
    In particular, we have $ s = \th + r e^{-2t_1}$.
    Hence, we compute using the cocycle property~\eqref{eq:cocycle property}:
    \begin{align*}
        \mrm{KZ}(g_{t_1} u( r e^{-2t_1}) , u(\th)\w) 
        = \mrm{KZ}(g_{t_1}, u(s)\w) 
        \cdot \mrm{KZ}( u( r e^{-2t_1}) , u(\th)\w) )
        = \mrm{KZ}(g_{t_1}, u(s)\w) ,
    \end{align*}
    where in the second equality, we used the fact that $u(\th)\w \in \Kcal$ and that $re^{-2t_1}$ is smaller than the injectivity radius of $\Kcal$.
    On the other hand, recalling that $y(s)=g_{t_1}u(\th)\w$, we obtain 
    \begin{align*}
        \mrm{KZ}(g_{t_1} u( r e^{-2t_1}) , u(\th)\w) 
        = \mrm{KZ}(u(r) g_{t_1} , u(\th)\w) 
        = \mrm{KZ}(u(r), y(s)) 
        \cdot \mrm{KZ}(g_{t_1} , u(\th)\w) 
        =\mrm{KZ}(g_{t_1} , u(\th)\w) ,
    \end{align*}
    where in the last equality, we used that $r=O(\d)$ is smaller than the injectivity radius at $y(s)$.
    The above two equations imply item~\eqref{item:apply same su}.
    Item~\eqref{item:apply same cs} follows from Property~\eqref{item:same cs} by a very similar computation.
\end{proof}
In view of~\eqref{eq:apply deletion} and Claim~\eqref{claim:apply same cs and su}, we have
\begin{align*}
    \mrm{KZ}(g_t,u(s)\w) =     \mrm{KZ}(g_{t-\s(s)},y(s))
    \cdot \mrm{KZ}(g_{\s(s)-\ell}, u(\vp_{t\to t-\ell}(s))\w)
    \stackrel{\eqref{eq:cocycle property}}{=} \mrm{KZ}(g_{t-\ell},u(\vp_{t\to t-\ell}(s))\w),
\end{align*}
thus verifying Part~\eqref{item:equal matrices}, and concluding the proof.

\qedhere
\end{proof}

\subsection{Weak-stable matching is preserved by tremors near the Veech curve}

Recall that $\g$ is the pseudo-Anosov element acting trivially on $\cylspace^0(\w)$, and let $\ell_0>0$ be its primitive period.
Given $\k>0$, let
\begin{align}\label{eq:T_kappa}
        \T_\k \stackrel{\mathrm{def}}{=} 
        \set{\Trem_\b(0,s,r): \b\in\twist^0(\w), \norm{\b}_\w=1, s\in [0,1], |r|<\k }.
    \end{align}
 
 Recall that $\mu_\T$ is a fully supported Lebesgue probability measure on $\T(\w)$.
 Up to replacing $\mu_\T$ with an equivalent measure in its class, we shall assume without loss of generality that $\mu_\T$ is $U$-invariant.
 We let $\mu_{\T_\k}$ be the restriction of $\mu_\T$ to $\T_\k$, normalized to be a probability measure.
 As in Section~\ref{sec:full Banach density}, we denote by $d_\ast(\cdot,\cdot)$ any metric inducing the weak-$\ast$ topology on Borel measures of total mass $\leq 1$.

\begin{prop}\label{prop:match tori trivial monodromy}
     For all $\e>0$ and $N \in \N$, there exist $\d=\d(\e,N)>0$, $t_0=t_0(\e,N,\d)>0$, and a Borel measure $\l=\l(\e,\d,N)$ on $\T(\w)$ so that the following hold.
    Let $\k = \d e^{-2(t_0+N\ell_0)}$.
    \begin{enumerate}
        \item $\l$ is absolutely continuous to $\mu_\T$ with Radon-Nikodym derivative satisfying $\frac{d\l}{d\mu_\T} \gg_{N,\d,t_0} 1$, and the total mass of $\l$ is $\asymp 1$.

        \item For all
    $t\geq t_0$, we have 
    \begin{align*}
        d_\ast \left((g_t)_\ast\l, \frac{1}{N} \sum_{k=1}^N (g_{t+k\ell_0})_\ast \mu_{\T_\k}\right )<\e.
    \end{align*}
    \end{enumerate}
\end{prop}

The remainder of this subsection is dedicated to the proof of this proposition.
Let $\e>0$ and $N\in \N$ be given parameters.
Let $\d =\d(\e,N)>0$ be a sufficiently small parameter to be specified over the course of the argument.
We apply Proposition~\ref{prop:match on Vcal trivial monodromy} with $\ell=k\ell_0$ for each $1\leq k\leq N$ to get $t_0 = t_0(\e,N,\d)>0$ and a compact set $\Kcal = \Kcal(\e,N) \subset \Vcal$, so that the conclusion holds for all $t\geq t_0$ and all $k\in \set{1,\dots, N}$.

 In particular, for each $k$, we get the following sets and maps
 \begin{align*}
     t_k \stackrel{\mrm{def}}{=} t_0 +k\ell_0,
     \qquad
     G_k \stackrel{\mrm{def}}{=} G_{t_k}^{t_0}\subseteq [0,1], \qquad
    \vp_k \stackrel{\mrm{def}}{=} \vp_{t_k \to t_0}:G_k \r [0,1],
 \end{align*}
 such that each $G_k$ has measure $\geq 1-\e$.
 From Conclusion~\eqref{item:lower triangular}, we also obtain maps $\t_k:G_k \r [-\d,\d]$ and $\s_k: G_k \r [-\d,\d]$ such that
\begin{align}\label{eq:weak-stable matching on Vcal}
    g_{\t_k(s)}\cdot \w(t_0,\vp_k(s))
    =  u^-(\s_k(s))\cdot  \w(t_k,s),    
\end{align}
where we recall that $u^-(\ast) = \left(\begin{smallmatrix}
    1 & 0 \\ \ast & 1
\end{smallmatrix} \right)$.
Let 
\begin{align*}
        \k =\d e^{-2(t_0+N\ell_0)}.
    \end{align*}

 \begin{definition}[The matching maps $\Phi_k$]
    For each $1\leq k\leq N$, let $\Phi_k: \T_\k \dashrightarrow \T(\w) $ denote a partially defined map from $\T_\k$ to $\T(\w)$, defined for each $s\in G_k$, $\b\in\twist^0(\w)$, and $|r|<\d e^{-2N\ell_0}$, by 
 \begin{align*}
    \Phi_k(  \Trem_\b(0,s,r) )
    &= \Trem_\b(0, \vp_k( s), e^{k\ell_0-\t_k(s)} r).    
 \end{align*}    
 We refer to the domain of $\Phi_k$ as \textbf{tremors arising from} $\mathbf{G_k}$.
 \end{definition}
 
In what follows, we extend the definition of $\t_k$ from $G_k$ to a partially defined map on $\T_\k$ by setting 
\begin{align}\label{eq:tau extension}
    \t_k(q)\stackrel{\mrm{def}}{=}\t_k(s), \qquad \text{ for } q=   \mrm{Trem}_\b(0,s,r)\in \T_\k, s\in G_k.
\end{align}

\begin{lem}\label{lem:Jacobian}
    For all $1\leq k\leq N$, the Jacobian $\frac{d(\Phi_k)_\ast\mu_{\T} }{d\mu_\T}$ of the map $\Phi_k$ is $\gg_N 1$ on its domain.
\end{lem}
\begin{proof}
    Indeed, since $\mu_\T$ is equivalent to the pushforward of the product Lebesgue measures under the parametrization $(\b,s,r)\mapsto \Trem_\b(0,s,r)$, the lemma follows from Proposition~\ref{prop:match on Vcal trivial monodromy}\eqref{item:Jacobian} and the definition of $\Phi_k$ in those coordinates.
\end{proof}

Recall that $\Kcal\subset \Vcal$ is the compact set provided by Proposition~\ref{prop:match on Vcal trivial monodromy}.
Let $\rho>0$ be sufficiently small so that $\rho$ is smaller than the injectivity radius of $\Kcal$, and for every $x$ in the unit neighborhood of $\Kcal$, the holonomy period coordinates map $\hol: B(x,\rho) \r H^1$ is injective on the open $\rho$-ball around $x$. 
We shall assume that $\d>0$ is chosen sufficiently small so that for every $p=\Trem_\b(0,s,r) \in  \T_\k$, with $s\in G_k$, and $q=\Phi_k(p)\in \cdot \T$, we have that 
\begin{align}\label{eq:matched points are close}
    \dAGY( g_{t_k}\cdot p, \w(t_k,s)) &<\rho/10,
    \nonumber\\
    \dAGY( g_{t_0}\cdot q, \w(t_0,\vp_k(s))) &<\rho/10,
    \nonumber\\
    \dAGY(\w(t_k,s), \w(t_0,\vp_k(s)) &< \rho/10.
\end{align}
In particular, for all $p$ and $q$ as above, we have
\begin{align}\label{eq:matched points are close 2}
    \dAGY(g_{t_k}p, g_{t_0}\Phi_k(p)) <\rho/2.
\end{align}

Recall the maps $\Psi_q^s$ parametrizing the local strong stable leaf of $q$ defined in \textsection\ref{sec:(un)stable}.
The following lemma is the key to our proof, and is the crucial reason we are able to obtain a stronger conclusion in Theorem~\ref{thm:full support} compared to Theorem~\ref{thm:density of tori}.
It roughly says that the weak-stable relation~\eqref{eq:weak-stable matching on Vcal} between matched points in $\Vcal$ persists after tremoring by small amounts. 
The essential input is Proposition~\ref{prop:match on Vcal trivial monodromy}\eqref{item:equal matrices} which says that the image of twist classes under the cocycle at the matched points agree, allowing us to apply small amounts of cylinder twists without picking up any divergence along the unstable manifold.
This is a very useful lemma, since the property of being connected along the weak-stable manifold survives for all future geodesic flow times.
The proof of the lemma is a simple consequence of definitions along with an application of the following formula for tremor paths in period coordinates, cf.~\eqref{eq:tremor in coordinates},
\begin{align*}
    \hol^{(x)}_{\Trem_\b(t_0,s,r)} = \hol^{(x)}_{\w(t_0,s)} + r \b(t_0,s),
    \qquad 
    \hol^{(y)}_{\Trem_\b(t_0,s,r)} = \hol^{(y)}_{\w(t_0,s)} .
\end{align*}

\begin{lem}\label{lem:stable paths}
    Assume $\d$ is chosen sufficiently small, depending on the compact set $\Kcal$.
    Then, for every $1\leq k\leq N$, there is a partially defined map $v^-: g_{t_k}\cdot \T_\k \dashrightarrow E^s(-)$ on tremors arising from $G_k$ and satisfying the following where defined:
    \begin{enumerate}
        
        \item $\norm{v^-(g_{t_k}q)}_{g_{t_k}q}=O_\Kcal(\d)$, and
        \item $
        \Psi^s_{g_{t_k}q}(v^-(g_{t_k}q))=g_{t_0+\t_k(q)} \cdot \Phi_k(q)$.

    \end{enumerate}
     
\end{lem}

\begin{proof}
Fix some $s\in G_k$, and let $q_1=\Trem_\b(t_k,s,r)\in g_{t_k}\cdot \T_\k$, for some $\b\in \twist^0(\w)$ of unit norm and $ |r|<\d e^{-2t_N}$.
Let
\begin{align*}
    q_2 
    = g_{t_0}\cdot \Phi_k(g_{-t_k}q_1).
\end{align*}
In what follows, in view of~\eqref{eq:matched points are close} and~\eqref{eq:matched points are close 2}, we use the fact that the points $q_1, q_2, \w(t_k,s),$ and $\w(t_0,\vp_k(s))$ all belong to a ball on which period coordinates $q\mapsto \hol_q$ are injective.

By Proposition~\ref{prop:match on Vcal trivial monodromy}\eqref{item:equal matrices}, we have the crucial identity: 
\begin{align*}
    \b(t_0,\vp_k(s))= e^{-k\ell_0} \b(t_k,s).    
\end{align*}
Here, we are identifying the cohomology groups of the two points $\w(t_0,\vp_k(s))$ and $\w(t_k,s)$, which are $O(\d)$-apart, so that we may regard the two cohomology classes on the two sides of the above equation as belonging to the same cohomology group.
It follows that 
\begin{align*}
    \hol^{(x)}_{q_2} 
    = \hol^{(x)}_{\w(t_0,\vp_k(s))} + r e^{k\ell_0-\t_k(s)} \b(t_0,\vp_k(s)),
    \qquad 
    \hol^{(y)}_{q_2} 
    = \hol^{(y)}_{\w(t_0,\vp_k(s))}.
\end{align*}

In view of~\eqref{eq:weak-stable matching on Vcal}, and the equivariance of $\hol$ under the action on of $\SL$, cf.~\eqref{eq:g_t action on hol} and~\eqref{eq:u-minus action on hol}, we obtain the following relations between the periods of $q_1$ and $q_2$:
\begin{align*}
    \hol^{(x)}_{q_2}
    &= e^{-\t_k(s)} \hol^{(x)}_{\w(t_k,s)} + r e^{-\t_k(s)} \b(t_k,s)
    = e^{-\t_k(s)} \hol^{(x)}_{q_1},
    \nonumber\\
    \hol^{(y)}_{q_2}
    &= e^{\t_k(s)} \left( \hol^{(y)}_{\w(t_k,s)} + \s_k(s) \hol^{(x)}_{\w(t_k,s)} \right)
    = e^{\t_k(s)} \left( \hol^{(y)}_{q_1} + \s_k(s) \hol^{(x)}_{\w(t_k,s)} \right) .
\end{align*}
Put together, we get
\begin{align}\label{eq:period relation between z1 and z2}
    \hol^{(x)}_{g_{\t_k(s)}q_2} =\hol^{(x)}_{q_1} ,
    \qquad \qquad 
    \hol^{(y)}_{g_{\t_k(s)}q_2} = \hol^{(y)}_{q_1}+\s_k(s) \hol^{(x)}_{\w(t_k,s)}.
\end{align}

Denote by $v^-(q_1)$ the image by parallel transport of the class $\s_k(s)\hol^{(x)}_{\w(t_k,s)}\in H^1_{\mathbf{i}\R}$ from $\w(t_k,s)$ to $q_1$ along the path $r'\mapsto \Trem_\b(t_k,s,r')$.
Then, since $\b(t_k,s)$ belongs to the balanced space at $\w(t_k,s)$, we get vanishing of the intersection product of $v^-(q_1)$ with $\hol^{(x)}_{q_1}$, i.e.,
$$\int_{M_{q_1}} v^-(q_1) \wedge \hol^{(x)}_{q_1} = 0.$$
In particular, $v^-(q_1)$ is tangent to the strong stable leaf through $q_1$.
Moreover, since $|\s_k(-)|=O(\d)$ and $\norm{\hol^{(x)}_\w}_\w =O_\Kcal(1)$ for all $\w$ in the compact set $\Kcal$, we have
\begin{align}\label{eq:norm of v-(q1)}
    \norm{v^-(q_1)}_{\w(t_k,s)}\ll_\Kcal \d,
\end{align}
which verifies the first assertion of the lemma.

Next, we wish to estimate $\norm{v^-(q_1)}_{q_1}$ using Proposition~\ref{prop:norm of parallel transport}. To that end, we need to check that the hypothesis of its last assertion is satisfied.
Recall that $|r|<\d e^{-2t_N}$, and hence, since  $N_\b(t_k,s) \leq e^{2t_k}N_\b(0,s)\ll e^{2t_N} $ by Lemma~\ref{lem:equivariance of notation}, we get that $|r|\ll \d / N_\b(t_k,s)$.
Thus, by Proposition~\ref{prop:norm of parallel transport} applied to the tremor path $\k$ joining $\w(t_k,s)$ to $q_1$ with $\dot{\k}\equiv \b(t_k,s)$, we obtain\footnote{Proposition~\ref{prop:norm of parallel transport} is stated for the marked stratum $\HHm$, however we note that it suffices to check the ensuing estimate for the lift of the tremor path to $\HHm$ since our tremor path is contained in an injective neighborhood of $q_1$.}
\begin{align}\label{eq:norm of vector field}
    \norm{v^-(q_1)}_{q_1}
    \leq \frac{\norm{v^-(q_1)}_{\w(t_k,s)}}{1-|r| \norm{v^-(q_1)}_{\w(t_k,s)}}
    \ll_\Kcal \d, 
\end{align}
where the second inequality follows by~\eqref{eq:norm of v-(q1)}. 
In particular, we may assume that $\d$ is sufficiently small so that $\norm{v^-(q_1)}_{q_1}<1/2$, and hence $\Psi^s_{q_1}(v^-(q_1))$ is well-defined.

Moreover, by Proposition~\ref{prop:stable exp lipschitz}, we have that
\begin{align*}
    \dAGY( q_1, \Psi^s_{q_1}(v^-(q_1)))
    \leq 2 \norm{v^-(q_1)}_{q_1}.
\end{align*}
In particular, by taking $\d$ small enough, depending only on $\Kcal$, and applying~\eqref{eq:norm of v-(q1)}, we get that $\Psi^s_{q_1}(v^-(q_1))$ belongs to the $\rho$-ball around $q_1$.
On the other hand, by~\eqref{eq:stable exp in coords} and~\eqref{eq:period relation between z1 and z2}, we have that the points $\Psi_{q_1}^s(v^-(q_1))$ and $g_{\t_k(s)} q_2$ have the same image in period coordinates.
By~\eqref{eq:matched points are close 2} and the fact that $|\t_k(s)|=O(\d)$, we also have that $g_{\t_k(s)} q_2$ is in the $\rho$-ball around $q_1$, whenever $\d$ is sufficiently small.
Hence, by injectivity of period coordinates on the $\rho$-ball around $q_1$, we obtain
\begin{align*}
   \Psi_{q_1}^s(v^-(q_1)) = g_{\t_k(s)} q_2,
\end{align*}
which concludes the proof.
\qedhere
\end{proof}

Combining Lemma~\ref{lem:stable paths} and non-uniform hyperbolicity of the geodesic flow, we obtain the following corollary.
\begin{cor}\label{cor:matched pts converge}
    For every $1\leq k\leq N$, and for $\mu_\T$-almost every tremor $q\in \T_\k$ arising from $G_k$, we have $\lim_{t\to\infty} \dAGY(g_{t+k\ell_0} q,g_{t+\t_k(q)}\Phi_k(q)) =0$, where $\t_k(\cdot)$ is as in~\eqref{eq:tau extension}.
\end{cor}

\begin{proof}
    Recall that $t_k=t_0+k\ell_0$.
    Let the notation be as in Lemma~\ref{lem:stable paths}. 
    By this lemma and Corollary~\ref{cor:equivariance of stable exp}, for all $t\geq t_0$, 
    \begin{align*}
        g_{t+\t_k(q)}\Phi_k(q) 
        = g_{t-t_0}\Psi^s_{g_{t_k}q}(v^-(g_{t_k}q))
        = \Psi^s_{g_{t+k\ell_0}q}(Dg_{t-t_0}(g_{t_k}q)\cdot v^-(g_{t_k}q)),
    \end{align*}
    where $Dg_t$ denotes the derivative of $g_t$.
    Proposition~\ref{prop:stable exp lipschitz} then implies the bound
    \begin{align*}
       \dAGY(g_{t+k\ell_0} q,g_{t+\t_k(q)}\Phi_k(q)) 
        &\leq 2 \norm{Dg_{t-t_0}(g_{t_k}q)\cdot v^-(g_{t_k}q)}_{g_{t+k\ell_0}q}
        \nonumber\\
        &\leq 2 \norm{Dg_{t+k\ell_0}(q)}_{q\r g_{t+k\ell_0} q}^s \norm{Dg_{-t_k}(g_{t_k}q)\cdot v^-(g_{t_k}q)}_q
        ,
    \end{align*}
    for all $t\geq t_0$, where $\norm{Dg_t(q)}_{q\r g_tq}^s $ is the operator norm of the restriction of the derivative of $g_t$ to the tangent space $E^s(q)$ to the strong stable leaf through $q$.
    This operator norm tends to $0$ as $t\r\infty$ for $\mu_\T$-almost every $q\in \T(\w)$ by Corollary~\ref{cor:contraction of vertical classes}, which concludes the proof.
  \qedhere
\end{proof}

\subsection{Conclusion of the proof of Proposition~\ref{prop:match tori trivial monodromy}}

We keep the notation of the previous subsection.
Denote by $\mu_{\T_\k}$ the restriction of $\mu_\T$ to $\T_\k$, normalized to be a probability measure.
Let $\mathbb{G}_k\subset\T_\k$ denote the domain of $\Phi_k$, i.e., $\mathbb{G}_k$ consists of all the tremors arising from $G_k\subset [0,1]$.
Then, $\mathbb{G}_k$ has $\mu_{\T_\k}$-measure at least $1-\e$.

Let $\th_\k \mu_\T(\T_\k) \asymp_{N,\d,t_0} 1$ and define $\l$ to be the following Borel measure on $\T(\w)$:
\begin{align*}
    \l =  \frac{\th_\k^{-1}}{N}\sum_{k=1}^N (\Phi_k)_\ast\mu_\T|_{\mathbb{G}_k} .
\end{align*}
Then, Lemma~\ref{lem:Jacobian} then shows that $d\l/d\mu_\T \gg_{N,\d,t_0} 1$.
Moreover, since $\t_k(-) = O(\d)$, Corollary~\ref{cor:matched pts converge} implies that for all $1\leq k\leq N$,
\begin{align*}
    \limsup_{t\r\infty}\, d_\ast ((g_t)_\ast (\Phi_k)_\ast \mu_\T|_{\mathbb{G}_k}, (g_{t+k\ell_0})_\ast\mu_\T|_{\mathbb{G}_k}) \ll \d \cdot \mu_\T(\mathbb{G}_k).
\end{align*}
Hence, $\mu_\T(\mathbb{G}_k)\geq (1-O(\e)) \th_\k$ for each $k$, it follows that for all large enough $t$,
    \begin{align*}
        d_\ast \left((g_t)_\ast\l, \frac{1}{N} \sum_{k=1}^N (g_{t+k\ell_0})_\ast \mu_{\T_\k}\right )\ll \d+\e.
    \end{align*}
Taking $\d<\e$, this concludes the proof of the proposition since $\e$ was arbitrary.

\subsection{A conesequence of Theorem~\ref{thm:uniform Banach density one}}
We record here a consequence of Theorem~\ref{thm:uniform Banach density one} to equidistribution of $g_t$-pushes of the pieces $\T_\k\subset \T(\w)$.

Recall we are fixing a pseudo-Anosov element $\g$ in the Veech group of $(M,\w)$ with primitive period $\ell_0>0$.
For $\eta>0$, we let $\T_\eta \subset \T(\w)$ be the subset of the twist torus defined as in~\eqref{eq:T_kappa} with $\eta$ in place of $\k$.
We also recall that $\mu_{\T_\eta}$ is the normalized restriction of $\mu_\T$ to the piece $\T_\eta$.

\begin{prop}\label{prop:Banach applied to torus slivers}
    For every $\e>0$ and  $f\in C_c(\Mcal)$, there is $N=N(\e,f)>0$ such that for any $\eta>0$, there is $T=T(\eta, \e,f)\geq 1$ so that for all $t\geq T$, we have
    \begin{equation}
    \left|\frac{1}{N}\sum_{k=1}^N \int f \,d (g_{t+k\ell_0})_\ast \mu_{\T_\eta}
 - \int f\;d\mu_\Mcal\right| <\e.
\end{equation}
\end{prop}

Analogously to how Theorem \ref{thm:uniform Banach density one} follows from Corollary \ref{cor:weak close} we have the following, which generalizes $[S,S+T] \cap \mathbb{N}$ to more general arithmetic progressions $[S,S+T]+(t+\ell \mathbb{N})$:
\begin{lem}\label{lem:weak close}

Let $\ell \in \mathbb{R}^+$ and 
$\delta>0$. 
There exist $T_0>0$ and proper $\SL$-orbit closures $\Ncal_1, \dots, \Ncal_n$ in $\Mcal$ such that for any compact set $F\subset \Mcal\setminus \cup_{i=1}^k \Ncal_i$, we can find $S_0\geq 0$, so that for all $T\geq T_0$, $t>0$,  $S\geq S_0$ and $x\in F$, we have
\begin{equation}
\bigg|\bigg\{k \in [S,S+T] \cap \ell \mathbb{N}: d_\ast\big( \int_{0}^1 \delta_{g_k u(s)g_tx}ds,\mu_{\Mcal}\big)<\delta \bigg\}\bigg|>(1-\delta)T.
\end{equation}
\end{lem}

We also need the following elementary lemma in measure theory and its proof is left to the reader.
\begin{lem}\label{lem:meas ex}
For any $f \in C_c(\mathcal{M})$ and $\epsilon>0$ there exists $\delta>0$ so that if $\Omega$ is the set of Borel probability measures of mass at most 1 and $\mathbb{P}$ is a probability measure on $\Omega$ with the property that 
$$\mathbb{P}(\{\sigma \in \Omega: |\int f d\sigma- \int f d\nu)<\delta\})>1-\delta$$ then
$$\bigg| \int \int f d \sigma d\mathbb{P}(\sigma) -\int f d\nu \bigg | <\epsilon.$$
\end{lem}

\begin{lem}\label{lem:to sum}
Let $\epsilon>0$ be given, there exist $N>0$ and proper $\SL$-orbit closures $\Ncal_1, \dots, \Ncal_n$ in $\mathcal{M}$ such that for any compact set $F\subset \mathcal{M}\setminus \cup_{i=1}^k \Ncal_i$, we can find $S_0\geq 0$, so that for all $t\geq S_0$,   and $x\in F$, we have
\begin{equation}
\label{eq:sum density}
\bigg| \sum_{k=1}^N \frac 1 N \int_{0}^1 f(g_{k\ell_0+t}u(s)x)ds -\int f d\mu_{\mathcal{M}}\bigg|<\epsilon.
\end{equation}
\end{lem}
\begin{proof}
Let $\delta>0$ be given by Lemma \ref{lem:meas ex}. We apply Lemma \ref{lem:weak close}, specialized to the case of a single function $f$, to obtain $N$ 
and letting $\mathbb{P}$ be uniform measure on $\{1,\dots,N\}$.  By Lemma \ref{lem:weak close}, we get for every $x\in F$ that
$$\mathbb{P}\left(\set{k \in \{1, \dots ,N\}: \bigg |\int_0^1f(g_{t+k\ell}u(s)x)-\int f d\mu_{\mathcal{M}}\bigg|<\delta } \right) >1-\delta.$$
Thus by Lemma \ref{lem:meas ex} we have 
$$\left|\sum_{k=1}^N\int_{0}^1 f({g_{k\ell_0+t}u(s)x})ds -\int f d\mu_{\mathcal{M}})\right |<\epsilon.$$
\end{proof}
\begin{proof}[Proof of Proposition \ref{prop:Banach applied to torus slivers}]
Let $\delta>0$ be given by Lemma \ref{lem:meas ex} for $f$. We apply Lemma \ref{lem:to sum} with $\delta$ in place of $\epsilon$ and obtain suborbit closures $\Ncal_1, \dots , \Ncal_n$. Observe that each $\Ncal_i$ has $\mu_{\mathcal{M}}$ measure 0. Because the $\Ncal_i$ are all $\SL$-invariant, it follows that $\mu_{\mathbb{T}_\eta}(\Ncal_i)=0$ for all $\eta>0$. 
We now choose a compact set 
$F\subset \mathbb{T}_\eta\setminus \bigcup_{i=1}^n \Ncal_i$ with $\mu_{\mathbb{T}_\eta}(F)>1-\delta.$
By Lemma \ref{lem:to sum} we have that 
$$\mu_{\mathbb{T}_\eta} \left(\set{x \in \mathbb{T}_\eta: \bigg|\sum_{k=1}^N \frac 1 N \int_{0}^1 f({g_{k\ell_0+t}u(s)x})ds -\int f d\mu_{\mathcal{M}}\bigg|<\delta }\right)>1-\delta.$$
By Lemma \ref{lem:meas ex} with $\mathbb{P}=\mu_{\mathbb{T}_\eta}$ we have 
$$\bigg|  \int \sum_{k=1}^N \frac 1 N \int_{0}^1 f({g_{k\ell_0+t}u(s)x})dsd\mu_{\mathbb{T}_\eta}	-\int f d\mu_{\mathcal{M}}\bigg|<\epsilon.$$
Because $\mathbb{T}_\eta$ is foliated by periodic horocycles, $u(s)\mu_{\mathbb{T}_\eta}=\mu_{\mathbb{T}_\eta}$ for all $s$, and so we have 
$$\bigg|  \int  \sum_{k=1}^N \frac 1 N  fd(g_{k\ell_0+t})_*\mu_{\mathbb{T}_\eta}	-\int f d\mu_{\mathcal{M}}\bigg|<\epsilon.$$ 
This completes the proof of the proposition.
\qedhere
\end{proof}

\subsection{Proof of Theorem~\ref{thm:full support}}

Let $U\subset \Mcal$ be a non-empty open set.
Let $\e>0$ to be chosen depending only on $U$.
Let $N$ be the parameter provided by Proposition~\ref{prop:Banach applied to torus slivers} when applied with $\e/2$ and with $f$ being a smooth bump supported in $U$ with Lipschitz constant $\mrm{Lip}(f)=O_U(1)$, and having $\int f\,\mu_\Mcal \geq \mu_\Mcal(U)-\e/2$.

Let $\d$ and $t_0$ be the parameters provided by Proposition~\ref{prop:match tori trivial monodromy}, and let $\k=\d e^{-2(t_0+N\ell_0)}$.
Then, by Proposition~\ref{prop:Banach applied to torus slivers} applied with $\eta=\k$, for all large enough $t$, we have
\begin{align*}
    \frac{1}{N} \sum_{k=1}^N (g_{t+k\ell_0})_\ast \mu_{\T_\k}(U) >\mu_\Mcal(U)-\e.
\end{align*}
Hence, applying the weak-$\ast$ closeness from Proposition~\ref{prop:match tori trivial monodromy} to the bump function $f$, we obtain for all large enough $t$ that
\begin{align*}
    (g_t)_\ast \l(U) > \mu_\Mcal(U) - 2\e - O_{\mrm{Lip}(f)}( \e).
\end{align*}
Taking $\e$ small enough depending on $U$, the above lower bound is $>\mu_\Mcal(U)/2$.
Thus, the estimate on the Radon-Nikodym derivative in Proposition~\ref{prop:match tori trivial monodromy} implies that
\begin{align*}
    (g_t)_\ast \mu_\T(U) \gg_{N,\d,t_0} \mu_\Mcal(U),
\end{align*}
for all large enough $t$. This concludes the proof.

\appendix

\section{Density of Translates of Twist Tori of the Decagon}
\label{sec:decagon}

In this section, we outline the modifications on the proof of Theorem~\ref{thm:density of tori} to show that its conclusion holds for the decagon surface despite it not satisfying the assumption of that result. 
\begin{thm}\label{thm:decagon dense}
Let $(M_5,\omega_5)$ denote the horizontally periodic translation surface obtained from the regular decagon with one horizontal edge by identifying parallel sides by translations. Then, for every $\epsilon>0$ and $\mathcal{K} \subset \mathcal{H}(1,1)$, there exists $t_0>0$, so that for all $t\geq t_0$ we have $g_t\T(\omega_5)$ is $\epsilon$-dense in $\mathcal{K}$.  
\end{thm}

  The proof of Theorem~\ref{thm:decagon dense} follows the same steps of the proof of Theorem~\ref{thm:density of tori} given in Section~\ref{sec:reduce to matching}, with the exception of Subsection~\ref{sec:R_delta}, which was the only place in the argument where $\Mcal$-primitivity was used.
  Our goal is to show how to carry out this part of the argument in absence of this hypothesis. 
   More concretely, we will define a suitable substitute of the sets $\Rcal_\d$ in~\eqref{eq:Rcal_delta0}, consisting of tremors at positive distance from any fixed finite collection of proper orbit closures, and satisfying Lemma~\ref{lem:producing sequence of approximating horocycles}.

We retain the notation of Section~\ref{sec:applying Banach density}, applied with $\w_5$ in place of $\w$.
In particular, throughout this section, we fix a choice of $\b\in \twist^0(\w_5)$ and $x=\Trem(\w_5,r\b)$ in~\eqref{eq:choice of beta} and~\eqref{eq:choice of x}, such that
\begin{align*}
    \overline{\SL\cdot x} = \overline{\SL \cdot \T(\w_5)}.
\end{align*}
Let $\mathcal{D}=\SL\cdot (M_5,\omega_5)$.
Recall that by McMullen's classification of orbit closures in genus two \cite{McMullen-ClassificationGenusTwo}, there is exactly one orbit closure between $\Dcal$ and $\Hcal(1,1)$, namely, the eigenform locus $\mathcal{E}_5$ with discriminant $5$. 
For $q\in\Ecal_5$, we denote by $T_q\Ecal_5\subset H^1_\C$ the tangent space of $\Ecal_5$ at $q$.
Recall the forgetful projection $p:H^1(M_5,\Sigma(\w_5); \R) \r H^1(M_5;\R)$ from relative to absolute cohomology.
Note that $\mathcal{E}_5$ has rank one, i.e., its tangent space splits as follows:
\begin{align}\label{eq:E5=taut+rel}
    T_q \Ecal_5 = \mrm{Taut}_q \oplus \mrm{Ker}(p),
\end{align}
 where $\mrm{Taut}_q$ is the tautological plane $\mrm{Taut}_q$; cf.~\textsection\ref{sec:balanced spaces} for definitions.

We begin by showing that our fixed tremor $\b$ is not trapped in $\Ecal_5$.
\begin{lem}
    \label{lem:decagon twist outside E5}
The restriction of the forgetful projection $p$ to $\twist^0(\w_5)$ is injective.
In particular, for $\b$ as above, $\b\notin T_{\w_5}\Ecal_5$.
\end{lem}
The proof is done by identifying a saddle connection $\g$ crossing a cylinder and connecting a cone point to itself. This gives a non-trivial integral absolute homology class $[\g]$ so that $p(\b)(\g)\neq 0$.

\begin{proof}[Proof of Lemma~\ref{lem:decagon twist outside E5}]

    Let $e$ be the top horizontal edge of the regular decagon, and let $p$ and $q$ be its endpoints. 
    Note that $p$ and $q$ give rise to two different singular points in $M_5$.
    In particular, $e$ is a saddle connection in $(M_5,\w_5)$ on the top boundary of a horizontal cylinder $C$. 
    Moreover, since $M_5$ has only two singularities, the bottom boundary of $C$ must contain another copy of either $p$ or $q$, say $p$.
    Consider the saddle connection $\g$ joining the two copies of $p$ on the top and bottom of $C$, and contained entirely within $C$. Such $\g$ exists by convexity of cylinders.
    Let $\b_C\in \twist(\w_5)$ be the class corresponding to horizontal twists in $C$; cf.~\textsection\ref{sec:twists} for a definition.
    Then, since $\g$ crosses $C$, we have $\b_C(\g)\neq 0$.

    Now, let $0\neq \a\in\twist^0(\w_5)$ be arbitrary.
 Since $M_5$ has exactly two horizontal cylinders, $C$ and $C'$, with disjoint interiors, we have that $\alpha=a\beta_C+a'\beta_{C'}$, for some $a,a'\in \R$. Moreover, since $\a$ has $0$ intersection pairing with $\mrm{Re}(\w_5)$, while the intersection pairing of $\b_C$ and $\b_{C'}$ with $\mrm{Re}(\w_5)$ is the non-zero (signed) area of their respective cylinders, we have that both coefficients $a$ and $a'$ are non-zero.
 Also, since $v\subset C$, we have $\beta_{C'}(\g)=0$, and thus $\alpha(\g)=a\beta_C(\g)\neq0$.
 
 On the other hand, since $\g$ joins a singularity to itself and has non-zero holonomy, it also represents a non-zero class in the absolute homology group $H_1(M_5;\Z)$, which we denote by the same name.
 In particular, $p(\a)(\g)\neq 0$. This proves the first assertion. 

 For the second assertion of the lemma, note that since $\b$ is a balanced class, its image $p(\b)$ belongs to $\mrm{Taut}^0_{\w_5}$, and is non-zero by the first assertion. On the other hand, we have by~\eqref{eq:E5=taut+rel} that $p(T_{\w_5}\Ecal_5)= \mrm{Taut}_{\w_5}$, which has trivial intersection with $\mrm{Taut}^0$, thus proving our claim.
\end{proof}

The following corollary is immediate from the classification results of McMullen and Lemma~\ref{lem:decagon twist outside E5}.
\begin{cor}[{\hspace{.1pt}\cite{McMullen-ClassificationGenusTwo}}]
    We have $\overline{\SL\cdot \T(\w_5)} = \Hcal(1,1)$.
\end{cor}

Next, we show that the image tremors $\b(t,s)$ do not collapse on $\Ecal_5$.
To this end, 
recall that the compact set $\Kcal\subset \Dcal$ and the parameter $\e$ chosen in \textsection\ref{sec:applying Banach density}.
For $t>0$, we let $\l_t$ denote the uniform measure on the following set
\begin{align*}
    F_t \stackrel{\mrm{def}}{=}
    \set{(\w_5(t,s),\b(t,s)/N_\b(t,s)): s\in [0,1], \, \w_5(t,s)\in\Kcal}
    \subset \P \mrm{Taut}_\bullet^0,
\end{align*}
where we view $F_t$ as a subset of the projective bundle over $\Dcal$ with fibers the projective space of the balanced space  $\mrm{Taut}_\bullet^0$, i.e., the  complementary space of the tautological plane defined in \textsection\ref{sec:balanced spaces}.

\begin{lem}\label{lem:avoid linear eigen}
    We have
    \begin{align*}
        \underset{t \to \infty}{\lim}\, \lambda_t(\{(q,\mathbb{P}T_q\mathcal{E}_5): q\in \mathcal{D}\})=0,
    \end{align*}
    where $\P T_q\Ecal_5$ is the projectivization of the tangent space of $\Ecal_5$.
\end{lem}

\begin{proof}
Let $\babs(t,s) = p(\b(t,s))$ denote the image of $\b(t,s)$ in real absolute cohomology $H^1_{\R,\mrm{abs}}$.
Let $\l_t^{\mrm{abs}}$ denote the associated measure on projective space.
Then, $\babs \neq 0$ by Lemma~\ref{lem:decagon twist outside E5}.
Using the $\SL$-invariant decomposition $H^1_{\R,\mrm{abs}}=\mrm{Taut}_q\oplus \mrm{Taut}^0_q$ at every $q\in \Dcal$, we write $\babs(t,s) = \babs^{st}(t,s)+\babs^0(t,s)$ for its components along the tautological and balanced spaces respectively.
Then, since $\b(t,s)$ is a balanced class by construction, we have $\babs^{st}(t,s)=0$, i.e., $\babs(t,s)=\babs^0(t,s)$.
It follows that $\l_t^{\mrm{abs}}$ lives on the projectivization of the balanced bundle, and hence so do all of its weak-$\ast$ limits.

On the other hand, by~\eqref{eq:E5=taut+rel}, the image of $T_{q}\Ecal_5$ in $H^1_{\R,\mrm{abs}}$ is the tautological plane $\mrm{Taut}_q$.
Hence, since the forgetful projection $p$ induces a continuous map on projective spaces, if the lemma fails to hold, we get a weak-$\ast$ limit $\l_\infty$ of the measures $\l_t$ whose image in $\P H^1_{\R,\mrm{abs}}$ lives on the projectivized tautological bundle, which yields a contradiction.
\qedhere
\end{proof}

As a first step towards constructing our good set of tremors, for $\d_0,t>0$,  we define 
$$\mathcal{P}_{\d_0}(t) 
= \set{\Trem_\b(t,s,r):s \in [0,1], \, \w_5(t,s)\in \Kcal, \, \delta_0/2<rN_\b(t,s)<\delta_0 }.$$ 
The following lemma substitutes for Lemma \ref{lem:dist to exceptional orbit closures}.
\begin{lem}
    \label{lem:avoid eigen} 
    Let $\Ncal \subsetneq \mathcal{H}(1,1)$ be a proper orbit closure. 
    Then, for all sufficiently small $\d_0>0$, if $\eta>0$ is sufficiently small,  then the set
     \begin{align}\label{eq:Pcal_delta cap Ncal(eta)}
     \set{s\in [0,1]: \Trem_\b(t,s,r) \in \Pcal_{\d_0}(t) \cap \Ncal(\d_0) \text{ for some } r \text{ with } \d_0/2 < r N_\b(t,s)<\d_0 }
 \end{align}
 has Lebesgue measure $<\e$ for all large enough $t$, where $\Ncal(\eta)$ is the $\eta$-neighborhood of $\Ncal$.
 \end{lem}

\begin{proof}
 
Lemma \ref{lem:dist to exceptional orbit closures} treats the possibility $\Ncal=\mathcal{D}$, as well as any proper orbit closure that does not contain $\Dcal$.
Indeed, this Lemma shows that if $\d_0>0$ is sufficiently small, then in fact $\cup_{t>0}\Pcal_{\d_0}(t)$ lies at positive distance from $\Ncal$.
In particular, taking $\eta$ to be half this positive distance, we get that the exceptional set in~\eqref{eq:Pcal_delta cap Ncal(eta)} is empty in those cases.

It remains to treat the case $\mathcal{D} \subsetneq \mathcal{N}\subsetneq \mathcal{H}(1,1)$.
As noted before, McMullen's classification gives that $\Ncal=\Ecal_5$.
Let $\l_t$ be the measures in Lemma~\ref{lem:avoid linear eigen}, and consider the map
\begin{equation}\label{eq:cont map}
(s,\ell) \mapsto \Trem_\beta(t,s,\ell/N_\b(t,s)).
\end{equation}
By Lemma~\ref{lem:avoid linear eigen}, we can find an open neighborhood $U$ of $\set{(q,\P T_q\Ecal_5):q\in \Dcal}$ so that $\l_t(U)<\e$ for all large enough $t$.
Since the map~\eqref{eq:cont map} is continuous and $\Kcal$ is compact, there exist $\delta_0,\eta>0$ so that if $\w_5(t,s) \in \mathcal{K}$ and 
$$(\w_5(t,s),\b(t,s)/N_\b(t,s)) \notin U,$$
then for all $\delta_0/2<r<\delta_0$, we have $\dAGY(\Trem_\beta(t,s,r/N_\b(t,s)), \Ecal_5)>\eta$. 
This proves the lemma.
\qedhere
\end{proof}

Next, let $\Ncal_1, \dots, \Ncal_k\subsetneq \Hcal(1,1)$ denote the finite collection of exceptional orbit closures produced using Theorem~\ref{thm:uniform Banach density one} as in \textsection\ref{sec:applying Banach density}.
Given $\d_0,\eta>0$, we define
\begin{align*}
     \Rcal^\Dcal_{\d_0,\eta} \stackrel{\mrm{def}}{=} \bigcup_{t>0} \Pcal_{\d_0}(t) \setminus  \bigcup_{i=1}^k\Ncal_i(\eta),
\end{align*}
where $\Ncal_i(\eta)$ denotes the $\eta$-neighborhood of $\Ncal_i$.
We now show that Lemma \ref{lem:producing sequence of approximating horocycles} holds with $\Rcal^\Dcal_{\d_0,\eta}$ in place of $\Rcal_{\d_0}$.
Recall the notation of that lemma.

\begin{lem}
    For all sufficiently small $\d_0>0$ and $\eta>0$,
    the statement of Lemma~\ref{lem:producing sequence of approximating horocycles} holds with $\Rcal_{\d_0 \eta}^\Dcal$ in place of $\Rcal_{\d_0}$ for all large enough $t$.
\end{lem}

\begin{proof}
 Parts~\eqref{item:good indices} and~\eqref{item:landing near torus} of Lemma~\ref{lem:producing sequence of approximating horocycles} are formal consequences of Proposition~\ref{prop:key matching}.
Part \eqref{item:s_0 in Rcal_delta0} also holds with the same argument after removing a small measure set coming from Lemma~\ref{lem:avoid eigen} as we now describe.

    First, we assume that $\d_0$ and $\eta$ are sufficiently small so that Lemma \ref{lem:avoid eigen} holds for $\Ncal= \Ncal_i$, for all $1\leq i\leq k$.
 Let $\mrm{Good}(t)\subseteq [0,1]$ be the complement of the set in \eqref{eq:Pcal_delta cap Ncal(eta)}.
Recall that the compact set $\Kcal\subset \Dcal$ was provided by Proposition~\ref{prop:key matching}. Then, Part \eqref{item:key match in cpt} of that proposition, applied with $\ell=0$, implies that $\w_5(t-T-L_0,s)$ belongs to $\Kcal$, for all but a set of $s\in [0,1]$ of measure at most $\e$.
Together with Lemma~\ref{lem:avoid eigen}, this implies that $\mrm{Good}(t-T-L_0)$ has measure at least $1-2\e$.
Hence, Part \eqref{item:s_0 in Rcal_delta0} of Lemma \ref{lem:producing sequence of approximating horocycles} follows by a very similar argument, but where we pick our interval $I$ so that $I\cap S_0 \cap \mrm{Good}(t-T_0-L_0) \neq \emptyset$.
The latter can be arranged since $S_0 \cap \mrm{Good}(t-T_0-L_0) $ has measure at least $1-3\e$.
 \qedhere
\end{proof}

The rest of the proof of Theorem~\ref{thm:decagon dense} now follows exactly as in \textsection\ref{sec:end of density proof} with the set $\Rcal^\Dcal_{\d_0,\eta}$ defined above in place of $\Rcal_{\d_0}$ defined in~\eqref{eq:Rcal_delta0}.

\section{Limiting Distributions of Output Directions}
\label{sec:appendix}

In this section, we show how Theorem~\ref{thm:A-invariance of flag distribution} can be used in the presence of natural additional hypotheses on the cocycle to establish uniqueness of the limit of the measures in~\eqref{eq:exp horocycles skewprods} as $t\r\infty$. 
The additional hypotheses we consider are either that the image of the representation is bounded (Theorem~\ref{thm:convergence of flag distribution compact}), or is proximal and irreducible (Theorem~\ref{thm:convergence of flag distribution proximal}).
In fact, we identify the limiting distribution in these two cases.

\subsection{Bounded representations}
\label{sec:compact monodromy}
The goal of this section is to outline a strengthening of Theorem~\ref{thm:A-invariance of flag distribution} under the added hypothesis on the image of the lattice $\Gamma$ landing in a compact group $K$.
In fact, we prove the following stronger statement regarding convergence of the distribution of the values of the cocycle as a measure on the compact group containing its image.
Note that, in this case, the cocycle induces a skew product action on $\Vcal\times_\Gamma K \stackrel{\mrm{def}}{=} (\SL\times K)/\Gamma$, where the action on the second factor is by left multiplication.

\begin{thm}\label{thm:convergence of flag distribution compact}
Assume that the image of the representation of $\Gamma$ is bounded, and let $K$ denote the smallest compact group containing its image.
Then, for all $(x,k)\in \Vcal\times_\Gamma K$, and all $f \in C_c(\Vcal\times_\Gamma K)$,
        \begin{align*}
            \lim_{t\to\infty} \int_0^1 f(g_tu(s) \cdot (x,k))\;ds = \int f\;d \mu_\Vcal \otimes m_K,
        \end{align*}
    where $\mu_\Vcal$ is the $\SL$-invariant probability measure on $\Vcal$ and $m_K$ is the Haar probability measure on $K$.
\end{thm}

In addition to Theorem~\ref{thm:A-invariance of flag distribution}, we need the following result which follows from entropy considerations and the fact that the action on the fibers is isometric. The method of proof is well-known and goes back to the proof of Ratner's theorems~\cite{Ratner-measure,Ratner-orbit} due to Margulis and Tomanov~\cite{MargulisTomanov}. It has also been applied in many other works on measure rigidity since.
\begin{prop}\label{prop:SL-invariance}
    Let $\hat{\nu}$ be an $A$-invariant Borel measure which projects to $\mu_\Vcal$. Then, $\hat\nu$ is $\SL$-invariant.
\end{prop}

\begin{proof}[Sketch of Proof of Proposition~\ref{prop:SL-invariance}]
    First note that it suffices to prove the statement under the assumption that $\hat{\nu}$ is $A$-ergodic.
    Indeed, by $A$-ergodicity of $\mu_\Vcal$, the image of almost every $A$-ergodic component of $\hat{\nu}$ under the projection to $\Vcal$ is $\mu_\Vcal$.
    
    Let $U^-\subset\SL$ be the subgroup of lower triangular unipotent matrices and $t>0$.
        Since the action on fibers is isometric, the measure theoretic entropy of $\hat\nu$ with respect to $g_t$ agrees with the metric entropy of $\mu_\Vcal$, which is $2t$.
        The isometric action on the fibers also implies that the stable manifolds of $g_t$ are given by orbits of $U^-$.
        It follows by~\cite[Eq (8.3)]{Brown-ZimmerNotes-arXiv} that the conditional entropy along those orbits is also equal to $2t$, which in turn agrees with the Lyapunov exponent (rate of contraction in this case) along the $U^-$-orbits.
        By work of Ledrappier~\cite{Ledrappier-EntropyInvariance}, it follows that $\hat\nu$ is invariant by $U^-$; cf.~\cite[Theorem 9.5]{Brown-ZimmerNotes-arXiv} for the precise statement and~\cite{EinsiedlerLindenstrauss-ClayNotes} for an exposition of the theory of entropy, conditional measures, and invariance in the context of homogeneous dynamics.
        
        Similarly, since the entropy of $g_t$ is the same as that of $g_t^{-1}$, we conclude that $\hat\nu$ is $U$-invariant.
        Since $U$ and $U^-$ generate $\SL$, this implies that $\hat{\nu}$ is $\SL$-invariant and concludes the proof.
\end{proof}

We are now ready for the proof of Theorem~\ref{thm:convergence of flag distribution compact}.

\begin{proof}[Proof of Theorem~\ref{thm:convergence of flag distribution compact}]
Let $\hat\nu$ be as in the statement of the theorem.
By Theorem~\ref{thm:A-invariance of flag distribution} and Proposition~\ref{prop:SL-invariance}, $\hat\nu$ is $\SL$-invariant\footnote{Theorem~\ref{thm:A-invariance of flag distribution} concerns distributions on  the projective bundle $\P\hV$, however the same proof works in the setting of Theorem~\ref{thm:convergence of flag distribution compact}. Indeed, the key sub-polynomial divergence estimate, Lemma~\ref{lem:subpolynomial divergence}, is immediate in this case since the action on the fiber is isometric.}.
Let $\hat\nu_x$ be a system of conditional measures for $\hat\nu$ on $K$ with respect to the projection to $\Vcal$. 
Since $\hat{\nu}$ is $\SL$-invariant and projects to Haar measure on $\Vcal$, uniqueness of the conditional measures $\hat\nu_x$ implies they are equivariant with respect to the $\SL$-action, i.e., for all $g\in\SL$, $\hat\nu_{gx} = B(g,x)_\ast \hat\nu_x$, where we recall that $B(g,x)$ denotes the map from the fiber over $x$ to the fiber over $gx$ induced from left multiplication by $g\in\SL$.
By considering the set of $g\in\SL$ such that $gx=x$, it follows that $\hat\nu_x$ is $K$-invariant and is independent of $x$.
It follows that $\hat\nu = \mu_\Vcal \otimes m_K$, where $\mu_\Vcal$ and $m_K$ are the Haar measures on $\Vcal$ and $K$ respectively. This concludes the proof.
\end{proof}

\subsection{Proximal and irreducible representations}
\label{sec:proximal}

Recall that the cocycle $B(g_t,x)$ is said to be \textit{proximal} with respect to an $A$-invariant measure $\mu$ on $\Vcal$ if its top Lyapunov exponent with respect to $\mu$ is simple, i.e., the corresponding top Lyapunov space has dimension $1$.
In this section, we use Theorem~\ref{thm:A-invariance of flag distribution} to show that proximality and irreducibility guarantee that expanding horocycle arcs on the projective bundle $\P\hV$ converge to a unique limiting measure, with atomic disintegration along fibers of the projection $\P\hV\r \Vcal$.

\begin{thm}\label{thm:convergence of flag distribution proximal}
Assume that the representation of $\Gamma$ is irreducible and that the cocycle $ B(g_t,x)$ is proximal with respect to the $\SL$-invariant measure $\mu_\Vcal$ on $\Vcal$.
Then, there exists a Borel probability measure $\hat{\nu}$ on $\P\hV$ which projects to $\mu_\Vcal$, so that for all $(x,v)\in \P\hV$, and all $f \in C_c(\P\hV)$, 
        \begin{align}\label{eq:proximal congergence of flags}
            \lim_{t\to\infty} \int_0^1 f(g_tu(s) \cdot (x,v))\;ds = \int f\;d\hat{\nu}.
        \end{align}
\end{thm}

\subsubsection{Irreducibility and non-concentration on $P^-$-invariant sub-bundles}

Recall that $P=AU$ (resp. $P^-=AU^-$) is the subgroup of upper (resp. lower) triangular matrices in $\SL$.
The following lemma is the key consequence of irreducibility we use in our proof. It is used to show that the Oseledets' distributions of slower-than-maximal growth receive $0$ mass.
    \begin{lem}\label{lem:nonconc on Pminus bundles}
        Let $\hat{\nu}$ be a $P$-invariant probability measure on $\P\hV$.
        Then, $\hat{\nu}(\Qcal)=0$, for every 
         proper, measurable, $P^-$-invariant sub-bundle $\Qcal\subset\P\hV$.
    \end{lem}
    \begin{proof}
        
        By the ergodic decomposition, it suffices to prove the lemma under the additional assumption that $\hat{\nu}$ is $P$-ergodic.
        Note that $\hat{\nu}$ projects to the $\SL$-invariant measure $\mu_\Vcal$ by $P$-invariance.
        By~\cite[Prop.~11.8]{EinsiedlerWard}, since $\hat{\nu}$ is $P$-ergodic, it is also $A$-ergodic.

        Suppose towards a contradiction that $\hat{\nu}$ does not satisfy the conclusion of the lemma and let $\Qcal$ be a proper $P^-$-invariant sub-bundle with minimal dimensional fibers having $\hat{\nu}(\Qcal)>0$.
        Since $\hat{\nu}$ is $A$-ergodic and $\Qcal$ is $A$-invariant,
        we have $\hat{\nu}(\Qcal)=1$.
        Since $\hat{\nu}$ is also $U$-invariant, it follows that for every $u\in U$, we have that $\hat{\nu}(\Qcal\cap u\Qcal)=1$.
        Hence, since $\Qcal$ is $P^-$ invariant, and $\SL$ is generated by $P^-$ and $U$, we get that
        $\hat{\nu}(\Qcal\cap g\Qcal)=1$ for all $g\in\SL$.

        Let $\set{\hat{\nu}_x}_x$ denote a disintegration of $\hat{\nu}$ with respect to the projection $\pi:\P\hV\r\Vcal$.
        In particular, each $\hat{\nu}_x$ is supported on the fiber $\pi^{-1}(x)$.
        It follows that for every $g\in\SL$, we have that $\hat{\nu}_x(\Qcal\cap g\Qcal)
        = \hat{\nu}_x(\Qcal_x \cap B(g,g^{-1}x)\Qcal_{g^{-1}x}) =1$ for a $\mu_\Vcal$-full measure set of $x\in \Vcal$ (depending on $g$).
        Here, $\Qcal_x$ denotes the fiber of $\Qcal$ over $x$, and $\mu_\Vcal$ is the $\SL$-invariant probability measure on $\Vcal$.
        
        By minimality of $\Qcal$, it follows that for all $g\in\SL$,
        \begin{align}\label{eq:inv subspace}
            \Qcal_x = B(g,g^{-1}x)\Qcal_{g^{-1}x},
            \qquad \text{for } \mu_\Vcal{-a.e. } x.
        \end{align}
        To see that this gives a contradiction to our irreducibility hypothesis, let $k$ is the almost sure constant value of the dimension of the fiber vector space corresponding to $\Qcal_x$.
        Such $k$ exists by ergodicity of $\mu_\Vcal$ and invariance of $\Qcal$.
        Let $\mrm{Gr}_k$ be the Grassmannian of $k$-dimensional subspaces of $\R^{d+1}$.
        Then, fixing a measurable trivialization of $\hV \cong \Vcal\times \R^{d+1}$, we may regard $x\mapsto \Qcal_x$ as a $B(-)$-invariant measurable map $\Vcal \r \mrm{Gr}_k$, where invariance is in the sense of~\eqref{eq:inv subspace}.
        Moreover, we can regard the cocycle $B(-)$ as taking values in the image of $\G$ under the representation, which we continue to denote by $\G$ for simplicity.

        Hence, we may apply Zimmer's cocycle reduction lemma~\cite[Lemma 5.2.11]{Zimmer1984} to the cocycle $B:\Vcal\times \SL\r \G$ and the induced $\G$-action\footnote{This algebraic $\G$-action satisfies the smoothness hypothesis of the cited lemma by a result of Borel-Serre~\cite[Theorem 3.1.3]{Zimmer1984}. Here, smoothness means that $\mrm{Gr}_k/\G$ is countably separated; cf.~\cite[Def.~2.1.9]{Zimmer1984}. The proof of~\cite[Lemma 5.2.11]{Zimmer1984} follows from an application of ergodicity of $\mu_\Vcal$ to the invariant measurable map $x\mapsto \Qcal_x$, viewed as a map to $\mrm{Gr}_k/\G$.} on $\mrm{Gr}_k$ to get a measurable change-of-basis map $\vp:\Vcal \r \G$, and a proper subspace $W\in \mrm{Gr}_k$ such that for all $g\in\SL$
        \begin{align}\label{eq:cocycle reduction}
            \vp(gx)^{-1}B(g,x)\vp(x) \in \mrm{Stab}_\G(W),
        \end{align}
        for $\mu_\Vcal$-a.e. $x\in\Vcal$.

        By Fubini's theorem, we get that for $\mu_\Vcal$-a.e. $x\in\Vcal$, there is a full measure set $G(x)\subseteq \SL$ so that~\eqref{eq:cocycle reduction} holds for all $g\in G(x)$.
        Moreover, since $\G$ is countable, there is a positive measure set $E\subseteq\Vcal$ on which $\vp$ is constant. 
        Let $F\subset G$ be a fundamental domain for $\G$ and let $\tilde{E}\subseteq F$ be a lift of $E$.
        Then, given $x\in E$ with lift $\tilde{x}\in \tilde{E}$, since $G(x)$ has full measure in $\SL$, it follows that $G(x)\cdot \tilde{x}$ intersects the positive measure set $\tilde{E}\g$ for all $\g\in\G$.
        It follows that $\set{B(g,x):g\in G(x), gx\in E}=\G$.
        Thus, by~\eqref{eq:cocycle reduction}, this means that $\G$ fixes the proper subspace $W$, which contradicts our irreducibility assumption.
        \qedhere

    \end{proof}

\subsubsection{Proof of Theorem~\ref{thm:convergence of flag distribution proximal}}

Let $\hat\nu$ we a weak-$\ast$ limit of the measures on the left hand-side of~\eqref{eq:proximal congergence of flags} as $t\to\infty$.
Then, as before $\hat\nu$ is automatically $U$-invariant. Moreover, by Theorem~\ref{thm:A-invariance of flag distribution}, we have that $\hat{\nu}$ is $A$-invariant.
Theorem~\ref{thm:convergence of flag distribution proximal} is an immediate consequence of the following measure classification statement.

\begin{prop}
    The $P$-action on $\P\hV$ is uniquely ergodic.
\end{prop}

\begin{proof}
    
Let $\hat{\nu}$ be a $P$-invariant measure on $\P\hV$.
 Then, $\hat \nu$ projects to Haar measure on $\Vcal$ by $P$-invariance.
  Let $\hat{\nu}_x$ denote the conditional measures of $\hat\nu$ along the fibers of the natural projection $\P\hV\r \Vcal$.
  The proof will follow upon uniquely characterizing the measures $\hat{\nu}_x$.
 By $A$-invariance, we have for $\hat\nu$-almost every $x$ that
    \begin{align}\label{eq:A-invariance of conditionals}
        \hat{\nu}_x = B(g_t,g_{-t}x)_\ast \hat{\nu}_{g_{-t}x} ,
    \end{align}
    where $B(-)$ is the cocycle given by the action on fibers.

 Let $\l_1=\lim_{t\r\infty}(1/t)\log\norm{B(g_{t},x)}_{\mrm{op}}$ be the top Lyapunov exponent for the cocycle over $g_{t}$.
 Define the sub-bundle $\hV^{<\l_1}$ with slower growth, i.e., fibers of $\hV^{<\l_1}$ are given by
 \begin{align*}
     V_x^{<\l_1}=\set{v\in V_x: \limsup_{t\r\infty} \log\norm{B(g_{t},x)v}^{1/t}_{g_tx} <\l_1 }.
 \end{align*}
 Then, arguing as in the proof of Lemma~\ref{lem:subpolynomial divergence}, we have that $\widehat{\Vcal}^{<\l_1}$ is a proper measurable $P^-$-invariant sub-bundle.
 Hence, by Lemma~\ref{lem:nonconc on Pminus bundles}, $\hat{\nu}(\widehat{\Vcal}^{<\l_1})=0$.

The rest of the argument is now similar to the proof of Lemma~\ref{lem:subpolynomial divergence}.
Let $\l_2\leq \l_1$ be the second Lyapunov exponent of the cocycle $B(g_t,-)$.
In particular, we have that
\begin{align*}
    \l_1+\l_2 = \lim_{t\to\infty} \log\norm{\wedge^2 B(g_t,x)}^{1/t}_{\mrm{op}}.
\end{align*}
with value independent of $x$. 
By assumption we have that $\l_2<\l_1$.
Let $0<\e \leq (\l_1-\l_2)/2$.
Then, we can find $t_\e>0$ so that the sets $F_1$ and $F_2$ defined by
\begin{align}\label{eq:proximal F1 and F2}
    F_1 &= \set{y\in\Vcal: \norm{\wedge^2 B(g_t,y)}_{\mrm{op} }\leq e^{(\l_1+\l_2+\e/2)t} \text{ for all } t>t_\e},
    \nonumber \\
    F_2 &= \set{(y,v)\in \P\hV: \norm{B(g_t,y)v}_{g_ty} \geq e^{(\l_1-\e/2) t}\norm{v}_y \text{ for all } t>t_\e },
\end{align}
each has measure $\geq 1-\e/2$ with respect to $\mu_\Vcal$ and $\hat{\nu}$ respectively.
Let $F=\pi^{-1}(F_1)\cap F_2$, where $\pi:\P\hV\r\Vcal$ denotes the natural projection.
In particular, we have that $\hat{\nu}_x(F)>1-\sqrt{\e}$ for a $E\subseteq \Vcal$ of measure $\geq 1-\sqrt{\e}$.
Moreover, given $(y,v),(y,w)\in F$ and $t>t_\e$, we have
\begin{align}\label{eq:projective contraction}
    \dist(B(g_t,y)v, B(g_t,y)w) \leq 
    \frac{\norm{\wedge^2 B(g_t,y)}_{\mrm{op}} \norm{v \wedge w}_y}{\norm{B(g_t,y)v}_{g_ty} \norm{B(g_t,y)w}_{g_ty}}
    \leq e^{-(\l_1-\l_2) t/2} \dist(v,w),
\end{align}
where $\dist$ is the metric on the fibers defined in~\eqref{eq:proj dist}.

Now, note that by Poincar\'e recurrence, we have that $\mu_\Vcal$-a.e. $x$ admits a sequence $t_n\r\infty$ such that $g_{-t_n}x\in E$.
Hence, in view of the identity~\eqref{eq:A-invariance of conditionals}, for $\mu_\Vcal$-a.e. $x$, $\hat{\nu}_x$ has an atom of mass $\geq 1-\sqrt{\e}$.
Taking $\e$ to $0$, it follows that the conditional measure $\hat{\nu}_x$ is a Dirac mass almost surely.

Finally, to show that this property implies uniqueness of $\hat{\nu}$, let $\hat{\nu}^i, i=1,2$ be two $P$-invariant probability measures on $\P\hV$.
Let $\k_i(x)\in V_x$ denote the single point supporting $\hat{\nu}^i_x,i=1,2$.
Let $\e=\min\set{(\l_1-\l_2)/2, 1/2}$ and let $F$ be the set defined below~\eqref{eq:proximal F1 and F2} for this $\e$.
Then, arguing as above, there is $E\subseteq \Vcal$ with $\mu_\Vcal(E)\geq 1-\sqrt{\e}$ so that for all $y\in E$, we have $(y,\k_i(y))\in F$ for $i=1,2$.
Hence, for $\mu_\Vcal$-almost every $x$,  we can apply~\eqref{eq:projective contraction} with $y=g_{-n}x, v=\k_1(g_{-n}x)$, and $w=\k_2(g_{-n}x)$ along a sequence of times $n\r\infty$ such that $g_{-n}x\in E$ to conclude that $\k_1(x)=\k_2(x)$, in view of the equivariance relation~\eqref{eq:A-invariance of conditionals}.
\qedhere
\end{proof}

\bibliography{bibliography}{} 
\bibliographystyle{amsalpha}

\end{document}